\documentclass[12pt]{amsart}
\usepackage{amssymb,amsmath,amsthm,latexsym,mathtools,dsfont}
\usepackage{mathrsfs}
\usepackage{mdwlist}
\usepackage{graphicx}
\usepackage{enumerate}
\usepackage[shortlabels]{enumitem}
\usepackage{a4wide}

\usepackage{ifthen}
\usepackage{xargs}

\makeatletter
\newcommand\@dotsep{4.5}
\def\@tocline#1#2#3#4#5#6#7{\relax
  \ifnum #1>\c@tocdepth % then omit
  \else
    \par \addpenalty\@secpenalty\addvspace{#2}%
    \begingroup \hyphenpenalty\@M
    \@ifempty{#4}{%
      \@tempdima\csname r@tocindent\number#1\endcsname\relax
    }{%
      \@tempdima#4\relax
    }%
    \parindent\z@ \leftskip#3\relax
    \advance\leftskip\@tempdima\relax
    \rightskip\@pnumwidth plus1em \parfillskip-\@pnumwidth
    #5\leavevmode\hskip-\@tempdima #6\relax
    \leaders\hbox{$\m@th
      \mkern \@dotsep mu\hbox{.}\mkern \@dotsep mu$}\hfill
    \hbox to\@pnumwidth{\@tocpagenum{#7}}\par
    \nobreak
    \endgroup
  \fi}
\makeatother 
\let\oldtocsection=\tocsection
\let\oldtocsubsection=\tocsubsection

\renewcommand{\tocsection}[2]{\hspace{0em}\oldtocsection{#1}{#2}}
\renewcommand{\tocsubsection}[2]{\hspace{22pt}\oldtocsubsection{#1}{#2}}
\newcommandx{\yaHelper}[2][1=\empty]{%
\ifthenelse{\equal{#1}{\empty}}%
  { \ensuremath{ \scriptstyle{ #2 } } } % no offset
  { \raisebox{ #1 }[0pt][0pt]{ \ensuremath{ \scriptstyle{ #2 } } } }  % with offset
}   
\newcommandx{\yrightarrow}[4][1=\empty, 2=\empty, 4=\empty, usedefault=@]{%
  \ifthenelse{\equal{#2}{\empty}}
  { \xrightarrow{ \protect{ \yaHelper[ #4 ]{ #3 } } } } % there's no text below
  { \xrightarrow[ \protect{ \yaHelper[ #2 ]{ #1 } } ]{ \protect{ \yaHelper[ #4 ]{ #3 } } } } % there's text below
}
\input xy
\xyoption{all}
\usepackage{physics}%nicely handles subscripts at absolute value and modulus signs (when using the abs and norm commands)

\usepackage{xcolor}
\definecolor{darkgreen}{RGB}{0, 153, 51}
\definecolor{violet}{RGB}{112, 73, 170}
\definecolor{darkred}{RGB}{153, 0, 0}
\definecolor{darkdarkblue}{RGB}{0, 0, 102}
\definecolor{darkblue}{RGB}{153, 204, 255}
\definecolor{bluee}{RGB}{204, 230, 255}
\definecolor{bluee2}{RGB}{128, 191, 255}
\definecolor{shadow}{RGB}{82, 122, 122}

\definecolor{my_green1}{RGB}{0, 153, 0}
\definecolor{my_orange}{RGB}{255, 204, 0}
\definecolor{my_yellow}{RGB}{255, 255, 102}

\usepackage{tikz}
\usepackage{pgfplots, tikz-3dplot}
\usetikzlibrary{shapes.geometric, arrows.meta, decorations.markings}
\pgfplotsset{compat=1.16}
\usetikzlibrary{cd}
\usetikzlibrary{calc}
\usetikzlibrary{arrows, decorations.pathmorphing}
\usetikzlibrary{hobby}
\usetikzlibrary{patterns, intersections}

\makeatletter
\tikzoption{canvas is plane}[]{\@setOxy#1}
\def\@setOxy O(#1,#2,#3)x(#4,#5,#6)y(#7,#8,#9)%
  {\def\tikz@plane@origin{\pgfpointxyz{#1}{#2}{#3}}%
   \def\tikz@plane@x{\pgfpointxyz{#4}{#5}{#6}}%
   \def\tikz@plane@y{\pgfpointxyz{#7}{#8}{#9}}%
   \tikz@canvas@is@plane
  }

\makeatletter
\newcommand{\gettikzxy}[3]{%
  \tikz@scan@one@point\pgfutil@firstofone#1\relax
  \edef#2{\the\pgf@x}%
  \edef#3{\the\pgf@y}%
}
\makeatother

\newcommand{\R}{\mathbb{R}}

\newcommand{\N}{\mathbb{N}}
\newcommand{\Z}{\mathbb{Z}}
\newcommand{\C}{\mathbb{C}}

\newcommand{\Aa}{\mathcal{A}}
\newcommand{\BB}{\mathcal{B}}

\newcommand{\GG}{\mathcal{G}}
\newcommand{\Hh}{\mathcal{H}}
\newcommand{\KK}{\mathcal{K}}

\newcommand{\QQ}{\mathcal{Q}}

\newcommand{\UU}{\mathscr{U}}

\newcommand{\ii}{\mathrm{i}}

\newcommand{\id}{\mathrm{id}}
\newcommand{\Find}{\mathrm{index}}
\newcommand{\p}{\varphi}
\newcommand{\e}{\varepsilon}
\newcommand{\w}{\widetilde}
\newcommand{\oo}{\overline}
\newcommand{\MM}{\mathbb{M}}
\newcommand{\ttb}{\mathbf{t}}
\newcommand{\ac}{\,{\scriptstyle\ll}\,}
\newcommand{\SOT}{\scalebox{0.82}{SOT}}

\newcommand{\sot}{\mathop{\scalebox{0.82}{\mbox{{\rm SOT-}}}\hspace*{-1pt}\lim}}
\newcommand{\wot}{\mathop{\scalebox{0.82}{\mbox{{\rm WOT-}}}\hspace*{-1pt}\lim}}

\newcommand{\n}[1]{\|#1\|}
\newcommand{\nn}[1]{{\vert\kern-0.25ex\vert\kern-0.25ex\vert #1 
    \vert\kern-0.25ex\vert\kern-0.25ex\vert}}
\newcommand{\lnn}[1]{{\left\vert\kern-0.25ex\left\vert\kern-0.25ex\left\vert #1 
    \right\vert\kern-0.25ex\right\vert\kern-0.25ex\right\vert}}
\newcommand{\dist}{\mathrm{dist}}

\renewcommand{\leq}{\leqslant}
\renewcommand{\geq}{\geqslant}
\newcommand{\wh}{\widehat}
\newcommand{\cs}{{\rm C}$^\ast$}

\newcommand{\Ext}{\mathrm{Ext}}
\newcommand{\Hom}{\mathrm{Hom}}

\newcommand{\usim}{\,\raise.17ex\hbox{$\scriptstyle\mathtt{\sim}$}}
\renewcommand{\ip}[2]{\langle #1,#2\rangle}

\newtheorem{theorem}{Theorem}[section]
\newtheorem{lemma}[theorem]{Lemma}
\newtheorem{proposition}[theorem]{Proposition}
\newtheorem{corollary}[theorem]{Corollary}

\theoremstyle{definition}
\newtheorem{definition}[theorem]{Definition}

\theoremstyle{remark}
\newtheorem{remark}[theorem]{Remark}
\newtheorem{example}[theorem]{Example}

\numberwithin{equation}{section}

\title{Compact perturbations of operator semigroups}
\author{Tomasz Kochanek}
\address{Institute of Mathematics, University of Warsaw, Banacha~2, 02-097 Warsaw, Poland}
\email{tkoch@mimuw.edu.pl}
\thanks{This work has been supported by the National Science Centre grant no. 2020/37/B/ST1/01052.}

\begin{document}
\begin{abstract}
We study lifting problems for operator semigroups in the Calkin algebra $\QQ(\Hh)$, our approach being mainly based on the Brown--Douglas--Fillmore theory. With any normal $C_0$-semigroup $(q(t))_{t\geq 0}$ in $\QQ(\Hh)$ we associate an~extension $\Gamma\in\Ext(\Delta)$, where $\Delta$ is the inverse limit of certain compact metric spaces defined purely in terms of the spectrum $\sigma(A)$ of the generator of $(q(t))_{t\geq 0}$. By using Milnor's exact sequence, we show that if each $q(t)$ has a~normal lift, then the question whether $\Gamma$ is trivial reduces to the question whether the corresponding first derived functor vanishes. With the aid of the CRISP property and Kasparov's Technical Theorem, we provide geometric conditions on $\sigma(A)$ which guarantee splitting of $\Gamma$. If $\Delta$ is a~perfect compact metric space, we obtain in this way a~$C_0$-semigroup $(Q(t))_{t\geq 0}$ which lifts $(q(t))_{t\geq 0}$ on dyadic rationals.
\end{abstract}
\maketitle
\tableofcontents

%%%%%%%%%%%%%%%%%%%%%%%%%%%%%%%%%%%%%%%%%%%%%%%%%%%%%%%%%%%%%%%%%%%%%%%%%%%%%%%%%%%%%%
%%%%%%%%%%%%%%%%%%%%%%%%%%%%%%%%%%%%%%%%%%%%%%%%%%%%%%%%%%%%%%%%%%%%%%%%%%%%%%%%%%%%%%
\section{Introduction}
\noindent
Throughout the paper, $\Hh$ stands for an infinite-dimensional separable Hilbert space, $\BB(\Hh)$ and $\KK(\Hh)$ denote the algebras of all bounded linear operators on $\Hh$ and all compact linear operators on $\Hh$, respectively. The Calkin algebra is defined by $\mathcal{Q}(\Hh)=\BB(\Hh)/\KK(\Hh)$ and $\pi\colon \BB(\Hh)\to\QQ(\Hh)$ stands for the canonical quotient map. Typically, we write $\mathbb{H}$ for a~Hilbert space of density continuum, on which we sometimes represent the Calkin algebra. We will also use standard notions and facts from the theory of $C_0$-semigroups which can be found, e.g., in \cite{EN}.

There is a big area of lifting problems in the Calkin algebra, or more general corona algebras (see the recent book \cite{farah}), the basic question being whether an~abelian \cs-subalgebra $A$ of $\QQ(\Hh)$ admits an~abelian lift, that is, a~commutative \cs-subalgebra $B\subset\BB(\Hh)$ with $\pi[B]=A$. This problem goes back to the famous Weyl--von~Neumann--Berg--Sikonia theorem (\cite{berg}, \cite{sikonia}), and it is known that separable commutative \cs-subalgebras of $\QQ(\Hh)$ which admit a~lift are characterized as subalgebras of real rank zero \cs-subalgebras of $\QQ(\Hh)$, whereas the nonseparable case is more problematic (see \cite{JP} and \cite{SS}).

In this paper, we are interested in various lifting problems for operator semigroups. One way of formulating such a~problem is to assume that $(Q(t))_{t\geq 0}\subset\BB(\Hh)$ is a~collection of normal operators which satisfy the semigroup condition modulo the compact operators, that is,
\begin{equation}\label{mod_K}
    Q(s+t)-Q(s)Q(t)\in\KK(\Hh)\quad\mbox{for all }s,t\geq 0. 
\end{equation}
By defining $q(t)=\pi Q(t)$, we obtain an~operator semigroup $(q(t))_{t\geq 0}$ in $\QQ(\Hh)$, and a~natural regularity condition is to assume that this is a~$C_0$-semigroup. Obviously, speaking about $C_0$-semigroups in the Calkin algebra makes sense only in reference to some fixed representation of $\QQ(\Hh)$ on a~Hilbert space $\mathbb{H}$. Section~4 is devoted to the study of classical Calkin's representations and conditions guaranteeing that the semigroup in question is in fact \SOT-continuous. 

Alternatively, we can forget about the operators $Q(t)$ and just assume that $(q(t))_{t\geq 0}$ is a~$C_0$-semigroup of normal elements in $\QQ(\Hh)$. In this case, however, because of usual Fredholm index issues, $q(t)$ may not lift to a~normal operator in $\BB(\Hh)$. For the same reason the~\cs-subalgebra $\mathrm{C}^\ast(\pi(s))\cong C(\mathbb{T})$ of $\QQ(\Hh)$, where $s$ is the unilateral shift, does not have a~lift to a~commutative \cs-subalgebra of $\BB(\Hh)$. Similarly, if we pick any essentially normal, Fredholm operator $T\in\BB(\Hh)$ with nonzero index ($T=s$ is a~good choice), then by a~result of Eisner \cite{eisner}, there is a~$C_0$-semigroup $(q(t))_{t\geq 0}\subset\BB(\mathbb{H})$ with $q(1)=\pi T$. On the other hand, no compact perturbation of $T$ is embeddable into a~$C_0$-semigroup, as for any $K\in\KK(\Hh)$ we have $\mathrm{ind}(T+K)=\mathrm{ind}(T)\neq 0$ and nonbijective Fredholm operators are not embeddable into $C_0$-semigroups (see \cite[Cor.~3.2]{eisner}).

Our approach is generally based on the theory of extensions of compact metric spaces developed by Brown, Douglas and Fillmore in their famous work \cite{BDF}, where using some deep connections between homotopy theory, Fredholm index theory, and theory of extensions, they showed that an~operator $T\in\BB(\Hh)$ is of the form `normal plus compact' if and only if $\mathrm{ind}(\lambda I-T)=0$ for every $\lambda\in\C\setminus\sigma_{\mathrm{ess}}(T)$. One of our basic results (Proposition~\ref{P_spectrum}) says that with every normal $C_0$-semigroup in $\QQ(\Hh)$ one can associate an~extension $\Gamma\in\Ext(\Delta)$ of a~certain inverse limit $\Delta=\varprojlim\Omega_n$ of compact metric spaces $\Omega_n$. We then also have Milnor's exact sequence of the form

\begin{equation}\label{intr_ex_seq}
\begin{tikzcd}
0 \arrow[r] & \varprojlim{}^{(1)}\Ext_2(\Omega_n) \arrow[r] & \Ext(\Delta) \arrow[r, "P"] & \varprojlim\Ext(\Omega_n) \arrow[r] & 0.
\end{tikzcd}
\end{equation}
In Proposition~\ref{P_kernel} we show that the classical BDF condition quoted above, assumed for each $q(2^{-n})$, $n=0,1,2,\ldots$, guarantee that the resulting extension $\Gamma$ lies in the kernel of $P$. Therefore, in order to ensure that $\Gamma$ splits, it suffices to show that the first derived functor in \eqref{intr_ex_seq} vanishes, and this is the topic of Section~5 (see Theorems~\ref{empty_dir_T} and \ref{main_K_T}). We denote by $\Theta$ the zero element of the group $\Ext(\Delta)$.

\vspace*{2mm}\noindent
{\bf Summary of the main results. }Let $(Q(t))_{t\geq 0}$ be a~collection of normal operators in $\BB(\Hh)$ satisfying \eqref{mod_K}. Assume that $(q(t))_{t\geq 0}\subset\QQ(\Hh)$, defined by $q(t)=\pi Q(t)$ for $t\geq 0$, is a~$C_0$-semigroup with respect to some faithful $^\ast$-representation $\gamma\colon \QQ(\Hh)\to\BB(\mathbb{H})$. Let also $A$ be its infinitesimal generator, densely defined on $\mathbb{H}$. Then:

\vspace*{0mm}
\begin{enumerate}[label={\rm (\roman*)}, leftmargin=26pt]
\setlength{\itemsep}{3pt}

\item The spectrum of the \cs-algebra $\mathrm{C}^\ast(q(2^{-n}),1_{\QQ(\Hh)})$ is homeomorphic to the inverse limit $\Delta=\varprojlim\{\Omega_n,p_n\}$, where $p_n(z)=z^2$ and
$$
\Omega_n=\oo{\exp(2^{-n}\sigma(A))}\qquad (n=0,1,2,\ldots).
$$

\item There is an~extension $\Gamma\in\Ext(\Delta)$ such that $\Gamma=\Theta$ implies that there exists a~semigroup $(T(t))_{t\in\mathbb{D}}\subset\BB(\Hh)$, defined on positive dyadic rationals, such that $\pi T(t)=q(t)$ for every $t\in\mathbb{D}$.

\item We have the Milnor exact sequence of the form \eqref{intr_ex_seq}, and $\Gamma\in\mathrm{ker}\,P$.

\item Assuming that for each $n\in\N$, the set $\mathsf{A}(\Omega_n)=\{z\in\Omega_n\colon -z\in \Omega_n\}$ satisfies
$$
\oo{\Omega_n\setminus\mathsf{A}(\Omega_n)}\cap\mathsf{A}(\Omega_n)=\varnothing,
$$
and that $\Omega_n$ satisfies either an~`empty direction' condition, or a~`cross retract' condition (for details, see Theorems~\ref{empty_dir_T} and \ref{main_K_T}), we have $\Gamma=\Theta$.

\item If $\Delta$ is a~perfect compact metric space, and $\gamma$ is one of Calkin's representations of $\QQ(\Hh)$, then the obtained lifting $(T(t))_{t\in\mathbb{D}}$ is \SOT-continuous and it extends to a~$C_0$-semigroup $(T(t))_{t\geq 0}\subset\BB(\Hh)$.
\end{enumerate}

The above list of assertions is a~combination of Propositions~\ref{P_spectrum}, \ref{P_kernel}, Lemmas~\ref{lifting_L}, \ref{delta_cont_L}, and Theorems~\ref{no_isolated_T}, \ref{empty_dir_T} and \ref{main_K_T}. There are plenty of examples of $\sigma(A)$ for which one of the conjunctions mentioned in assertion (iv) is satisfied. Some of them are listed in the corollary below. Interestingly, there are also situations where Milnor's exact sequence on its own allows us to conclude that $\Gamma=\Theta$, even though neither the `empty direction' nor `cross retract' condition is satisfied. A~suitable example is given by the $2$-adic solenoid $\Sigma_2=\varprojlim\{\mathbb{T},z^2\}$; see Example~\ref{ex_imaginary}.

\vspace*{2mm}\noindent
{\bf Corollaries from the main results. }Assume $(Q(t))_{t\geq 0}$, $(q(t))_{t\geq 0}$, $A$ are as above, and let $\Omega_n$ ($n=0,1,2,\ldots$) be defined as in (i). In each of the following cases there exists a~semigroup $(T(t))_{t\in\mathbb{D}}\subset\BB(\Hh)$ such that $\pi T(t)=q(t)$ for every $t\in\mathbb{D}$:
\vspace*{0mm}
\begin{itemize}[leftmargin=26pt]
\setlength{\itemsep}{3pt}
\item $\sup\{\abs{\mathrm{Im}\,z}\colon z\in\sigma(A)\}<+\infty$;
\item the set $\mathrm{Re}\,\sigma(A)$ is finite and every `circle section' $\Omega_n\cap\{\abs{z}=r\}$ is a~symmetric set whose each connected component has length smaller than $\pi r$;
\item the set $\mathrm{Re}\,\sigma(A)$ is either finite or an~interval, and for every $s\in\mathrm{Re}\,\sigma(A)$, the section $\{t\in\R\colon s+\ii t\in\sigma(A)\}$ is a~half-line.
\end{itemize}

\noindent
For the proof, see Corollaries \ref{empty_cross_C} and \ref{milnorr_C}.

We also distinguish the class of compact metric spaces arising as inverse limits of the same form as $\Delta$ in assertion (i) above, and call them {\it admissible}. In Subsection~4.1 we study such spaces {\it per se}, and we provide necessary and sufficient conditions for the trivial extension $\Theta\in\Ext(X)$ to induce a~$C_0$-semigroup in $\QQ(\Hh)$, where $X$ is an~admissible compact metric space $X$.

%We also, in a sense, obtain a generalized version of Coburn's theorem which says that there is an~exact sequence
%$$
%\exi{i}{\phi}{\KK(H^2(\TT))}{\mathscr{T}}{C(\TT)}
%$$
%see \cite[p. 5]{douglas}. As Douglas writes on page 5, Coburn exhibited two nonequivalent extensions of $C(\TT)$ by the compacts, and motivated by this he raise the problem of determining all such extensions. About the same time Atiyah and Songer arrived at the same problem.

%%%%%%%%%%%%%%%%%%%%%%%%%%%%%%%%%%%%%%%%%%%%%%%%%%%%%%%%%%%%%%%%%%%%%%%%
%%%%%%%%%%%%%%%%%%%%%%%%%%%%%%%%%%%%%%%%%%%%%%%%%%%%%%%%%%%%%%%%%%%%%%%%
\section{Preliminaries and tools}
\subsection{Connections with the BDF theory}
Here, we shall collect some necessary basic facts from the Brown--Douglas--Fillmore theory of extensions, and explain its relation to our results. For more details on the BDF theory, the reader is referred to the seminal paper \cite{BDF} or to nice expositions of that theory in \cite[Ch.~VII]{blackadar-K} or \cite[Ch.~IX]{davidson}.

Let $X$ be a compact metric space. By an~{\it extension} of $C(X)$ ({\it by} $\KK(\Hh)$) we mean any pair $(\Aa,\p)$, where $\Aa$ is a~\cs-subalgebra of $\BB(\Hh)$ containing the compact operators and the~identity operator, and $\p\colon\Aa\to C(X)$ is a~$^\ast$-homomorphism such that
%$$
%\exi{\iota}{\p}{\KK(\Hh)}{\Aa}{C(X)}
%$$
\begin{equation*}
\begin{tikzcd}
0 \arrow[r] & \KK(\Hh) \arrow[r, "\iota"] & \mathcal{A} \arrow[r, "\p"] & C(X) \arrow[r] & 0
\end{tikzcd}
\end{equation*}
is an exact sequence, where $\iota$ is the inclusion map. To every extension one can associate the~so-called {\it Busby invariant} which is a~unital $^\ast$-monomorphism $\tau\colon C(X)\to\QQ(\Hh)$ defined as $\tau=\pi\p^{-1}$, which is the inverse of the~identification $\mathcal{A}/\KK(\Hh)\cong C(X)$. Conversely, any such $^\ast$-monomorphism gives rise to an~extension $(\pi^{-1}\tau(C(X)),\tau^{-1}\pi)$. In this setting, two extensions of $C(X)$ are called {\it equivalent} if the associated Busby invariants $\tau_1$ and $\tau_2$ satisfy $\tau_2=\pi(U)^\ast\tau_1\pi(U)$ for some unitary $U\in\BB(\Hh)$. We write $[\tau]$ for the extension (equivalence class) generated by a~unital $^\ast$-monomorphism $\tau\colon C(X)\to\QQ(\Hh)$.

Of fundamental importance is the fact that the~collection $\Ext(X)$ of all equivalence classes of extensions of $C(X)$ forms a~group when equipped with an~operation $+$ defined in terms of $^\ast$-monomorphisms $C(X)\to\QQ(\Hh)$ as $[\tau_1]+[\tau_2]=[\tau_1\oplus\tau_2]$. Here, we use an~isomorphism $\Hh\oplus\Hh\cong\Hh$ which allows us to identify the matrix algebra $\MM_2(\QQ(\Hh))$ with $\QQ(\Hh)$, as $\MM_2(\KK(\Hh))$ is mapped onto $\KK(\Hh)$. The zero element of that group can be constructed as follows. Take any infinite direct sum decomposition $\Hh=\bigoplus_{i=1}^\infty\Hh_i$, where each $\Hh_i$ is infinite-dimensional, pick a~countable dense subset $\{\xi_{i}\colon i\in\N\}$ of $X$ and define $\sigma\colon C(X)\to\BB(\Hh)$ by 
\begin{equation}\label{sigma_neutral}
\sigma(g)=\bigoplus_{i=1}^\infty g(\xi_{i})I_i,
\end{equation}
where $I_i$ is the identity operator on $\Hh_i$. Plainly, there are no nonzero compact operators in the range of $\sigma$, which implies that $\pi\sigma\colon C(X)\to\QQ(\Hh)$ is a~$^\ast$-monomorphism admitting the section $\sigma$, hence it determines the trivial extension of $C(X)$. That all trivial extensions are equivalent follows from the celebrated Weyl-von~Neumann--Berg theorem; see \cite{berg} and \cite[Thms.~II.4.6 and IX.2.1]{davidson}. Throughout this paper, we denote by $\Theta\in\Ext(X)$ the zero element in the~group of extensions of $X$, i.e. $\Theta=[\sigma]$, where $\sigma$ is given as in \eqref{sigma_neutral}. The property that each $[\tau]\in\Ext(X)$ has an~inverse, that is, there exists a~Busby invariant $\sigma$ with $[\tau\oplus\sigma]=\Theta$, is based on the~lifting property of positive unital maps on $C(X)$ and Naimark's dilation theorem (see \cite[Thm.~IX.5.1]{davidson}).

In this paper, we shall be concerned with extensions of some special projective limits of compact subsets of the complex plane. Recall that the projective (inverse) limit of an~inverse system $\{X_n,f_n\}_{n\geq 0}$, that is, a~sequence of topological spaces and continuous maps $f_n\colon X_{n+1}\to X_n$, is defined as
$$
\varprojlim X_n=\Big\{\mathbf{x}=(x_n)_{n=0}^\infty\in \prod_{n=0}^\infty X_n\colon f_n(x_{n+1})=x_n\,\,\,\mbox{for }n\geq 0\Big\}.
$$
We denote by $\pi_n\colon\varprojlim X_n\to X_n$ the projection onto the~$n^{\mathrm{th}}$ coordinate, that is, $\pi_n(\mathbf{x})=x_n$ for $\mathbf{x}=(x_n)_{n=0}^\infty$. The topology on $\varprojlim X_n$ is the weakest topology under which all the maps $\pi_n$ are continuous. 

As we will see, certain inverse limits of compact metric spaces arise naturally when considering operator semigroups indexed by positive dyadic rational numbers. We denote by $\mathbb{D}$ the set of such numbers, i.e. $\mathbb{D}=\{k2^{-m}\colon k,m\in\N\}$. The following observation will be used several times in the sequel. Namely, if $(q(2^{-n}))_{n\geq 0}$ is a~sequence of elements of any \cs-algebra $\mathcal{A}$ such that \vspace*{-1pt} $q(2^{-(n+1)})^2=q(2^{-n})$ for each $n=0,1,2,\ldots,$ then for any $t\in\mathbb{D}$ written in the form $t=t_0+\sum_{i=1}^j 2^{-m_i}$ with integers $t_0,j\geq 0$ and $1\leq m_1<\ldots<m_j$, we define $q(t)=q(1)^{t_0}q(2^{-m_1})\cdot\ldots\cdot q(2^{-m_j})$. It is then readily seen that $q(s+t)=q(s)q(t)$ for all $s,t\in\mathbb{D}$, hence we have extended $(q(2^{-n}))_{n\geq 0}$ to a~semigroup $(q(t))_{t\in\mathbb{D}}$, to which we shall refer as a~{\it dyadic semigroup}. Obviously, such an~extension is unique if we want to preserve the semigroup property. Given a~faithful $^\ast$-representation $\gamma$ of $\mathcal{A}$ on a~Hilbert space $\mathbb{H}$, we call $(q(t))_{t\in\mathbb{D}}$ a~$C_0$-{\it semigroup}, provided that there exists a~$C_0$-semigroup $(T(t))_{t\geq 0}\subset\BB(\mathbb{H})$ such that $T(t)=\gamma(q(t))$ for every $t\in\mathbb{D}$. In particular, if $(Q(t))_{t\in\mathbb{D}}\subset\BB(\Hh)$ is a~dyadic semigroup, we say that $(Q(t))_{t\in\mathbb{D}}$ is a~$C_0$-semigroup if it can be extended to a~$C_0$-semigroup in $\BB(\Hh)$ (in the usual sense) defined on $[0,\infty)$. This can be done if and only if $(Q(t))_{t\in\mathbb{D}}$ is \SOT-continuous and uniformly bounded on every bounded set of dyadic rationals (see Remark~\ref{cont1_R} below).

Anticipating our considerations presented in the next sections we distinguish the following two classes of compact metric spaces.

\begin{definition}\label{admissible_D}
{\bf (a) }We call a compact metric space $X$ {\it admissible} if $X=\varprojlim\{X_n,f_n\}_{n\geq 0}$ for some compact subsets $X_n$ of $\C$ being of the form
\begin{equation}\label{adm_Z_D}
X_n=\oo{\exp(2^{-n}Z)}\qquad (n=0,1,2,\ldots),
\end{equation}
where $Z\subset\C$ is a~fixed closed set such that $\sup_{z\in Z}\mathrm{Re}\,z<\infty$, and the connecting maps are all given by $f_n(z)=z^2$.

\vspace*{2mm}\noindent
{\bf (b) }Let $\gamma\colon\QQ(\Hh)\to\BB(\mathbb{H})$ be a~faithful $^\ast$-representation of the Calkin algebra. An~admissible compact metric space $X=\varprojlim\{X_n,f_n\}$ is said to have a~$\gamma$-$C_0$-{\it lifting property}, provided that the following condition holds true: For every $C_0$-semigroup $(q(t))_{t\in\mathbb{D}}$ with generator $A$ such that $\sigma(A)=Z$, where $Z$ satisfies \eqref{adm_Z_D}, every dyadic semigroup $(Q(t))_{t\in\mathbb{D}}$ satisfying:

\begin{itemize}[leftmargin=24pt]
\setlength{\itemsep}{3pt}
\item $\pi Q(t)=q(t)$, and 

\item $\n{Q(t)}\leq Ce^{\zeta t}$ for every $t\in\mathbb{D}$,
\end{itemize}
with some constants $C,\zeta<\infty$, is \SOT-continuous and therefore it is a~$C_0$-semigroup.
\end{definition}

The main starting point for our considerations is the fact that every normal $C_0$-semigroup in $\QQ(\Hh)$ gives rise to an~extension of an~admissible compact metric space which is completely described in terms of the spectrum of the infinitesimal generator. This will be proved in details in Section~3 (see Proposition~\ref{P_spectrum}). For now, we note that with every such extension one can naturally associate a~dyadic operator semigroup in $\QQ(\Hh)$.

\begin{definition}\label{ind_D}
Let $X$ be an admissible compact metric space and let $\Gamma\in\Ext(X)$ be the extension induced by a~Busby invariant $\tau\colon C(X)\to\QQ(\Hh)$. We say that $\Gamma$ {\it induces a~dyadic semigroup} $(q(t))_{t\in\mathbb{D}}\subset\QQ(\Hh)$, provided that $(q(t))_{t\in\mathbb{D}}$ is the unique semigroup extension of $(q(2^{-n}))_{n\geq 0}$ given by 
$$
q(2^{-n})=\tau(\pi_n)\quad\,\, (n=0,1,2,\ldots),
$$

\vspace*{1mm}\noindent
where $\pi_n\in C(X)$ is the projection onto the $n^{\mathrm{th}}$ coordinate.
\end{definition}

\begin{remark}\label{rem_exp}
If $X$ is an inverse limit of sets $X_n\subset\C$ given by \eqref{adm_Z_D}, then since any Busby invariant (being a~$^\ast$-monomorphism) is an~isometry, the dyadic semigroup $(q(t))_{t\in\mathbb{D}}$ induced by an~arbitrary element of $\Ext(X)$ satisfies the estimate
$$
\n{q(2^{-n})}\leq\exp(2^{-n}\zeta)\quad\,\, (n=0,1,2,\ldots),
$$
where $\zeta=\sup_{z\in Z}\mathrm{Re}\,z$.
\end{remark}

It is not obvious at all whether the dyadic semigroup induced by means of Definition~\ref{ind_D} is \SOT-continuous or, assuming it is \SOT-continuous and admits a~lift to an~operator semigroup in $\BB(\Hh)$, whether this lift must be a~$C_0$-semigroup. We shall deal with these two problems in Section~4 where we investigate a~special class of representations of $\QQ(\Hh)$ introduced by Calkin~\cite{calkin}. In particular, we show there that every admissible compact metric space with no isolated points has the $\gamma$-$C_0$-lifting property (see Theorem~\ref{no_isolated_T}).

\begin{definition}\label{C0_D}
For an admissible compact metric space $X$ and a~faithful $^\ast$-representation $\gamma\colon\QQ(\Hh)\to\BB(\mathbb{H})$, we define $\Ext_{C_0,\gamma}(X)\subseteq\Ext(X)$ to be the set of extensions inducing $C_0$-semigroups in $\QQ(\Hh)$.
\end{definition}

\begin{remark}\label{C0_correct_R}
The above definition is correctly posed in the sense that the question whether a~given $[\sigma]\in\Ext(X)$ induces a~$C_0$-semigroup does not depend on the choice of representative. For, suppose $U\in\BB(\Hh)$ is unitary and $\tau=\pi(U)^\ast\sigma\pi(U)$. Then, for any $\mathbf{x}\in\mathbb{H}$, $m,n\in\N$, and any complex polynomials $P,P_m\in\C[z]$, we have
\begin{equation*}
\begin{split}
    \gamma\big[\pi(U)^\ast\sigma(P_m\circ\pi_m) \pi(U)\big]\mathbf{x}&-\gamma(\sigma(P\circ\pi_n))\mathbf{x}\\
    &=\gamma\big[\pi(U)^\ast\sigma(P_m\circ\pi_m-P\circ\pi_n)\pi(U)\big]\mathbf{x}\xrightarrow[\,m\to\infty\,]{}0,
\end{split}
\end{equation*}
provided that $\sigma$ induces a~$C_0$-semigroup and $P_m\circ\pi_m$ converge pointwise on $X$ to $P\circ\pi_n$. This shows that $\tau$ generates a~$C_0$-semigroup if and only if so does $\sigma$. We have used here an~observation that each element of the induced dyadic semigroup corresponds, as in Definition~\ref{ind_D}, to a~complex polynomial on the $n^{\mathrm{th}}$ coordinate of $X$, for some $n\in\N$ (see also the proof of Lemma~\ref{lifting_L}).
\end{remark}

%%%%%%%%%%%%%%%%%%%%%%%%%%%%%%%%%%%%%%%%%%%%%%%%%%%%%%%%%%%%%%%%%%%%%%%%%%%%%%%%%%%
%%%%%%%%%%%%%%%%%%%%%%%%%%%%%%%%%%%%%%%%%%%%%%%%%%%%%%%%%%%%%%%%%%%%%%%%%%%%%%%%%%%
\subsection{Lifting problems}
We shall now briefly explain how the lifting problem for operator semigroups is related to the lifting problem for separable abelian \cs-subalgebras of $\QQ(\Hh)$. We start with a~lemma saying that, apart from the continuity issue, our lifting problem is really about deciding whether the induced extension is the trivial one.

\begin{lemma}[{\bf `Lifting lemma'}]\label{lifting_L}
Let $X$ be an admissible compact metric space having the $\gamma$-$C_0$-lifting property, for a~fixed faithful $^\ast$-representation $\gamma\colon\QQ(\Hh)\to\BB(\mathbb{H})$, and let $\Gamma\in\Ext_{C_0,\gamma}(X)$. Then $\Gamma=\Theta$ if and only if the dyadic semigroup $(q(t))_{t\in\mathbb{D}}\subset\QQ(\Hh)$ induced by $\Gamma$ admits a~lift to a~dyadic $C_0$-semigroup of normal operators $(Q(t))_{t\in\mathbb{D}}\subset\BB(\Hh)$ {\rm (}i.e. $\pi Q(t)=q(t)$ for every $t\in\mathbb{D}${\rm )} such that the spectra of $q(2^{-n})$ and $Q(2^{-n})$ coincide for each $n\in\N$.
\end{lemma}
\begin{proof}
The necessity follows almost directly from Definitions~\ref{admissible_D}(b) and \ref{C0_D}. Indeed, suppose that $\Gamma=\Theta$ is given by a~Busby invariant $\sigma\colon C(X)\to\QQ(\Hh)$, $\sigma=\pi\p^{-1}$, where $\p$ is the quotient map. Let $\rho\colon C(X)\to\BB(\Hh)$ be the right section, that is, $\p\rho=\mathrm{id}_{C(X)}$, and define $Q(2^{-n})=\rho(\pi_n)$ for $n=0,1,2,\ldots$ Next, let $(Q(t))_{t\in\mathbb{D}}\subset\BB(\Hh)$ be the dyadic extension. Notice that $\pi Q(2^{-n})=\pi\p^{-1}\p\rho(\pi_n)=\pi\p^{-1}(\pi_n)=\sigma(\pi_n)=q(2^{-n})$ for each $n=0,1,2,\ldots$, and that the analogous relation holds true for all dyadic numbers. Indeed, if $\mathbb{D}\ni t=t_0+\sum_{i=1}^j 2^{-m_i}$, where $t_0, j\geq 0$ and $1\leq m_1<\ldots<m_j$ \vspace*{-1pt}are integers, then $q(t)=q(1)^{t_0}q(2^{-m_1})\cdot\ldots\cdot q(2^{-m_j})$. Since $q(2^{-k})=q(2^{-m_j})^{2^{m_j-k}}$ for every $k\leq m_j$, we have $\pi Q(t)=\sigma(\pi_{m_j}^N)=q(t)$, where $N=t_02^{m_j}+\sum_{i=1}^j 2^{m_j-m_i}$. Moreover,\vspace*{-1pt} as $Q(t)=\rho(\pi_{m_j}^N)$ and $\rho$ is an~isometry, we infer that $\n{Q(t)}\leq\exp\small(2^{-m_j}N\zeta\small)=\exp\small(t\zeta\small)$ with $\zeta$ as in Remark~\ref{rem_exp}. Therefore, since $X$ has the $\gamma$-$C_0$-lifting property, $(Q(t))_{t\in\mathbb{D}}$ is a~$C_0$-semigroup. Notice also that the spectrum of $Q(2^{-n})$ is the same as the one of $q(2^{-n})$, as $\rho$ is a~unital $^\ast$-monomorphism.

For the converse, let again $\sigma$ be the Busby invariant for $\Gamma$, $\p$ be the quotient map, and assume that the dyadic semigroup determined by the formula $q(2^{-n})=\sigma(\pi_n)$ ($n=0,1,2,\ldots$) admits a~lifting to a~$C_0$-semigroup $(Q(t))_{t\in\mathbb{D}}$. Write $X=\varprojlim X_n$ and let $\mathbf{A}\subset C(X)$ be the $^\ast$-subalgebra of functions which coincide with a~polynomial in variables $z$, $\oo{z}$ on some $X_n$, that is, $\mathbf{A}=\{P\circ\pi_n\colon P\in\C[z,\oo{z}],\,n\in\N\}$. Note that $\mathbf{A}$ is closed under addition and multiplication due to the fact that every polynomial acting on $X_n$ is also a~polynomial on $X_m$, for every $m>n$. Define $\rho$ on $\mathbf{A}$ by $\rho(f)=P(Q(2^{-n}),Q(2^{-n})^\ast)$ for $f=P\circ\pi_n$. Since addition, multiplication and involution in $\mathbf{A}$ reduce to the same operations on polynomials on some $X_n$, we see that $\rho\colon\mathbf{A}\to\BB(\Hh)$ is a~unital $^\ast$-homomorphism. Moreover, for any $f\in\mathbf{A}$ as above, $\|\rho(f)\|$ equals the supremum of $\vert P(z,\oo{z}\,)\vert$ over $z$ in the spectrum of $Q(2^{-n})$, which is the same as the spectrum of $q(2^{-n})$, and hence $\|\rho(f)\|=\sup\{\vert P(z,\oo{z}\,)\colon z\in X_n\}=\|f\|_\infty$, the supremum norm in $C(X)$. By the Stone--Weierstrass theorem, $\mathbf{A}$ is dense in $C(X)$, thus there exists a~unique continuous extension of $\rho$ to a~$^\ast$-homomorphism. Since for every $f\in\mathbf{A}$, $f=P\circ\pi_n$, we have
\begin{equation*}
\begin{split}
    \p\rho(f) &=\p\big[P(Q(2^{-n}),Q(2^{-n})^\ast)\big]=P\big[\p(Q(2^{-n})),\p(Q(2^{-n}))^\ast\big]\\
    &=P\big[\p\rho(\pi_n),\oo{\p(\rho(\pi_n)}\big]=P(\pi_n,\oo{\pi_n})=f,
\end{split}
\end{equation*}
we conclude that the obtained extension is a~right section for $\Gamma$. Hence, $\Gamma=\Theta$.
\end{proof}

Given a \cs-subalgebra of $\QQ(\Hh)$ generated by some elements of an~operator semigroup, our goal is to find a~lift in $\BB(\Hh)$ which preserves all the algebraic relations in the original semigroup (and possibly preserves \SOT-continuity). Such approach can be regarded as a~`semigroup' version of the usual lifting problem which is now quite well understood. By a~{\it lift} of a~\cs-subalgebra $\mathcal{A}\subseteq\QQ(\Hh)$ we mean any \cs-subalgebra $\mathcal{E}\subseteq\BB(\Hh)$ such that $\pi[\mathcal{E}]=\mathcal{A}$. The algebra $\mathrm{C}^\ast(\pi(S))\cong C(\mathbb{T})$, where $S$ is the~unilateral shift, is the prototypical example of a~commutative \cs-subalgebra of $\QQ(\Hh)$ with no abelian lift (see \cite[Prop.~12.4.3]{farah}). On the other hand, subalgebras which admit an~abelian lift are completely characterized with the aid of the Weyl--von~Neumann--Berg--Sikonia theorem.
\begin{theorem}[{see \cite[Cor.~12.4.5]{farah}}]
A separable commutative \cs-subalgebra of $\QQ(\Hh)$ has an~abelian (diagonalizable) lift if and only if it is included in a~commutative \cs-subalgebra of $\QQ(\Hh)$ of real rank zero.
\end{theorem}

Not surprisingly, there are some similarities between the `usual' and the `semigroup' versions of the lifting problem. Recall that a~commutative \cs-algebra $C_0(X)$ has real rank zero if and only if $X$ is zero-dimensional, and hence every separable commutative \cs-subalgebra of $\QQ(\Hh)$ with zero-dimensional spectrum has an~abelian lift. Similarly, according to \cite[Thm.~1.15]{BDF}, for any zero-dimensional compact metric space $X$, we have $\Ext(X)=0$. This means that every normal $C_0$-semigroup in $\QQ(\Hh)$ which induces an~extension of a~zero-dimensional space (as in Proposition~\ref{P_spectrum}) admits a~semigroup lift, which may also be \SOT-continuous depending on whether Lemma~\ref{lifting_L}, and the results of Section~4, are applicable.
\begin{example}\label{ex_2k}
Let $Z=\{2k\pi\ii\colon k\in\Z\}$ and 
$$
X_n=\oo{\exp(2^{-n}Z)}=\{\exp\small(2^{-n+1}k\pi\ii\small)\colon k=0,1,\ldots,2^n-1\},
$$
which is the set of all roots of unity of degree $2^n$, for $n=0,1,2,\ldots$ Let also $f_n\colon X_{n+1}\to X_n$ be given by $f(z)=z^2$. Then $X=\varprojlim\{X_n,f_n\}_{n\geq 0}$ is homeomorphic to the Cantor set, hence $\Ext(X)=0$. Consequently, any normal $C_0$-semigroup in $\QQ(\Hh)$ whose infinitesimal generator has spectrum $Z$ admits a~semigroup lift to $\BB(\Hh)$ (see Propostion~\ref{P_spectrum}).
\end{example}

%To convince ourselves that the `semigroup' variant of the lifting problem does not just reduce to the question whether the whole corresponding extension group is trivial, we consider the next example.
\begin{example}\label{ex_imaginary}
Let $Z=\ii\R$ be the imaginary axis, so that $X_n=\exp(2^{-n}Z)=S^1$ for every $n=0,1,2,\ldots$ Let again each $f_n\colon S^1\to S^1$ be given by $f_n(z)=z^2$. Then, the inverse limit 
$$
\Sigma_2=\varprojlim\{S^1,z^2\}
$$
is the~$2$-{\it adic solenoid}. It is known (see \cite{KS}) that $\Ext(\Sigma_2)=0$. This can be proved, for example, by using Milnor's exact sequence (quoted below in Theorem~\ref{milnor_thm}), which in this case has the form
$$
\begin{tikzcd}
0 \arrow[r] & \varprojlim^{(1)}\Ext(S^2) \arrow[r] & \Ext(\Sigma_2) \arrow[r] & \varprojlim \Z \arrow[r] & 0.
\end{tikzcd}
$$
By the well-known result \cite[Cor.~7.1]{BDF}, we have $\Ext(S^2)=0$, 
and since the connecting maps $\Z\to\Z$ in the quotient group are all given by $x\mapsto 2x$, we obviously have $\varprojlim\Z=0$. Therefore, $\Ext(\Sigma_2)$ is trivial and hence any normal $C_0$-semigroup in $\QQ(\Hh)$ whose generator has spectrum $\ii\R$ admits a~semigroup lift to $\BB(\Hh)$.
\end{example}

We will return to this example in Section~4 to show the lack of \SOT-continuity, despite of the fact that all extensions of $\Sigma_2$ are trivial (see Example~\ref{T_not_cont_E}). Roughly speaking, the reason is that in the~infinite toin coss, both outcomes occur infinitely often almost surely.

%(...) With every element of $\Ext(X)$ one can thus associate a~dyadic operator semigroup in $\QQ(\Hh)$ which, however, does not need to be \SOT-continuous. Below, we give an~example of such situation for $X=\varprojlim X_n$ and each $X_n$ being the unit circle $\TT$.

%\begin{theorem}
%Every admissible space with no isolated points has the $C_0$-lifting property.
%\end{theorem}

%%%%%%%%%%%%%%%%%%%%%%%%%%%%%%%%%%%%%%%%%%%%%%%%%%%%%%%%%%%%%%%%%%%%%%%%%
%%%%%%%%%%%%%%%%%%%%%%%%%%%%%%%%%%%%%%%%%%%%%%%%%%%%%%%%%%%%%%%%%%%%%%%%%
\section{Extensions generated by semigroups}
\subsection{An inverse limit associated with the spectrum}
As we have already announced, normal $C_0$-semigroups in the Calkin algebra naturally generate extensions of admissible compact metric spaces. Of course, to speak sensibly about $C_0$-semigroups we need to choose a~faithful $^\ast$-representation of $\QQ(\Hh)$ on a~Hilbert space $\mathbb{H}$. In the results of the present section, the choice of representation is arbitrary. 

Below, we can either assume that $(q(t))_{t\geq 0}$ is a~$C_0$-semigroup of normal elements of $\QQ(\Hh)$, or assume a~weaker condition that the~dyadic semigroup $(q(t))_{t\in\mathbb{D}}$ is a $C_0$-semigroup (i.e. it can be extended to a~$C_0$-semigroup $(T(t))_{t\geq 0}\subset\BB(\mathbb{H})$), however, we then need to assume that all $T(t)$ are normal operators. The reason is that we apply the spectral mapping theorem for normal $C_0$-semigroups (see \cite[Cor.~2.12]{EN}).

\begin{proposition}\label{P_spectrum}
Let $(q(t))_{t\geq 0}\subset\QQ(\Hh)$ be a $C_0$-semigroup of normal operators in the Calkin algebra. Let
$$
\Aa_0=\mathrm{C}^\ast\big(\{q(2^{-n})\colon n=\infty,0,1,2,\ldots\}\big)
$$

\vspace*{1mm}\noindent
be the \cs-subalgebra of $\QQ(\Hh)$ generated by the identity and all $q(2^{-n})$ for $n\in\N_0$, and let
$$
\mathcal{E}=\pi^{-1}(\Aa_0)
$$
be the \cs-subalgebra of $\BB(\Hh)$ generated by $\{q(2^{-n})\colon n=\infty,0,1,2,\ldots\}+\KK(\Hh)$.

\vspace*{1mm}
\begin{itemize}[leftmargin=20pt]
\setlength{\itemsep}{2pt}
    \item[{\rm (a)}] Let $A$ be the generator of $(q(t))_{t\geq 0}$ and define
    $$
    \Omega_n=\oo{\exp(2^{-n}\sigma(A))}\quad (n=0,1,2,\ldots)
    $$
Then, $\Aa_0$ is a~commutative \cs-algebra and its maximal ideal space $\Delta$ is homeomorphic to the projective limit of the inverse system $\{\Omega_n,p_n\}_{n\geq 0}$, where $p_n(z)=z^2$ for each $n=0,1,2,\ldots$
    
    \item[{\rm (b)}] The \cs-algebra $\mathcal{E}$ contains $\KK(\Hh)$ as an ideal and there is an~exact sequence
    $$
    \begin{tikzcd}
    0 \arrow[r] & \KK(\Hh) \arrow[r, "\iota"] & \mathcal{E} \arrow[r, "\theta"] & C(\Delta) \arrow[r] & 0,
    \end{tikzcd}
    $$
where $\theta(T)=\widehat{\pi(T)}$ and $\Aa_0\ni q\xmapsto[\phantom{xx}]{}\widehat{q}\in C(\Delta)$ is the Gelfand transform.
\end{itemize}
\end{proposition}

\begin{proof}
(a) First, observe that since there exists $\zeta<\infty$ such that $\mathrm{Re}\,\lambda\leq\zeta$ for each $\lambda\in\sigma(A)$, all the sets $\Omega_n$ are compact subsets of $\C$. Moreover, $A$ is normal and if $\mathsf{E}^A$ stands for the spectral decomposition of $A$, then each $q(t)$ can be calculated via functional calculus in $L_\infty(\mathsf{E}^A)$ by 
$$
q(t)=\int_{\sigma(A)}e^{t\lambda}\,\dd \mathsf{E}^A(\lambda)\qquad (t\geq 0),
$$
as $\vert e^{t\lambda}\vert\leq e^{t\zeta}$ and hence the function under the integral is bounded. Plainly, $q(s), q(t), q(t)^\ast$ commute for all $s,t\geq 0$, thus $\Aa_0$ is commutative.

The joint spectrum of the set $\{q(2^{-n})\colon n=0,1,2,\ldots\}$ is a~compact subset of $\C^\infty$ defined by
$$
\sigma_{\Aa_0}\big(q(2^{-n})\colon n=0,1,2,\ldots\big)=\big\{(\p(q(2^{-n})))_{n=0}^\infty\colon \p\in\Delta\big\}
$$
and the map 
$$
\Delta\ni\p\xmapsto[\phantom{xxx}]{}(\p(q(2^{-n})))_{n=0}^\infty
$$
is a homeomorphism between $\Delta$ and $\sigma_{\Aa_0}(q(2^{-n})\colon n=0,1,2,\ldots)$ (see \cite[Cor.~3.1.13]{rickart}). On the other hand, a~sequence $\boldsymbol{\lambda}=(\lambda)_{n=1}^\infty\in\C^\infty$ belongs to $\sigma_{\Aa_0}(q(2^{-n})\colon n=0,1,2,\ldots)$ if and only if
\begin{equation}\label{q_lambda}
q(\boldsymbol{\lambda})\coloneqq \sum_{n=0}^\infty 2^{-n}\frac{(\lambda_n I-q(2^{-n}))^\ast(\lambda_n I-q(2^{-n}))}{\n{\lambda_n I-q(2^{-n})}^2}
\end{equation}
is not invertible in $\QQ(\Hh)$. Indeed, as each summand is a~positive operator, we infer that for every linear multiplicative functional $\p\in\Delta$ we have $\p(q(\boldsymbol{\lambda}))=0$ if and only if $\p(q(2^{-n}))=\lambda_n$ for each $n=0,1,2,\ldots$ Hence, if $q(\boldsymbol{\lambda})$ is not invertible we pick $\p\in\Delta$ so that $\p(q(\boldsymbol{\lambda}))=0$ to see that $\boldsymbol{\lambda}$ belongs to the joint spectrum. Conversely, if $q(\boldsymbol{\lambda})$ is invertible, then we have $\p(q(\boldsymbol{\lambda}))\neq 0$ for every $\p\in\Delta$, thus $\boldsymbol{\lambda}$ is not in the joint spectrum.

Fix any $\boldsymbol{\lambda}=(\lambda_n)_{n=0}^\infty\in\C^\infty$. The operator $(\lambda_n I-q(2^{-n}))^\ast (\lambda_n I-q(2^{-n}))$ corresponds via functional calculus to the map $\phi_n\in L_\infty(\mathsf{E}^A)$ given by
$$
\phi_n(z)=\vert\lambda_n-\exp(2^{-n}z)\vert^2.
$$
For every $z\in\sigma(A)$, we have $\mathrm{Re}\, z\leq\zeta$ and hence
$$
\n{\phi_n}_\infty\leq \big(\abs{\lambda_n}+\exp(2^{-n}\zeta)\big)^2
$$
which implies that each denominator in formula \eqref{q_lambda} is majorized by a~constant and cannot become arbitrarily large after applying functional calculus and varying $z$ over $\sigma(A)$. Hence, $q(\boldsymbol{\lambda})$ is noninvertible if and only if $0$ lies in the closure of the range of the map
$$
\sigma(A)\ni z\xmapsto[\phantom{xxx}]{}\sum_{n=0}^\infty 2^{-n}\frac{\phi_n(z)}{\n{\phi_n}_\infty},
$$
which implies that each $\lambda_n$ must belong to the closure of $\exp(2^{-n}\sigma(A))$ which is denoted by $\Omega_n$. Moreover, for any $n=0,1,2,\ldots$ we can pick $z\in\sigma(A)$ so that both $\phi_n(z)$ and $\phi_{n+1}(z)$ are arbitrarily close to zero. Since $\exp(2^{-n-1}z)^2=\exp(2^{-n}z)$, we infer that for $q(\boldsymbol{\lambda})$ being noninvertible we also must have $\lambda_{n+1}^2=\lambda_n$ ($n=0,1,\ldots$). This means that every element of the joint spectrum belongs to the inverse limit $\varprojlim\Omega_n$. 

%and this is in turn equivalent to the existence of a~sequence $(z_i)_{i=1}^\infty\subset \sigma(A)$ such that $\lim_{i\to\infty}\sup_n \abs{\phi_n(z_i)}=0$. Therefore, each $\lambda_n$ must belong to the closure of $\exp(2^{-n}\sigma(A))$ which we denote by $\Omega_n$. Moreover, since $\exp(2^{-n-1}z)^2=\exp(2^{-n}z)$ and the sequence $(z_i)_{i=1}^\infty$ is common for all the $\lambda_n$, we infer that for $q(\boldsymbol{\lambda})$ being noninvertible we also must have $\lambda_{n+1}^2=\lambda_n$ ($n=0,1,\ldots$). 

Conversely, fix any $\boldsymbol{\lambda}\in\varprojlim\Omega_n$ and any $\e>0$. Take $C>1$ such that $\vert \lambda_n+\exp(2^{-n}z)\vert\leq C$ for all $n\in\N_0$ and $z\in\sigma(A)$, and define
$$
W_n=\big\{z\in\sigma(A)\colon \vert \lambda_n-\exp(2^{-n}z)\vert\leq C^{-n}\e\big\}\quad (n=0,1,2,\ldots).
$$
Observe that for $z\in W_n$, we have
\begin{equation*}
    \begin{split}
        \vert \lambda_{n-1}-\exp(2^{-n+1}z)\vert &=\vert\lambda_n^2-\exp(2\cdot 2^{-n}z)\vert\\
        &=\vert \lambda_n-\exp(2^{-n}z)\vert\cdot\vert \lambda_n+\exp(2^{-n}z)\vert\leq C^{-n+1}\e
    \end{split}
\end{equation*}
which shows that $z\in W_{n-1}$ and, similarly, $z\in W_{n-2},\ldots,W_0$. Therefore, $\bigcap_{j=0}^n W_j\neq\varnothing$. This means that the identity map
$$
\mathrm{id}\colon \sigma_{\Aa_0}(q(2^{-n})\colon n=0,1,\ldots)\xrightarrow[\phantom{xxx}]{}\,\varprojlim\Omega_n
$$
has dense range. Thus, it is an~onto homeomorphism, as both topologies are the product topology which is compact and Hausdorff. Consequently, $\Delta$ is homeomorphic to $\varprojlim\Omega_n$.

\vspace*{2mm}\noindent
(b) Of course, $\KK(\Hh)$ forms an ideal in $\mathcal{E}$. For every $T\in\mathcal{E}$, we have $\pi(T)\in\Aa_0$ and each element in $\Aa_0$ is of this form. Hence, the formula $\theta(T)= \skew{4}\widehat{\textrm{\emph{\rule{0ex}{1.3ex}\smash{$\pi(T)$}}}}$ yields a~$^\ast$-homomorphism onto $C(\Delta)$. Obviously, $T\in\mathrm{ker}\,\theta$ if and only if $\pi(T)=0$, i.e. $T\in\KK(\Hh)$.
\end{proof}

Having established the fact that every normal $C_0$-semigroup $(q(t))_{t\geq 0}$ in $\QQ(\Hh)$ generates an~extension of $C(\Delta)$ by $\KK(\Hh)$, with $\Delta$ depending only on the spectrum of the generator of $(q(t))_{t\geq 0}$, we now describe the strategy of finding conditions which would guarantee that the resulting extension is the trivial one. To this end, we need to recall some more facts from the BDF theory. 

Given two compact metric spaces $X$ and $Y$, and a~continuous map $f\colon X\to Y$, there is an~induced map $f_\ast\colon\Ext(X)\to\Ext(Y)$ defined as
$$
f_\ast(\tau)(g)=\tau(g\circ f)\oplus \sigma(g)\quad (g\in C(Y)),
$$
where $\sigma$ is any $^\ast$-monomorphism corresponding to the trivial extension of $C(Y)$. We add the second direct summand in order to guarantee that the resulting map $f_\ast(\tau)$ is injective. One can verify that $(fg)_\ast=f_\ast g_\ast$ whenever these compositions make sense.

Recall that for any compact metric space $X$, the {\it cone} $CX$ over $X$ is obtained from $X\times I$ by collapsing $X\times\{0\}$ to a~single point, where $I=[0,1]$. The {\it suspension} $SX$ is obtained from $X\times I$ by collapsing $X\times\{0\}$ and $X\times\{1\}$ to two distinct points.

The extension functor is defined for ranks $q\leq 1$ by $\Ext_q(X)=\Ext(S^{1-q}X)$. It was shown in \cite[\S 6]{BDF} that, analogously to Bott's periodicity in $K$-theory, there exist isomorphisms
$$
\mathrm{Per}_\ast\colon \Ext_{q-2}(X)\xrightarrow[\phantom{xx}]{}\Ext_q(X)\quad (r\leq 1).
$$
This allows us to extend the definition of $\Ext$ to all integer dimensions:
$$
\Ext_q(X)=\left\{\begin{array}{ll}
\Ext(X) & \mbox{if }q\mbox{ is odd},\\
\Ext(SX) & \mbox{if }q\mbox{ is even}.
\end{array}\right.
$$
Since every continuous map $f\colon X\to Y$ naturally induces a~map $Sf\colon SX\to SY$ by appropriately quotienting $f\times\mathrm{id}\colon X\times I\to Y\times I$, there is an~induced homomorphism $f_\ast\colon \Ext_q(X)\to\Ext_q(Y)$, for any $q\in\Z$. We also define extensions of compact pairs by the obvious formula $\Ext_q(X,A)=\Ext_q(X/A)$. Therefore, for any admissible map $f\colon (X,A)\to (Y,B)$ between compact pairs (i.e. $f$ is continuous and $f(A)\subseteq B$), there is an~induced homomorphism $f_\ast\colon \Ext_q(X,A)\to\Ext_q(Y,B)$.

Suppose $\{X_n,p_n\}_{n=0}^\infty$ is an~inverse system of compact metric spaces. Let $X=\varprojlim X_n$ and $q_n\colon X\to X_n$ stand for the coordinate maps, for $n\in\N_0$, so that $p_nq_{n+1}=q_n$. Hence, we have another inverse system of groups $\{\Ext(X_n),p_{n\ast}\}_{n=0}^\infty$. Since $p_{n\ast}q_{(n+1)\ast}=q_{n\ast}$, we can define an~{\it induced map} 
$$
P\colon \Ext(X)\to\varprojlim\Ext(X_n),\quad P(\tau)=(q_{n\ast}\tau)_{n=0}^\infty. 
$$
By \cite[Thm.~8.4]{BDFu}, the induced map is always surjective, but in general not injective. However, Milnor \cite{milnor} showed that for any homology theory satisfying the Steenrod axioms, except the dimension axiom (see \cite{ES}), one can build an~exact sequence which measures the lack of continuity of the $\Ext$-functor with respect to inverse limit (see \cite[\S 5]{Dou}).
\begin{theorem}[see {\cite[Thm.~4]{milnor}} and {\cite[Cor.~7.4]{BDF}}]\label{milnor_thm}
For any inverse system $\{X_n\}$ of compact metric spaces, and any $k\in\Z$, there exists an~exact sequence
\begin{equation*}
\begin{tikzcd}
0 \arrow[r] & \varprojlim{}^{(1)}\Ext_{k+1}(X_n) \arrow[r] & \Ext_k(\varprojlim X_n) \arrow[r, "P"] & \varprojlim\Ext_k(X_n) \arrow[r] & 0
\end{tikzcd}
\end{equation*}
where $\varprojlim{}^{(1)}$ is the first derived functor of inverse limit.
\end{theorem}
Applying Milnor's theorem to the inverse system $\{\Omega_n,p_n\}_{n\geq 0}$ produced by Proposition~\ref{P_spectrum} (and $k=1$), we arrive at the functor $\varprojlim{}^{(1)}\Ext(S\Omega_n)$, at \vspace*{-1pt}which we shall have a~closer look in Section~5. We provide therein geometric conditions on $\sigma(A)$ which imply that the corresponding connecting maps are surjective. Those conditions are later applied in Section~6, where we verify that in some situations the first derived functor vanishes.

Our next step at this point is to decide when the element of $\Ext(\Omega)$ produced via Proposition~\ref{P_spectrum} actually belongs to the kernel of $P$.

\subsection{Fredholm index map and the induced homomorphism}
% First, we shall look for conditions on $A$ which would ensure that the resulting extension corresponds canonically to the zero element of the group $\varprojlim\Ext(\Omega_n)$.

%$A$ of $(q(t))_{t\geq 0}$. In the sequel, we will investigate more closely the relation between $(q(t))_{t\geq 0}$ and the associated element of $\Ext(\Delta)$. 
Henceforth, we will be using the notation introduced in Proposition~\ref{P_spectrum} without explanation. In particular, we identify the extension given by $(\mathcal{E},\theta)$ with an~element of $\Ext(\Omega)$. We are going to show that the classical Brown--Douglas--Fillmore condition imposed on each $q(t)$ implies that the resulting extension corresponds canonically to the zero element of the group $\varprojlim\Ext(\Omega_n)$.

Recall that a crucial fact in the theory of essentially normal operators is that for any $X$ being a~compact subset of $\C$, the group $\Ext(X)$ can be nicely described with the aid of the Fredholm index map. In $C(X)$ consider the relation of homotopy equivalence and let $\GG_0(C(X))$ be the~equivalence class of the constant one function. By $\pi^1(X)$ we denote the group $\GG(C(X))/\GG_0(C(X))$ of homotopy classes of invertible functions. 
\begin{theorem}[{\cite[Thm.~10.5]{BDFu}; see also \cite[Thm.~IX.7.2]{davidson}}]\label{T_index}
For any compact set $X\subset\C$, there is a~well-defined map 
$$
\gamma\colon\Ext(X)\xrightarrow[\phantom{xx}]{}\Hom(\pi^1(X),\Z),\quad \gamma[\tau]([f])=\Find\,\tau(f) 
$$
which is a group isomorphism.
\end{theorem}

\begin{proposition}\label{P_kernel}
Let $(q(t))_{t\geq 0}\subset\QQ(\Hh)$ be a $C_0$-semigroup of normal operators and let

\vspace*{1mm}
$$
P\colon \Ext(\Omega)\xrightarrow[\phantom{xx}]{}\varprojlim\Ext(\Omega_n),\quad \mbox{where }\,\Omega=\varprojlim\Omega_n\approx\Delta,
$$
be the induced surjective map. Then, $(\mathcal{E},\theta)\in\ker\,P$ if and only if 
$$
\Find(\lambda I-q(2^{-n}))=0\quad\mbox{for all }\,n\in\N_0,\,\lambda\not\in\Omega_n.
$$
\end{proposition}
\begin{proof}
First of all, notice that since each $q(t)$ is normal, the semigroup $(q(t))_{t\geq 0}$ satisfies the spectral mapping theorem, that is, $\sigma(q(2^{-n}))=\Omega_n$ (see \cite[Cor.~V.2.12]{EN}).

For any $n\in\N_0$, choose a $^\ast$-monomorphism $\sigma_n\colon C(\Omega_n)\to\BB(\Hh)$ determining the zero element of the group $\Ext(\Omega_n)$. The extension $(\mathcal{E},\theta)$ produced by $(q(t))_{t\geq 0}$ as in Proposition~\ref{P_spectrum} is generated by the $^\ast$-monomorphism $\tau\colon C(\Delta)\to\QQ(\Hh)$ given by $\tau(f)=\pi\theta^{-1}(f)$. Since $\Delta\approx \Omega$, there is the~corresponding $^\ast$-monomorphism $\w\tau\colon C(\Omega)\to\QQ(\Hh)$. By the definition of $P$, we have
$$
P(\w\tau)=(q_{n\ast}(\w\tau))_{n=0}^\infty\in \varprojlim\Ext(\Omega_n),
$$
where $q_n$ is the $n^{\mathrm{th}}$ coordinate map for the inverse system $\{\Omega_n,p_n\}_{n=0}^\infty$. Hence, $\w\tau\in\ker\, P$ if and only if for every $n\in\N_0$, $q_{n\ast}(\w\tau)$ determines the trivial extension of $C(\Omega_n)$ which is in turn equivalent to the condition that for each $n\in\N_0$ there is a~unitary operator $U_n\in\BB(\Hh)$ such that
\begin{equation}\label{equiv_uni}
    \pi(U_n)^\ast [\w\tau(g\circ q_n)\oplus \pi\sigma_n(g)]\pi(U_n)=\pi\sigma_n^{(2)}(g)\quad\mbox{for every }g\in C(\Omega_n),
\end{equation}
where $\pi\sigma_n^{(2)}$ stands for the amplification of $\pi\sigma_n$ to $\MM_2(\QQ(\Hh))$, which yields the zero element of $\Ext(\Omega_n)$ under the identification $\Hh\cong \Hh\oplus\Hh$. Comparing the two blocks in \eqref{equiv_uni} we can rewrite that condition as
$$
\pi(U_n)^\ast \w\tau(g\circ q_n)\pi(U_n)=\sigma_n(g)\quad\mbox{for every }g\in C(\Omega_n).
$$
Therefore, $\w\tau\in\ker\,P$ if and only if for every $n\in\N_0$, the $^\ast$-monomorphism
\begin{equation}\label{map_tau}
C(\Omega_n)\ni g\xmapsto[\phantom{xxx}]{}\w\tau(g\circ q_n)\in\QQ(\Hh)
\end{equation}
yields the trivial extension of $C(\Omega_n)$.

Now, for any $n\in\N_0$, consider the isomorphism given by Theorem~\ref{T_index}, that is,
$$
\gamma_n\colon \Ext(\Omega_n)\xrightarrow[\phantom{xx}]{} \Hom(\pi^1(\Omega_n),\Z),\quad \gamma_n[\mu]([f])=\Find\,\mu(f).
$$
Appealing to the fact that the group of homotopy classes $\pi^1(\Omega_n)$ is the free abelian group generated by $\{[z-\lambda_j]\}$, with one $\lambda_j$ in each bounded component of $\C\setminus\Omega_n$ (see \cite[Thm.~IX.7.1]{davidson}), we infer that \eqref{map_tau} yields the trivial extension if and only if
\begin{equation}\label{index_bfx}
\Find\, \w\tau(\mathbf{x}\xmapsto[\phantom{xx}]{} x_n-\lambda)=0
\end{equation}
for any $\lambda$ lying in a~bounded component of $\C\setminus\Omega_n$.

Recall that $\Delta$ and $\Omega$ are homeomorphic via the map $\varrho\colon \p\xmapsto[]{\phantom{xx}} (\p(q(2^{-n})))_{n=0}^\infty$ (see the proof of Proposition~\ref{P_spectrum}). Let $\psi\colon C(\Delta)\to C(\Omega)$ be the isomorphism given by $\psi(f)=f\circ \varrho^{-1}$, so that we have the corresponding extension $(\mathcal{E},\psi\theta)$ of $C(\Omega)$. Notice that for any $T\in\mathcal{E}$ we have:
\begin{equation*}
\begin{split}
    \psi\theta(T)=q_n\,\,\,\, & \Longleftrightarrow\,\,\,\, \psi(\widehat{\pi(T)})=q_n\\
     & \Longleftrightarrow\,\,\,\, \widehat{\pi(T)}(\varrho^{-1}(\mathbf{x}))=x_n\quad (\mathbf{x}\in\Omega)\\
     & \Longleftrightarrow\,\,\,\, \widehat{\pi(T)}(\p)=\p(q(2^{-n}))\quad (\p\in\Delta)\\
     & \Longleftrightarrow\,\,\,\,
     \widehat{\pi(T)}(\p)=\widehat{q(2^{-n})}(\p)\quad (\p\in\Delta)\\
     & \Longleftrightarrow\,\,\,\, \pi(T)=q(2^{-n}).
\end{split}    
\end{equation*}
Therefore, $\pi(\psi\theta)^{-1}(q_n)=q(2^{-n})$ for each $n\in\N_0$. Since $\w{\tau}=\pi(\psi\theta)^{-1}$, we can rewrite \eqref{index_bfx} as $\Find(\lambda I-q(2^{-n}))=0$ for any $\lambda$ from a~bounded component of $\C\setminus\Omega_n$. Of course, since the Fredholm index is continuous and $\lambda I-q(2^{-n})$ is invertible for $\lambda$ sufficiently large, we conclude that the last equality holds true for every $\lambda\in\C\setminus\Omega_n$.
\end{proof}

The condition occuring in Proposition~\ref{P_kernel} is exactly the same as the one in the famous Brown--Douglas--Fillmore theorem which says that an~operator $T\in\BB(\Hh)$ is of the form `normal plus compact' if and only if it is essentially normal (i.e. $T^\ast T-TT^\ast\in\KK(\Hh)$) and $\mathrm{index}(\lambda I-T)=0$ for every $\lambda$ outside of the essential spectrum of $T$ (i.e. $\lambda\in\C\setminus\sigma(\pi T)$).

\begin{corollary}\label{split_C}
Suppose that $(Q(t))_{t\geq 0}\subset\BB(\Hh)$ is a~collection of normal operators which satisfy the semigroup property modulo the compacts, that is,
$$
Q(s+t)-Q(s)Q(t)\in\KK(\Hh)\quad (s,t\geq 0).
$$
Assume also that $(q(t))_{t\geq 0}=(\pi Q(t))_{t\geq 0}\subset\QQ(\Hh)$ is a~$C_0$-semigroup under any faithful $^\ast$-representation of $\QQ(\Hh)$. Then, there exists a~dyadic semigroup $(T(t))_{t\in\mathbb{D}}\subset\BB(\Hh)$ such that $Q(t)-T(t)\in\KK(\Hh)$ for each $t\in\mathbb{D}$, whenever
\begin{equation}\label{bdf_simple_cor}
\varprojlim{}^{(1)}\Ext(S\Omega_n)=0.
\end{equation}
\end{corollary}
\begin{proof}
By the assumption, each $q(t)$ has a~normal lift and therefore the condition of Proposition~\ref{P_kernel} is satisfied. Hence, the extension of $\Omega$ produced by Proposition~\ref{P_spectrum} belongs to $\mathrm{ker}\,P$. A~glance at Milnor's exact sequence now shows that \eqref{bdf_simple_cor} implies that the resulting extension is trivial. It remains to apply the lifting lemma (Lemma~\ref{lifting_L} without the continuity assertions).
\end{proof}

%%%%%%%%%%%%%%%%%%%%%%%%%%%%%%%%%%%%%%%%%%%%%%%%%%%%%%%%%%%%%%%%%%%%
%%%%%%%%%%%%%%%%%%%%%%%%%%%%%%%%%%%%%%%%%%%%%%%%%%%%%%%%%%%%%%%%%%%%
\section{Continuity of lifting}
\subsection{Calkin representation and the zero element}

%\textcolor{blue}{Here: Calkin representation, (...) infinite coin toss, examples concerning `Farah liftings', \textcolor{red}{\bf quasi-uniform convergence: }\cite[\S VI.6]{DS} and \cite[Cor.~VI.6.12]{DS}}

We will be now dealing with $C_0$-semigroups in $\QQ(\Hh)$ with respect to some concrete representations of $\QQ(\Hh)$ as subalgebras of $\BB(\mathbb{H})$ for a~Hilbert space $\mathbb{H}$. The original construction by Calkin \cite{calkin} goes as follows.

Let $\UU$ be any nonprincipal ultrafilter on $\N$ which we keep fixed for the rest of the paper. The limit along $\UU$ gives rise to a~positive functional $\mathrm{LIM}_\UU\in\ell_\infty^\ast$ which satisfies $\liminf_n a_n\leq\mathrm{LIM}_{n,\UU}a_n\leq\limsup_n a_n$ for every real valued $(a_n)\in\ell_\infty$. Let $\mathscr{W}$ be the collection of all weakly null sequences in $\Hh$ on which we consider an~equivalence relation $(x_n)\sim (y_n)$ defined by $\lim_n\n{x_n-y_n}=0$. We write $\mathscr{W}^{\usim}$ for the collection of all equivalence classes and $[(x_n)]_{\sim}$ for the equivalence class of a~sequence $(x_n)\in\mathscr{W}$. The formula
$$
\big\langle [(x_n)]_\sim, [(y_n)]_\sim\big\rangle=\mathop{\mathrm{LIM}}\limits_{n,\UU}(\ip{x_n}{y_n})_{n=1}^\infty
$$
gives rise to an inner product in $\mathscr{W}^\sim$. Notice that, by the Cauchy--Schwarz inequality,  the right-hand side does not depend on the choice of representatives. Also, if $(x_n)$ is not a~null sequence, then $\ip{[(x_n)]_\sim}{[(x_n)]_\sim}>0$. 

Let $\mathbb{H}$ be the completion of $\mathscr{W}^{\usim}$ under the norm $\n{[(x_n)]_\sim}=\mathrm{LIM}_{n,\UU}\n{x_n}$. Then, $\mathbb{H}$ is a~Hilbert space of density $\mathfrak{c}$ such that $\QQ(\Hh)$ can be faithfully represented on $\BB(\Hh)$ in the following way. For any $T\in\BB(\Hh)$, define $\Phi_0(T)$ to be the linear operator on $\mathscr{W}^{\usim}$ given by $\Phi_0(T)[(x_n)]_\sim=[(Tx_n)]_\sim$. Since $\n{\Phi_0T}\leq\n{T}$, there is a~unique extension of $\Phi_0(T)$ to a~bounded linear operator on $\mathbb{H}$ which we denote by $\Phi(T)$. Obviously, the map $T\mapsto\Phi(T)\in\BB(\mathbb{H})$ is a~$^\ast$-homomorphism and since $\mathrm{ker}\,\Phi=\KK(\Hh)$, we can define an~induced map
$$
\gamma\colon\QQ(\Hh)\to\BB(\mathbb{H}),\,\,\, \gamma (\pi(T))=\Phi(T)\quad (T\in\BB(\Hh)),
$$
which is a~faithful representation of $\QQ(\Hh)$. In what follows, we shall call $\gamma$ the {\it Calkin representation}. If we want to stress that it comes from the ultrafilter $\UU$, we call it the $\UU$-{\it Calkin representation}. A~thorough study of Calkin's representations was done by Reid in \cite{reid}, whereas other, more subtle representations, with type II$_\infty$ factor as the range, were constructed by Anderson and Bunce (see \cite{anderson} and \cite{AB}).

The next technical lemma refers to the well-known criterion for \SOT-continuity of an~operator semigroup $(T(t))_{t\geq 0}$ on a~Banach space $X$, which requires $T(t)$ to be uniformly bounded on an~interval $[0,\delta]$ and converge in norm pointwise on a~dense subset of $X$ (see \cite[Prop.~1.3]{EN}).

\begin{lemma}\label{cont1_L}
Let $X$ be an admissible compact metric space and let $\Gamma\in\Ext(X)$ induce a~dyadic semigoup $(q(t))_{t\in\mathbb{D}}\subset\QQ(\Hh)$. Let also $\gamma$ be a~faithful $^\ast$-representation of $\QQ(\Hh)$ on a~Hilbert space $\mathbb{H}$. Assume that for some dense set $D\subset\mathbb{H}$ we have 
$$
\lim_{t\in\mathbb{D},\,t\to 0}\gamma (q(t))\mathbf{x}=\mathbf{x}\quad\mbox{for each }\, \mathbf{x}\in D.
$$
Then $\sot_{t\in\mathbb{D},\, t\to 0}\gamma(q(t))=I_{\mathbb{H}}$ and there is a~$C_0$-semigroup $(T(t))_{t\geq 0}\subset\BB(\mathbb{H})$ which extends $(\gamma(q(t)))_{t\in\mathbb{D}}$.
\end{lemma}
\begin{proof}
For any sequence $(t_n)_{n=1}^\infty\subset\mathbb{D}$ with $t_n\to 0$, consider the compact set $K=\{0\}\cup\{t_n\colon n\in\N\}$ and notice that the map $K\ni t\mapsto \gamma (q(t))$, where $\gamma(q(0))=I_{\mathbb{H}}$, is norm bounded and such that $K\ni t\mapsto \gamma(q(t))\mathbf{x}$ is continuous on $K$ for $\mathbf{x}\in D$. Fix $\mathbf{y}\in\mathbb{H}$, $\e>0$, and pick $\mathbf{x}\in D$, $\delta>0$ and $M>0$ such that $\n{\mathbf{x}-\mathbf{y}}<\e$, $\n{\gamma(q(t))\mathbf{x}-\mathbf{x}}<\e$ for $t\in K$ with $t<\delta$, and $\n{\gamma(q(t))}\leq M$ for every $t\in K$. Then, by triangle inequality, $\n{\gamma(q(t))\mathbf{y}-\mathbf{y}}<(M+2)\e$ for each $t\in K$, $t<\delta$, which shows that $\sot_{n\to\infty}\gamma(q(t_n))=I_{\mathbb{H}}$ and hence $\sot_{t\in\mathbb{D},\,t\to 0}\gamma(q(t))=I_{\mathbb{H}}$.

Fix any $\mathbf{x}\in\mathbb{H}$ and notice that in view of Remark~\ref{rem_exp}, for any $s,t\in\mathbb{D}$, $s<t$, we have
$$
\n{\gamma(q(t))\mathbf{x}-\gamma(q(s))\mathbf{x}}\leq\n{q(s)}\!\cdot\!\n{\gamma(q(t-s))\mathbf{x}-\mathbf{x}}\leq e^{s\zeta}\n{\gamma(q(t-s))\mathbf{x}-\mathbf{x}}.
$$
Hence, the map $\mathbb{D}\ni t\mapsto \gamma(q(t))\mathbf{x}$ is uniformly continuous on $[0,T]\cap\mathbb{D}$ for any $T>0$. We denote its continuous extension to the whole of $[0,\infty)$ by $T(\cdot)\mathbf{x}$, so that for every $\mathbf{x}\in\mathbb{H}$ we have a~continuous function $[0,\infty)\ni t\mapsto T(t)\mathbf{x}$. Since for any $t\in [0,\infty)\setminus\mathbb{D}$ and $\mathbf{x}\in\mathbb{H}$, we have $T(t)\mathbf{x}=\lim_{u\in\mathbb{D},\,u\to t}\gamma(q(u))\mathbf{x}$, the map $\mathbf{x}\mapsto T(t)\mathbf{x}$ is a~bounded linear operator on $\mathbb{H}$. Therefore, $(T(t))_{t\geq 0}$ is a~$C_0$-semigroup in $\BB(\mathbb{H})$ which satisfies $T(t)=\gamma(q(t))$ for $t\in\mathbb{D}$. 
\end{proof}

\begin{remark}\label{cont1_R}
Of course, the same proof as above works also in a~simpler situation, where $(Q(t))_{t\in\mathbb{D}}\subset\BB(\Hh)$ is a~dyadic semigroup satisfying:
\begin{itemize}[leftmargin=24pt]
\setlength{\itemsep}{3pt}
\item $\sup_{t\in [0,s]\cap\mathbb{D}}\n{Q(t)}<\infty$ for every $s>0$;

\item $\sot_{t\in\mathbb{D},\,t\to 0}Q(t)=I.$
\end{itemize}
Then there is a $C_0$-semigroup $(Q(t))_{t\geq 0}\subset\BB(\Hh)$ which extends the given one.
\end{remark}

\begin{theorem}\label{C0_i_iii_T}
Let $X$ be an admissible compact metric space, $X=\varprojlim X_n$, \vspace*{-1pt}where each $X_n$ is given by \eqref{adm_Z_D} and $Z\subset\C$ is a~closed set with 
$$
\eta\coloneqq\inf_{z\in Z}\mathrm{Re}\,z\leq\sup_{z\in Z}\mathrm{Re}\,z\eqqcolon\zeta<\infty.
$$
\begin{enumerate}[label={\rm (\roman*)}, leftmargin=24pt, labelindent=0pt, wide]
\setlength{\itemindent}{4pt}
\setlength{\itemsep}{9pt}
\item In order that $\Theta\in\Ext_{C_0,\gamma}(X)$ for the fixed $\UU\!$-Calkin representation, it is necessary that $-\infty<\eta\leq\zeta<\infty$ and $\lim_{n\to\infty}\abs{1-\pi_n(\xi)}=0$ for every $\xi\in X$.

\item In order that $\Theta\in\Ext_{C_0,\gamma}(X)$ for all $\mathscr{V}\!$-Calkin representations, with any $\mathscr{V}\in\beta\N\setminus\N$, it suffices that $\abs{1-\pi_n(\xi)}=O(2^{-n})$ uniformly for $\xi\in X$.
%$$
%\big(\vert 1-\pi_{n}(\xi_k)\vert\big)_{k=1}^\infty\yrightarrow[\,n\to\infty\,][2pt]{\,\,w\,\,}[0pt] 0\quad\mbox{in }\,\,\ell_\infty.
%$$

\item We have $\Theta\in\Ext_{C_0,\gamma}(X)$ for all $\mathscr{V}\!$-Calkin representations, with any $\mathscr{V}\in\beta\N\setminus\N$, if and only if for one (equivalently: for every) dense subset $\{\xi_k\colon k\in\N\}\subset X$, and any sequences $(L_n)_{n=1}^\infty$, $(S_n)_{n=1}^\infty$ of positive integers satisfying $\lim_{n\to\infty} 2^{-L_n}S_n=0$, we have
$$
\big(\vert 1-\pi_{L_n}(\xi_k)^{S_n}\vert\big)_{k=1}^\infty\yrightarrow[\,n\to\infty\,][2pt]{\,\,w\,\,}[0pt] 0\quad\mbox{in }\,\,\ell_\infty.
$$
\end{enumerate}
\end{theorem}
\begin{proof}

As we know, the zero element $\Theta$ is determined by a~Busby invariant $\tau=\pi\sigma$, where $\sigma\colon C(X)\to\BB(\Hh)$ is a~$^\ast$-homomorphism given by $\sigma(f)=\bigoplus_{k=1}^\infty f(\xi_k)I_{\Hh_k}$, where $\{\xi_k\colon k\in\N\}$ is an~arbitrary dense subset of $X$ and $\Hh\cong\bigoplus_{k=1}^\infty\Hh_k$ is a~decomposition into infinite-dimensional subspaces. Recall also that the dyadic semigroup $(q(t))_{t\in\mathbb{D}}$ generated by $\Theta$ is the extension of the semigroup $(q(2^{-n}))_{n\geq 0}$ given by $q(2^{-n})=\tau(\pi_n)$. 

Consider any $t\in\mathbb{D}$ of the form $t=2^{-\ell_1}+\ldots+2^{-\ell_k}$, where $1\leq \ell_1<\ldots<\ell_k$ are integers. Then
$$
q(t)=q(2^{-\ell_1})\cdot\ldots\cdot q(2^{-\ell_k})=\tau(\pi_{\ell_1}\cdot\ldots\cdot\pi_{\ell_k})
$$
and since
\begin{equation}\label{piii}
\pi_{\ell_i}^{2^{\ell_i-\ell_1}}=\pi_{\ell_1}\quad\,\mbox{for each }\,\,1\leq i\leq k,
\end{equation}
we can write $q(t)=\tau(\pi_{\ell_k}^s)$, where $s=\sum_{i=1}^k 2^{\ell_i-\ell_1}$. Now, let $(t_n)_{n=1}^\infty\subset\mathbb{D}$ be any sequence converging to zero, 
$$
t_n=2^{-\ell_{n,1}}+\ldots+2^{-\ell_{n,k_n}}\quad (n\in\N)
$$
with integers $1\leq \ell_{n,1}<\ldots<\ell_{n,k_n}$, and define: 
\begin{equation}\label{FLS_def}
F_n=\ell_{n,1},\,\,\,\, L_n=\ell_{n,k_n},\,\,\,\, S_n=\sum_{i=1}^{k_n} 2^{\ell_{n,i}-F_n}\quad\, (n\in\N).
\end{equation}
Note that $t_n\to 0$ if and only if $F_n\to \infty$. Observe also that $S_n=1+\sum_{m\in M}2^m$ for some $M\subseteq\{1,2,\ldots,L_n-F_n\}$, hence $1\leq S_n\leq 2^{L_n-F_n+1}-1$ from which it follows that $2^{-L_n}S_n\to 0$ as $n\to\infty$. Applying \eqref{piii} to $t_n$ we obtain $q(t_n)=\tau(\pi_{L_n}^{S_n})$ for each $n\in\N$. Therefore, for any $\mathbf{x}=[(x_j)]_\sim\in\mathscr{W}^{\usim}$ and $n\in\N$, we have
\begin{equation*}
    \begin{split}
        \mathbf{x}-\gamma(q(t_n))\mathbf{x} &=\mathbf{x}-\gamma\Bigg\{\bigoplus_{k=1}^\infty\pi_{L_n}(\xi_k)^{S_n} I_{\Hh_k}\Bigg\}\mathbf{x}\\[2ex]
        &=\Bigg[\Big(\bigoplus_{k=1}^\infty\big(1-\pi_{L_n}(\xi_k)^{S_n}\big) I_{\Hh_k}x_j\Big)_{\!\!j=1}^{\!\!\infty}\Bigg]_\sim\\[2ex]
        &=\Bigg[\Big(\sum_{k=1}^\infty\big(1-\pi_{L_n}(\xi_k)^{S_n}\big)P_kx_j\Big)_{\!\!j=1}^{\!\!\infty}\Bigg]_\sim,
    \end{split}
\end{equation*}
where $P_k$ is the orthogonal projection onto $\Hh_k$. Set $\omega_{j,k}=\n{P_kx_j}^2$ and notice that if $(x_j)$ runs through the collection of all weakly null sequences in $\Hh$, then $\oo{\mathbf{\omega}}=(\omega_{j,k})_{j,k\in\N}$ is a~matrix of nonnegative entries such that $\sup_j\sum_{k=1}^\infty \omega_{j,k}<\infty$, as $\sum_{k=1}^\infty\omega_{j,k}=\sum_{k=1}^\infty\n{P_kx_j}^2=\n{x_j}^2$ for every $j\in\N$. Conversely, since all the $\Hh_k$ are infinite-dimensional, every such matrix corresponds to some weakly null sequence $(x_j)\subset\Hh$. Therefore, the condition
\begin{equation}\label{condition_x}
    \lim_{n\to\infty}\|\mathbf{x}-\gamma(q(t_n))\mathbf{x}\|_{\mathbb{H}}=0\quad\mbox{for every }\,\,\mathbf{x}\in\mathscr{W}^{\usim}
\end{equation}
is equivalent to saying that
\begin{equation}\label{condition_LIM}
    \lim_{n\to\infty}\mathop{\mathrm{LIM}}_{j,\UU}\,\sum_{k=1}^\infty\omega_{j,k}\vert 1-\pi_{L_n}(\xi_k)^{S_n}\vert^2=0
\end{equation}
for every matrix $\oo{\omega}$ as described above. 

For any $n\in\N$, the fixed dense set $\{\xi_k\colon k\in\N\}\subset X$ and $(t_n)\subset\mathbb{D}$ with $t_n\to 0$, define 
\begin{equation}\label{xi_def}
\boldsymbol{\xi}^{(n)}=\big(\vert 1-\pi_{L_n}(\xi_k)^{S_n}\vert\big)_{k=1}^\infty\in\ell_\infty;
\end{equation}
notice that this is indeed a~bounded sequence, as
\begin{equation}\label{est_l}
    \begin{split}
        \n{\boldsymbol{\xi}^{(n)}}_\infty &=\sup_{k\in\N}\vert 1-\pi_{L_n}(\xi_k)^{S_n}\vert\\
        &=\sup\big\{\vert 1-z^{S_n}\vert\colon z\in\exp(2^{-L_n}Z)\big\}\leq 1+\exp(2^{-L_n}S_n\zeta).
    \end{split}
\end{equation}
Putting $\omega_{j,k}=c_k$ for all $j,k\in\N$ in condition \eqref{condition_LIM}, where $c_k\geq 0$ and $\sum_{k=1}^\infty c_k=1$, \vspace*{-2pt}we obtain $\sum_{k=1}^\infty c_k\boldsymbol{\xi}^{(n)}_k\to 0$ as $n\to\infty$ from which it follows that $(\boldsymbol{\xi}^{(n)})_{n=1}^\infty$ converges to zero in the weak$^\ast$ topology on $\ell_\infty$. Note that if $\eta=-\infty$, we would have $0\in X_n$ for each $n\geq 0$, hence $\boldsymbol{0}=(0,0,\ldots)\in X$ and taking e.g. $\xi_1=\boldsymbol{0}$ we see that the $(\boldsymbol{\xi}^{(n)})_{n=1}^\infty$ would not converge weak$^\ast$ to zero, hence we must have $\eta>-\infty$. Finally, for any $\zeta<\infty$, \eqref{est_l} shows that $\sup_{n\in\N}\n{\boldsymbol{\xi}^{(n)}}_\infty$ is finite, thus the weak$^\ast$ convergence of this sequence is equivalent to coordinatewise convergence. This finishes the proof of assertion (i).

\vspace*{2mm}
Next, we shall prove assertion (iii). According to Lemma~\ref{cont1_L}, we have $\Theta\in\Ext_{c_0,\gamma}(X)$ under a~$\mathscr{V}\!$-Calkin representation $\gamma$ if and only if condition \eqref{condition_x} is valid for every sequence $(t_n)_{n=1}\subset\mathbb{D}$ with $t_n\to 0$. This is in turn equivalent to condition \eqref{condition_LIM} being valid for:
\begin{itemize}[leftmargin=24pt]
\setlength{\itemsep}{2pt}
\item every $\mathscr{V}\in\beta\N\setminus\N$ in place of $\UU$, 

\item every matrix $\oo{\omega}=(\omega_{j,k})_{j,k\in\N}$ with nonnegative entries such that $\sup_j\sum_{k=1}^\infty \omega_{j,k}<\infty$,

\item and every dense set $\{\xi_k\colon k\in\N\}\subset X$.
\end{itemize}
Of course, it does not depend on the particular choice of $\{\xi_k\colon k\in\N\}$, as we know that the property of inducing a~\SOT-continuous semigroup is preserved by taking an~equivalent extension (see Remark~\ref{C0_correct_R}). Moreover, the first two quantifiers can be replaced by saying simply that 
\begin{equation}\label{for_every_V}
\lim_{n\to\infty}\mathop{\mathrm{LIM}}_{k,\mathscr{V}}\,\boldsymbol{\xi}^{(n)}_k=0\quad\mbox{for every }\,\,\mathscr{V}\in\beta\N.
\end{equation}
To see this, note that for every choice of $\mathscr{V}$ and $\oo{\omega}$, the map $\ell_\infty \ni\boldsymbol{\xi}\mapsto\mathop{\mathrm{LIM}}_{j,\mathscr{V}}\sum_{k=1}^\infty \omega_{j,k}\boldsymbol{\xi}_k$ is an element of $\ell_\infty^\ast$, therefore \eqref{condition_LIM} assumed for all choices of $\mathscr{V}$ and $\oo{\omega}$ is (formally) weaker than saying that $(\boldsymbol{\xi}^{(n)})_{n=1}^\infty$ converges weakly to zero (here, we can obviously omit the square). On the other hand, taking $\omega_{j,k}=1$ for any $j=k\in\N$ and $\omega_{j,k}=0$ otherwise, we see that \eqref{condition_LIM} implies \eqref{for_every_V} (we can include $\mathscr{V}\in\N$ in that condition, as we have already observed that \eqref{condition_LIM} implies the coordinatewise convergence). Therefore, if $\Theta\in\Ext_{C_0,\gamma}(X)$ for every $\mathscr{V}\!$-Calkin representation, then, as we have seen in part (i), the sequence $(\boldsymbol{\xi}^{(n)})_{n=1}^\infty$ is uniformly bounded and, by \eqref{for_every_V}, we have $\lim_{n\to\infty}\boldsymbol{\xi}^{(n)}(\mathscr{V})$ for every $\mathscr{V}\in\beta\N$. As it is well-known, boundedness and pointwise convergence of sequences in $C(K)$-spaces implies weak convergence (see, e.g., \cite[Cor.~3.138]{FHHMZ}), therefore $\boldsymbol{\xi}^{(n)}\xrightarrow[]{w}0$. 

Recall that the parameters $L_n$ and $S_n$ were defined in terms of $t_n\in\mathbb{D}$ and since $(t_n)_{n=1}^\infty$ was an~arbitrary sequence of positive dyadic numbers converging to zero, it is easily seen that the above condition must be valid for $(\boldsymbol{\xi}^{(n)})_{n=1}^\infty$ defined by \eqref{xi_def} with arbitrary sequences $(L_{n})_{n=1}^\infty, (S_n)_{n=1}^\infty\subset\N$ satisfying $2^{-L_n}S_n\to 0$. Indeed, if we suppose that this is not true for some choice of such sequences, then, by passing to subsequences, we could assume that $2^{-L_n}S_n<2^{-n}$ ($n\in\N$) are such that condition \eqref{condition_LIM} fails to hold for every $\oo{\omega}$ and every $\mathscr{V}\in\beta\N$ in place of $\UU$. But then, for each $n\in\N$, we define $t_n\in\mathbb{D}$ by $t_n=2^{-n}+\e_1 2^{-(n+1)}+\ldots+\e_{L_n-n}2^{-L_n}$ with $\e_i\in\{0,1\}$ arranged so that $S_n=1+\sum_{i=1}^{L_n-n}\e_i2^{i}$. This is, of course, possible since $1\leq S_n<2^{L_n-n}$. In this way we obtain a~sequence $(t_n)_{n=1}^\infty\subset\mathbb{D}$ with $t_n\to 0$ and such that $F_n=n$, $L_n$ and $S_n$ ($n\in\N$) correspond to $t_n$ via definition \eqref{FLS_def}. But for all such sequences condition \eqref{condition_LIM} is true, a~contradiction.

Conversely, if for all sequences $(L_{n})_{n=1}^\infty, (S_n)_{n=1}^\infty\subset\N$ satisfying $2^{-L_n}S_n\to 0$ and any dense set $\{\xi_k\colon k\in\N\}\subset X$, we have $\boldsymbol{\xi}^{(n)}\xrightarrow[]{w}0$, then \eqref{condition_x} holds true for every $(t_n)_{n=1}^\infty\subset\mathbb{D}$ and $\gamma$ being any $\mathscr{V}\!$-Calkin representation, which, as we know from Lemma~\ref{cont1_L}, guarantees that $\Theta\in\Ext_{C_0,\gamma}(X)$. This finishes the proof of assertion (iii).

\vspace*{2mm}

In order to show (ii), fix any sequences $(L_n)_{n=1}^\infty, (S_n)_{n=1}^\infty\subset\N$ with $2^{-L_n}S_n\to 0$. For any $\xi\in X$, we have
\begin{equation}\label{simple_product}
\vert 1-\pi_{L_n}(\xi)^{S_n}\vert=\vert 1-\pi_{L_n}(\xi)\vert\cdot\vert 1+\pi_{L_n}(\xi)+\pi_{L_n}(\xi)^2+\ldots+\pi_{L_n}(\xi)^{S_n-1}\vert.
\end{equation}
For any sequence $(\e_n)_{n=1}^\infty$ of arbitrarily small positive numbers, and $n\in\N$, we may pick $z_n\in Z$ such that $\vert\pi_{L_n}(\xi)-\exp\small(2^{-L_n}z_n\small)\vert<\e_n$. By subtracting suitable integer multiples of $2\pi\mathrm{i}$ from each $z_n$, we may also assume that $2^{-L_n}z_n\to 0$ (though $z_n$ may now not belong to $Z$). Next, by increasing $\e_n$'s, we see that the second factor in \eqref{simple_product} can be arbitrarily close to 
$$
\Bigg|\frac{1-\exp(2^{-L_n}S_nz_n)}{1-\exp(2^{-L_n}z_n)}\Bigg|=\Bigg|\frac{1-\exp(2^{-L_n}S_nz_n)}{2^{-L_n}S_nz_n}\Bigg|\cdot\Bigg|\frac{2^{-L_n}z_n}{1-\exp(2^{-L_n}z_n)}\Bigg|\cdot S_n=O(S_n)
$$
(we may, of course, consider only those $n$ for which all the denominators are nonzero). By the assumption of (ii), the first factor in \eqref{simple_product} is $O(2^{-L_n})$. Hence,
$$
\vert 1-\pi_{L_n}(\xi)^{S_n}\vert=O(2^{-L_n}S_n)=o(1)
$$
uniformly for $\xi\in X$, which obviously implies that the condition of weak convergence in assertion (iii) holds true.
\end{proof}

\begin{example}
Observe that the sufficient condition in assertion (ii) above is compatible with the well-known asymptoticity $\lim_{n\to\infty}n(\sqrt[n]{a}-1)=\log\,a$ for any $a>0$. For instance, let $Z=\{s+\ii\hspace*{1pt} t\colon (s,t)\in [-1,0]\times [-\pi,\pi]\}$. Then $X_n=\exp\small(2^{-n}Z\small)$ is the set of complex numbers $z$ with $e^{-2^{-n}}\leq\abs{z}\leq 1$ and $-2^{-n}\pi\leq \mathrm{Arg}\,z\leq 2^{-n}\pi$. Hence, for any $\xi\in X$ we have
$$
\vert 1-\pi_n(\xi)\vert=\max\big\{\vert 1-e^{2^{-n}\pi\ii}\vert,\,\vert 1-e^{-2^{-n}(1+\pi\ii)}\vert\big\}
$$
which for $n$ large enough (so that $2^{-n}$ is smaller than the smallest positive root of the equation $1+e^{-x}=2\cos\pi x$) equals 
$$
|1-e^{-2^{-n}(1+\pi\ii)}|=\sqrt{(1-e^{-2^{-n}})^2+2e^{-2^{-n}}(1-\cos 2^{-n}\pi)}=O(2^{-n}).
$$
Obviously, it follows from this example that the sufficient condition from assertion (ii) is satisfied for $Z$ being an~arbitrary compact subset of the complex plane.
\end{example}

\begin{example}\label{T_not_cont_E}
As in Example~\ref{ex_imaginary}, let $Z=\ii\R$ and $X_n=\exp(2^{-n}Z)=S^1$ for $n=0,1,2,\ldots$ Notice that each element $\xi$ of $\Sigma_2=\varprojlim\{S^1,z^2\}$ is uniquely determined \vspace*{-1pt}by a~choice of $\pi_0(\xi)\in S^1$ and a~sequence $(\e_n)_{n=1}^\infty\in\{0,1\}^\N$, in the sense that for each $n\in\N$, if $\pi_{n-1}(\xi)=e^{it}$, then $\pi_{n}(\xi)=e^{\ii(t+2\e_n\pi)/2}$. Obviously, for almost all choices of $(\e_i)_{i=1}^\infty$, the sequence $(\pi_n(\xi))_{n=0}^\infty$ diverges, thus the necessary condition in Theorem~\ref{C0_i_iii_T}(i) is not satisfied. Hence, we have $\Theta\not\in\Ext_{C_0,\gamma}(\Sigma_2)$.
\end{example}

It is known that weak convergence in $B(\Sigma)$-spaces, that is, Banach spaces of bounded measurable functions, where $\Sigma$ is a~$\sigma$-algebra on some set $A$, can be characterized in terms of quasi-uniform convergence; see \cite[\S VI.6]{DS}. Namely, a~sequence $(f_n)_{n=1}^\infty\subset B(\Sigma)$ is weakly null if and only if it is bounded, pointwise convergent to zero and every its subsequence $(f_{n_k})_{k=1}^\infty$ satisfies the following condition: for each $\e>0$ and $k_0\in\N$ there are indices $k_0\leq k_1<\ldots<k_j$ such that $\min_{1\leq i\leq j}\vert f_{n_{k_i}}(a)\vert<\e$ for every $a\in A$ (see \cite[Thm.~VI.6.31]{DS}). We can therefore reformulate the characterization in assertion (iii) into perhaps a~more directly applicable form.

\begin{corollary}
Let $X$ be an admissible compact metric space, $X=\varprojlim X_n$, \vspace*{-1pt}where each $X_n$ is given by \eqref{adm_Z_D} and $Z$ is a~closed subset of $\{z\in\C\colon \mathrm{Re}\,z\leq \zeta<\infty\}$. Then $\Theta\in\Ext_{C_0,\gamma}(X)$ for all $\mathscr{V}\!$-Calkin representations, with any $\mathscr{V}\in\beta\N\setminus\N$, if and only if the following two conditions hold true:
\begin{itemize}[leftmargin=24pt]
\setlength{\itemsep}{3pt}
\item $\lim_{n\to\infty}\vert 1-\pi_n(\xi)\vert=0$ for every $\xi\in X$;

\item for any sequences $(L_n)_{n=1}^\infty$, $(S_n)_{n=1}^\infty$ of positive integers satisfying $\lim_{n\to\infty} 2^{-L_n}S_n=0$, and $\e>0$ and $n_0\in\N$, there exist indices $n_0\leq n_1<\ldots<n_k$ such that
\begin{equation}\label{min_e}
\min\limits_{1\leq i\leq k}\vert 1-\pi_{L_{n_i}}(\xi)^{S_{n_i}}\vert<\e\quad\mbox{for every }\,\, \xi\in X.
\end{equation}
\end{itemize}
\end{corollary}

\begin{proof}
The necessity follows directly from Theorem~\ref{C0_i_iii_T}(iii) which yields \eqref{min_e} for $\xi_k$ from a~dense subset of $X$ in place of $\xi$, but this, of course, implies that the estimate is valid for every $\xi\in X$. For sufficiency, recall that by \eqref{est_l}, the assumption $\zeta=\sup_{z\in Z}\mathrm{Re}\,z<\infty$ implies that for any dense set $\{\xi_k\colon k\in\N\}\subset X$ and any $(L_n)_{n=1}^\infty$, $(S_n)_{n=1}^\infty$ \vspace*{-1pt}as above, the sequence $(\boldsymbol{\xi}^{(n)})_{n=1}^\infty\subset\ell_\infty$ given by \eqref{xi_def} is bounded. Also, since we allow $(L_n)_{n=1}^\infty$ and $(S_n)_{n=1}^\infty$ to be arbitrary sequences satisfying $2^{-L_n}S_n\to 0$, every subsequence of $(\boldsymbol{\xi}^{(n)})_{n=1}^\infty$ has the property of quasi-uniform convergence. Therefore, it is weakly convergent to zero due to the above quoted characterization \cite[Thm.~VI.6.31]{DS}. The~result now follows from Theorem~\ref{C0_i_iii_T}(iii).
\end{proof}

\subsection{Lifting to a $C_0$-semigroup}
Now, we show that (under Calkin's representation) any semigroup lifting preserves $\SOT$-coninuity, provided that the underlying inverse limit is a~perfect compact metric space.
\begin{theorem}\label{no_isolated_T}
Let $(q(t))_{t\geq 0}\subset\QQ(\Hh)$ be a~$C_0$-semigroup of normal operators under Calkin's representation\footnote{More generally, we may assume that $(q(t))_{t\in\mathbb{D}}\subset\QQ(\Hh)$ is a~$C_0$-semigroup such that its $C_0$-extension $(T(t))_{t\geq 0}\subset\BB(\mathbb{H})$ consists only of normal operators; see the analogous remark before Proposition~\ref{P_spectrum}.}, which corresponds to the zero extension $\Theta$ of the inverse limit $\Delta$ as in Proposition~\ref{P_spectrum}. If $\Delta$ has no isolated points, then $(q(t))_{t\in\mathbb{D}}$ admits a~lift to a~$C_0$-semigroup $(Q(t))_{t\in\mathbb{D}}\subset \BB(\Hh)$.
\end{theorem}
\begin{proof}
For readability we divide the proof into several parts.

\vspace*{2mm}\noindent
\underline{\it Part 1.} Using the notation from Proposition~\ref{P_spectrum}, consider a~splitting unital $^\ast$-homomorphism $\rho\colon C(\Delta)\to\mathcal{E}$ so that in the diagram below
\begin{equation}\label{splittt}
\begin{tikzcd}
0 \arrow[r] & \KK(\Hh) \arrow[r, "\iota"] & \mathcal{E} \arrow[r, "\theta"] & C(\Delta) \arrow[r] \arrow[l, dotted, bend left=50, pos=0.48, "\rho", xshift = -0.8ex] & 0
\end{tikzcd}
\end{equation}
we have $\theta\rho=\mathrm{id}_{C(\Delta)}$. Let $\Aa_0=\mathrm{C}^\ast(\{q(2^{-n})\}_{n\geq 0},1_{\QQ(\Hh)})$. Then $\mathbb{A}=\gamma(\mathcal{A}_0)$ is a~separable commutative \cs-subalgebra of $\BB(\mathbb{H})$; we denote by $\mathsf{E}$ be its spectral measure. According to the spectral theorem, the inverse of the Gelfand transform of $\mathbb{A}$, which we call $\Psi$, extends to an~isometric $^\ast$-isomorphism of $L^\infty(\mathsf{E})$ onto a~\cs-subalgebra $\mathbb{B}\subset\BB(\mathbb{H})$.
$$
\begin{tikzcd}
\BB(\mathbb{H})\supset\mathbb{A}=\gamma(\mathcal{A}_0) \arrow[r, "\widehat{\gamma^{-1}(\cdot)}"]  & [12pt] C(\Delta) \arrow[l, bend left=35, pos=0.5, "\Psi", xshift = 3ex]
\end{tikzcd}
$$
For $\mathbf{x},\mathbf{y}\in\mathbb{H}$, let $\mathsf{E}_{x,y}$ be the corresponding scalar measure induced by $\mathsf{E}$, so that we have
\begin{equation}\label{Teqq}
\ip{T\mathbf{x}}{\mathbf{y}}=\int_\Delta \widehat{\gamma^{-1}(T)}\,\dd\mathsf{E}_{\mathbf{x},\mathbf{y}}\quad (T\in\mathbb{A}).
\end{equation}

Define $f_n=\widehat{q(2^{-n})}$ for $n=0,1,2,\ldots$ and note that 
$$
\sot_{n\to\infty}\Psi(f_n)=\sot_{n\to\infty}\gamma(q(2^{-n}))=I,
$$
as $(q(t))_{t\geq 0}$ is a~$C_0$-semigroup. 

\vspace*{2mm}\noindent
{\it Claim 1. }For every $n=0,1,2,\ldots$, we have $\Psi(f_n)=\gamma(\pi(\rho (f_n)))$.

\vspace*{2mm}
To see this, recall that $\theta\rho=\mathrm{id}_{C(\Delta)}$, thus for each $n=0,1,2,\ldots$, we have $\widehat{\pi\rho(f_n)}=f_n$ and hence $\pi\rho(f_n)=q(2^{-n})$ as the Gelfand transform is injective. Composing with $\gamma$ we obtain our claim.

\vspace*{2mm}\noindent
{\it Claim 2. }For every $\mathbf{x}\in\mathbb{H}$, we have $f_n\xrightarrow[]{\phantom{xx}}\mathbf{1}$ in $L^2(\Delta,\mathsf{E}_{\mathbf{x},\mathbf{x}})$.

\vspace*{2mm}
By Claim~1, $\sot_{n\to\infty}\Psi(f_n)=I$ and hence for any fixed $\mathbf{x}\in\mathbb{H}$, we have
\begin{equation}\label{c011}
\lim_{n\to\infty}\n{\Psi(f_n-\mathbf{1})\mathbf{x}}=0.
\end{equation}
On the other hand,
\begin{equation*}
    \begin{split}
        \n{\Psi(f_n-\mathbf{1})\mathbf{x}}^2 &=\ip{\Psi(f_n-\mathbf{1})\mathbf{x}}{\Psi(f_n-\mathbf{1})\mathbf{x}}\\
        &=\ip{\Psi(\abs{f_n-\mathbf{1}}^2)\mathbf{x}}{\mathbf{x}}=\int_\Delta\abs{f_n-\mathbf{1}}^2\,\dd\mathsf{E}_{\mathbf{x},\mathbf{x}}
    \end{split}
\end{equation*}
which, jointly with \eqref{c011}, proves our claim.

\vspace*{2mm}\noindent
\underline{\it Part 2.} Now, coming back to the splitting sequence \eqref{splittt}, let $\mathcal{E}_0=\rho(C(\Delta))$. As $\rho$ is a~representation of $C(\Delta)$ on the separable Hilbert space $\Hh$, there exists a~regular Borel measure $\mu$ on $\mathfrak{M}\coloneqq\mathcal{E}_0^{\prime\prime}\subset\BB(\Hh)$ such that $\mathfrak{M}$ is $^\ast$-isomorphic to $L^\infty(\mu)$ via a~homeomorphism between $\mathfrak{M}$ with the weak operator topology and $L^\infty(\mu)$ with the weak$^\ast$ topology (see \cite[Thm.~II.2.5]{davidson}). Moreover, $\mu$ is obtained by the~formula $\mu=\sum_{n=1}^\infty 2^{-n}\mu_n$ corresponding to a~decomposition $\rho=\bigoplus_n\rho_n$ into at most countably many cyclic subrepresentations of $\rho$; each $\mu_n$ is the representing measure for the functional $C(\Delta)\ni g\mapsto\ip{\rho(g)x_n}{x_n}$, where $x_n\in\Hh$ is a~fixed unit cyclic vector of $\rho_n$.

Consider the following two statements:
\begin{itemize}[leftmargin=24pt]
\setlength{\itemsep}{3pt}
\item[($\star$)] For every $x\in\Hh$, there is $\mathbf{x}\in\mathbb{H}$ such that the representing measure for the functional $C(\Delta)\ni g\xmapsto[]{} \ip{\rho(g)x}{x}_{\Hh}$ is absolutely continuous with respect to the representing measure for the functional $C(\Delta)\ni g\xmapsto[]{} \ip{\gamma\circ\pi\circ\rho(g)\mathbf{x}}{\mathbf{x}}_{\mathbb{H}}$.

\item[($\star\star$)] For every $x\in\Hh$, there is $\mathbf{x}\in\mathbb{H}$ such that 
$$
\ip{\rho(g)x}{x}_\Hh\leq\ip{\gamma\circ\pi\circ\rho(g)\mathbf{x}}{\mathbf{x}}_{\mathbb{H}}
$$
for any function $g\in C(\Delta)$, $g\geq 0$.
\end{itemize}
Suppose, for a moment, that condition ($\star$) is satisfied and for each cyclic vector $x_n$, corresponding to the above-mentioned decomposition $\rho=\bigoplus_n\rho_n$, pick a~unit vector $\mathbf{x}_n\in\mathbb{H}$ such that 
\begin{equation}\label{mulambdaeqq}
\mu_n\ac\lambda_n,\quad\mbox{where }\,\,\,\,\,\big\langle\gamma\circ\pi\circ\rho(g)\mathbf{x}_n,\mathbf{x}_n\big\rangle_{\mathbb{H}}=\int_\Delta g\,\dd\lambda_n\,\,\,\, (g\in C(\Delta)).
\end{equation}
Then, obviously,
\begin{equation}\label{mulambda}
    \mu\ac\lambda\coloneqq \sum_n 2^{-n}\lambda_n
\end{equation}
(as before, the sum has finitely or countably many terms).

\vspace*{2mm}\noindent
{\it Claim 3. }$f_n\xrightarrow[]{\,w^\ast\,}\mathbf{1}$ in $L^\infty(\lambda)$ and $f_n\xrightarrow[]{\phantom{xx}}\mathbf{1}$ in $L^2(\lambda)$.

\vspace*{2mm}
To see this, first note that $M\coloneqq\sup_n\n{f_n-\mathbf{1}}_\infty<\infty$ which follows from the fact that $\n{f_n}_\infty=\n{q(2^{-n})}=\n{\exp(2^{-n}A)}\leq\max\{1,\exp\small(\sup_{z\in\sigma(A)}\mathrm{Re}\, z\small)\}$, where $A$ is the generator of $(q(t))_{t\geq 0}$. Now, fix any $h\in L^1(\lambda)$, that is, 
$$
\sum_n 2^{-n}\int_\Delta\abs{h}\,\dd\mathsf{E}_{\mathbf{x}_n,\mathbf{x}_n}<\infty.
$$
(Notice that in view of \eqref{Teqq} and \eqref{mulambdaeqq}, each $\lambda_n$ is in fact $\mathsf{E}_{\mathbf{x}_n,\mathbf{x}_n}$, so that $\lambda=\sum_n 2^{-n}\mathsf{E}_{\mathbf{x}_n,\mathbf{x}_n}$.) For any $\e>0$, pick $N\in\N$ so large that $\sum_{j>N}2^{-j}\n{h}_{L^1(\lambda_j)}<\e/M$ and then, using Claim~2, pick $n_0\in\N$ such that $\abs{\int_\Delta (f_n-\mathbf{1})h\,\dd\lambda_{j}}<\e$ for all $n\geq n_0$ and $1\leq j\leq N$. Then, we have
\begin{equation*}
    \begin{split}
        \Bigg|\int_\Delta (f_n-\mathbf{1})h\,\dd\lambda\,\Bigg| &=\Bigg|\sum_{j=1}^\infty 2^{-j}\int_\Delta (f_n-\mathbf{1})h\,\dd\lambda_j\Bigg|\\
        &\leq \sum_{j=1}^N 2^{-j}\Bigg|\int_\Delta (f_n-\mathbf{1})h\,\dd\lambda_j\Bigg|+M\!\sum_{j=N+1}^\infty 2^{-j}\n{h}_{L^1(\lambda_j)}<2\e
    \end{split}
\end{equation*}
which proves the first part of our claim. The second one is proved by a~similar calculation.
%For any $\e>0$, pick $N\in\N$ so large that $\sum_{j>N}2^{-j}<\e/M^2$ and then, using Claim~2, pick $n_0\in\N$ such that $\int_\Delta \vert f_n-\mathbf{1}\vert^2\,\dd\lambda_{j}<\e$ for all $n\geq n_0$ and $1\leq j\leq N$. Then, for $n\geq n_0$, we have
%\begin{equation*}
 %   \begin{split}
  %      \int_\Delta \vert f_n-\mathbf{1}\vert^2\,\dd\lambda &=\sum_{j=1}^\infty 2^{-j}\int_\Delta \vert f_n-\mathbf{1}\vert^2\,\dd\lambda_j\\
   %     &\leq \sum_{j=1}^N 2^{-j}\int_\Delta \vert f_n-\mathbf{1}\vert^2\,\dd\lambda_j+M^2\!\sum_{j=N+1}^\infty 2^{-j}<2\e
    %\end{split}
%\end{equation*}
%which proves our claim.

\vspace*{2mm}
Consequently, once \eqref{mulambda} has been proved, Claim~3 implies that we also have $f_n\xrightarrow[]{\,w\ast\,}\mathbf{1}$ in $L^\infty(\mu)$, as well as $f_n\xrightarrow[]{}\mathbf{1}$ in $L^2(\mu)$. Indeed, if $h=\dd\mu/\dd\lambda\in L^1(\lambda)$ is the Radon--Nikodym derivative, then for any bounded Borel function $g\colon\Delta\to\C$, we have $gh\in L^1(\lambda)$ and by Claim~3,
$$
\int_\Delta (f_n-\mathbf{1})g\,\dd\mu=\int_\Delta (f_n-\mathbf{1})gh\,\dd\lambda\xrightarrow[\,n\to\infty\,]{}0.
$$
Using the fact that every bounded sequence in a~dual Banach space $X^\ast$ which is pointwise convergent on a~dense subset of $X$ must be weak$^\ast$ convergent, jointly with the fact that bounded Borel functions are dense in $L^1(\mu)$, we obtain the first assertion. For showing the $L^2$-convergence, we note that a~simple argument, similar to the one in the proof of Claim~3, yields 
$$
\int_\Delta \vert f_n-\mathbf{1}\vert^2\,\dd\mu=\int_\Delta \vert f_n-\mathbf{1}\vert^2 h\,\dd\lambda\xrightarrow[\,n\to\infty\,]{}0.
$$

According to the remarks from the first paragraph of this part of the proof, we conclude that condition ($\star$) implies that $(\rho(f_n))_{n=0}^\infty$ converges in the weak operator topology to the identity $I\in\BB(\Hh)$. However, the fact that $f_n\to\mathbf{1}$ in $L^2(\mu)$ implies that the corresponding multiplication operators converge in the strong operator topology (cf. \cite[Thm.~II.2.5 and Lemma~II.2.3]{davidson}), and in fact we have
\begin{equation}\label{wotcon}
\sot_{n\to\infty}\rho(f_n)=I\quad\mbox{in }\,\,\BB(\Hh).
\end{equation}

\vspace*{2mm}\noindent
\underline{\it Part 3.} Using the operators $Q(2^{-n})\coloneqq\rho(f_n)=\rho\widehat{q(2^{-n})}$ we can now define $(Q(t))_{t\geq 0}$ on the set $\mathbb{D}$ of positive dyadic rationals by the formula $Q(t)=Q(1)^{t_0}Q(2^{-m_1})\cdot\ldots\cdot Q(2^{-m_j})$ for any $t\in\mathbb{D}$ written in the form $t=t_0+\sum_{i=1}^j 2^{-m_i}$, where $t_0,j\geq 0$ and $1\leq m_1<\ldots<m_j$ are integers. Then $(Q(t))_{t\in\mathbb{D}}$ is a~dyadic semigroup which, in view of \eqref{wotcon}, satisfies $\sot_{n\to\infty}Q(2^{-n})=I$. 

Now, fix any $t\in\mathbb{D}$ and any sequence $(t_n)_{n=1}^\infty\subset\mathbb{D}$, $t_n\to t$. Instead of the sequence $(f_n)_{n=0}^\infty$ considered above, we define $g_n\in C(\Delta)$ to be the Gelfand transform of $q(t_n)$, for $n\in\N$. Modifying Claims 1--3 in obvious ways we obtain:

\vspace*{1mm}
\begin{itemize}[leftmargin=24pt]
\setlength{\itemsep}{3pt}
\item $\Psi(g_n)=\gamma(\pi(\rho(g_n)))$ for $n\in\N$;

\item $\sot_{n\to\infty}\Psi(g_n)=\gamma(q(t))$;

\item $g_n\xrightarrow[]{\mathmakebox[12pt]{}}\widehat{q(t)}$ in $L^2(\Delta,\mathsf{E}_{\mathbf{x},\mathbf{x}})$, for every $\mathbf{x}\in\mathbb{H}$;

\item $g_n\xrightarrow[]{\mathmakebox[12pt]{}}\widehat{q(t)}$ in $L^2(\lambda)$,
\end{itemize}

\vspace*{1mm}\noindent
and, likewise for $(f_n)_{n=0}^\infty$, the last assertion implies that $g_n\xrightarrow[]{}\widehat{q(t)}$ in $L^2(\mu)$ which in turn yields
$$
Q(t)=\rho\widehat{q(t)}=\sot_{n\to\infty}\rho(g_n)=\sot_{n\to\infty}Q(t_n).
$$
Therefore, $(Q(t))_{t\in\mathbb{D}}$ is a~\SOT-continuous semigroup and hence it is a~$C_0$-semigroup (see Remark~\ref{cont1_R}). Moreover, for any $t\in\mathbb{D}$, we have 
$$
\widehat{\pi Q(t)}=\widehat{\pi\rho\widehat{q(t)}}=\theta\rho\widehat{q(t)}=\widehat{q(t)},
$$
hence $\pi Q(t)=q(t)$ and $(Q(t))_{t\in\mathbb{D}}$ is a~lift of $(q(t))_{t\in\mathbb{D}}$. Consequently, in order to finish the proof we need to show that condition ($\star$) is satisfied for Calkin's representation.

\vspace*{2mm}\noindent
\underline{\it Part 4.} First, we reduce our task to showing ($\star\star$).

\vspace*{2mm}\noindent
{\it Claim 4. }($\star\star$) implies ($\star$).

\vspace*{2mm}
For any $x\in\Hh$ and $\mathbf{x}\in\mathbb{H}$, denote the two functionals mentioned in ($\star$) by $\Lambda_x$ and $\Phi_{\mathbf{x}}$, respectively, i.e. $\Lambda_x g=\ip{\rho(g)x}{x}$ and $\Phi_{\mathbf{x}} g=\ip{\gamma\circ\pi\circ\rho(g)\mathbf{x}}{\mathbf{x}}$. We regard them as positive functionals on the space $C(\Delta)_\R$ of real-valued continuous functions on $\Delta$. The measures corresponding to $\Lambda_x$ and $\Phi_{\mathbf{x}}$ via the Riesz representation theorem are defined on open set by $\kappa(V)=\sup\Lambda_x g$ and $\nu(V)=\sup\Phi_{\mathbf{x}} g$, where the suprema are taken over all $g\in C(\Delta)_\R$ such that $0\leq g\leq 1$ and $\mathrm{supp}(g)\subset V$ (see, e.g., \cite[Thm.~2.14]{rudin}). Since they are also regular, we infer that for condition ($\star$) to hold it suffices that $\Phi_{\mathbf{x}}-\Lambda_x$ is a~positive functional as stated in ($\star\star$).

\vspace*{2mm}\noindent
{\it Claim 5. }($\star\star$) holds true.

\vspace*{2mm}
Recall that for every $T\in\BB(\Hh)$ we have $\gamma(\pi(T))\mathbf{x}=[(T\xi_n)]_{\sim}$ for any equivalence class $\mathbf{x}=[(\xi_n)]_{\sim}\in\mathscr{W}^{\usim}$. Our goal is to show that for any fixed unit vector $x_0\in\Hh$ there is a~weakly null sequence $\mathbf{w}(x_0)\subset\Hh$ such that for every $T\in\mathcal{E}_0=\rho(C(\Delta))$, $T\geq 0$, we have
\begin{equation}\label{T_part5}
\ip{Tx_0}{x_0}\leq \mathop{\mathrm{LIM}}\limits_{n,\mathscr{U}}\big(\ip{T\mathsf{w}(x_0)_n}{\mathsf{w}(x_0)_n}\big)_{n=1}^\infty.
\end{equation}
This condition can be rewritten in terms of the spectral measure $\mathsf{E}^0$ of $\mathcal{E}_0$. Namely, since for any $x\in\Hh$ we have $\ip{Tx}{x}=\int_\Delta \widehat{T}\,\dd\mathsf{E}^0_{x,x}$, condition \eqref{T_part5} is equivalent to 
\begin{equation}\label{g5}
    \int_\Delta f\,\dd \mathsf{E}^0_{x_0,x_0}\leq \mathop{\mathrm{LIM}}\limits_{n,\mathscr{U}}\Big(\int_\Delta f\,\dd\mathsf{E}^0_{\mathsf{w}(x_0)_n,\mathsf{w}(x_0)_n}\Big)_{n=1}^\infty\quad\mbox{for every }\,\, f\in C(\Delta),\, f\geq 0.
\end{equation}

\vspace*{2mm}
In order to find a desired map $x\mapsto\mathsf{w}(x)\in\mathscr{W}^{\usim}$, pick a~sequence $(\mathcal{F}_n)_{n=1}^\infty$ of partitions $\mathcal{F}_n=\{F_{n,1},\ldots,F_{n,k_n}\}$ of $\Delta$ into pairwise disjoint Borel sets such that for all $n\in\N$ and $1\leq j\leq k_n$, we have $\mathrm{int}\,F_{n,j}\neq\varnothing$ and $\mathrm{diam}\,F_{n,j}<\tfrac{1}{n}$. To see that such a~sequence exists, fix $n\in\N$ and pick any finite open cover $\{U_i\}_{i=1}^N$ of $\Delta$ such that $\mathrm{diam}\,U_i<\tfrac{1}{n}$ for each $1\leq i\leq N$. After replacing each $U_i$ by $\mathrm{int}\,\mathrm{cl}\, U_i$ we may assume that all $U_i$'s are regularly open. Let $\mathcal{F}=\{F_j\}_{j=1}^M$ be the collection of all nonempty minimal (with respect to inclusion) sets belonging to the~$\sigma$-algebra generated by $\{U_i\}_{i=1}^N$. Plainly, each $F_j$ is of the form 
\begin{equation}\label{Fj_form}
F_j=\bigcap_{i\in I}U_i\cap\bigcap_{i\in J}\,(\Delta\setminus U_i)\quad\mbox{for some }\,\,I,J\subseteq\{1,\ldots,N\},
\end{equation}
hence $\mathcal{F}$ is a~partition of $\Delta$ into pairwise disjoint Borel sets, each of which has diameter smaller than $\tfrac{1}{n}$. It remains to show that each $F_j$ has nonempty interior. Suppose the contrary and, assuming $F_j$ is given by \eqref{Fj_form}, pick the smallest $i_0\in J$ such that 
$$
\mathrm{int}\,\Bigg[\bigcap_{i\in I}U_i\cap\!\!\bigcap_{i\in J,i\leq i_0}\!\!(\Delta\setminus U_i)\Bigg]=\varnothing.
$$
This means that $V\coloneqq\mathrm{int}\,[\bigcap_{i\in I}U_i\cap\bigcap_{i\in J, i<i_0}(\Delta\setminus U_i)]$ is a~nonempty open set such that $\mathrm{int}\,(V\setminus U_{i_0})=\varnothing$. Hence, $V\subseteq \mathrm{cl}\, U_{i_0}$ and since $U_{i_0}$ is regularly open, we would obtain $V\subseteq U_{i_0}$ which contradicts the fact that $F_j$ is nonempty.

For every $F_{n,j}\in\mathcal{F}_n$ pick a~point $\xi_{n,j}\in F_{n,j}$. Then, for any $f\in C(\Delta)$, we have
\begin{equation}\label{f_riem}
\int_\Delta f\,\dd E^0_{x_0,x_0}=\lim_{n\to\infty}\sum_{j=1}^{k_n} f(\xi_{n,j})\n{\mathsf{E}^0(F_{n,j})x_0}^2.
\end{equation}
Now, we are going to use the assumption that $\Delta$ has no isolated points. It implies that every nonempty set $U\subset\Delta$ contains infinitely many disjoint nonempty sets, hence $\mathsf{E}^0(U)$ is a~projection onto an~infinite-dimensional subspace of $\Hh$. Fix $n\in\N$ and assume that that we have already defined pairwise orthogonal unit vectors $\mathsf{w}(x_0)_1,\ldots,\mathsf{w}(x_0)_{n-1}\in\Hh$ such that $\n{\mathsf{E}^0(F_{i,j})x_0}=\n{\mathsf{E}^0(F_{i,j})\mathsf{w}(x_0)_i}$ for all $1\leq i<n$ and $1\leq j\leq k_i$. Define  $Z=\mathrm{span}\{\mathsf{w}(x_0)_1,\ldots,\mathsf{w}(x_0)_{n-1}\}^\perp$. For each $1\leq j\leq k_n$, the subspace $\mathrm{rg}\,\mathsf{E}^0(F_{n,j})$ is infinite-dimensional, whereas $Z$ is finite-codimensional, hence $\dim (\mathrm{rg}\,\mathsf{E}^0(F_{n,j})\cap Z)=\infty$. We can therefore find a~sequence of orthogonal unit vectors $(v_j)_{j=1}^{k_n}\subset Z$ such that $v_j\in \mathrm{rg}\,\mathsf{E}^0(F_{n,j})$. Define
$$
\mathsf{w}(x_0)_n=\sum_{j=1}^{k_n}\n{\mathsf{E}^0(F_{n,j})x_0}v_j,
$$
which is a unit vector (recall that $\n{x_0}=1$, thus $\n{\mathsf{E}^0(\cdot)x_0}^2$ is a~probabilistic measure), orthogonal to each of $\mathsf{w}(x_0)_1,\ldots,\mathsf{w}(x_0)_{n-1}$ and satisfying 
\begin{equation}\label{e00}
\n{\mathsf{E}^0(F_{n,j})x_0}=\n{\mathsf{E}^0(F_{n,j})\mathsf{w}(x_0)_n}\quad\mbox{for each }\,\, 1\leq j\leq k_n.
\end{equation}
In this way, we obtain an~orthogonal sequence $(\mathsf{w}(x_0)_n)_{n=1}^\infty\subset\Hh$ of unit vectors, hence a~weakly null sequence, satisfying \eqref{e00} for every $n\in\N$.

For any $f\in C(\Delta)$ and $\e>0$, by appealing to \eqref{f_riem} and \eqref{e00}, we can pick $n_0\in\N$ such that for both integrals with respect to $\dd\mathsf{E}^0_{x_0,x_0}$ and $\dd\mathsf{E}^0_{\mathsf{w}(x_0)_n,\mathsf{w}(x_0)_n}$, we have
$$
\Bigg|\int_\Delta f -\sum_{j=1}^{k_n} f(\xi_{n,j})\n{\mathsf{E}^0(F_{n,j})\mathsf{w}(x_0)_n}^2\Bigg|<\e \quad \mbox{for each }\,\, n\geq n_0.
$$
This shows that our choice of the map $x_0\mapsto\mathsf{w}(x_0)$ guarantees that \eqref{g5} holds true with equality, which completes the proof of Claim~5.

\vspace*{2mm}
The last two claims show that Calkin's representation satisfies ($\star$) which, as we explained in Part~3, completes the proof of the theorem.
\end{proof}

As we have seen in the above proof, the issue of continuity of a~lift of $(q(t))_{t\geq 0}$ can be reduced to verification of condition ($\star\star$) which, for Calkin's representation, happens to hold true with equality. However, keeping that condition in its original form gives us a~more general abstract criterion for \SOT-continuity, which may be useful for possible future reference.

%%%%%%%%%%%%%%%%%%%%%%%%%%%%%%%%%%%%%%%%%%%%%%%%%%%%%%%
%%%%%%%%%%%%%%%%%%%%%%%%%%%%%%%%%%%%%%%%%%%%%%%%%%%%%%%
\section{Geometric conditions on the spectrum}
\subsection{Twisting maneuver}
The main goal of this section is to provide geometric conditions on the spectrum $\sigma(A)$ which guarantee surjectivity of the connecting maps corresponding to the first derived functor $\varprojlim{}^{(1)}\Ext_2(\Omega_n)$ in Milnor's exact sequence applied to the inverse system $\{\Omega_n,p_n\}_{n\geq 0}$ described in Proposition~\ref{P_spectrum}. Recall that by Milnor's Theorem~\ref{milnor_thm}, we have an~exact sequence
\begin{equation}\label{milnor_for_P}
\begin{tikzcd}
0 \arrow[r] & \varprojlim{}^{(1)}\Ext(S\Omega_n) \arrow[r] & \Ext(\varprojlim \Omega_n) \arrow[r, "P"] & \varprojlim\Ext(\Omega_n) \arrow[r] & 0
\end{tikzcd}
\end{equation}
associated with any $C_0$-semigroup $(q(t))_{t\geq 0}\subset\QQ(\Hh)$ of normal operators. The connecting maps in the inverse system of suspensions $\{S\Omega_n, (Sp_n)_\ast\}_{n\geq 0}$ are defined as
\begin{equation}\label{Sp_formula}
(Sp_n)_\ast\colon \Ext(S\Omega_{n+1})\xrightarrow[\phantom{xx}]{}\Ext(S\Omega_n),\quad (Sp_n)_\ast\tau(g)=\tau(g\circ Sp_n)
\end{equation}
for every $[\tau]\in \Ext(S\Omega_{n+1})$ and $g\in C(S\Omega_n)$. Note that since $Sp_n$ is surjective, we do not need to add the trivial extension to guarantee that $(Sp_n)_\ast\tau$ given by formula \eqref{Sp_formula} is a~$^\ast$-monomorphism. Hence, to show that $(Sp_n)_\ast$ is surjective, we need to ensure that for every unital $^\ast$-monomorphism $\lambda\colon C(S\Omega_n)\to\QQ(\Hh)$ there is a~unital $^\ast$-monomorphism $\tau\colon C(S\Omega_{n+1})\to\QQ(\Hh)$ satisfying
\begin{equation}\label{surj_tau_g}
\tau(g\circ Sp_n)=\lambda(g)\quad\mbox{for every }\,g\in C(S\Omega_n),
\end{equation}
where the equality is understood as unitary equivalence between the both sides regarded as $^\ast$-homomorphisms on $C(S\Omega_n)$. In the sequel, we will denote elements of any suspension space $SX$ by $[x,t]$ ($x\in X$, $t\in I$), remembering that $X\times\{0\}$ and $X\times\{1\}$ collapse to two distinct points.

Our first step is to show that the question of surjectivity of $(Sp_n)_\ast$ can be reduced to a~certain problem of extending $^\ast$-homomorphisms into the Calkin algebra. Notice that all functions of the form $g\circ Sp_n$ occuring in \eqref{surj_tau_g} have the property of preserving antipodal points, that is,
$$
(g\circ Sp_n)([x,t])=(g\circ Sp_n)([-x,t]),
$$
whenever both $[x,t]$ and $[-x,t]$ belong to $S\Omega_{n+1}$. We are now going to enlarge this class of functions in $C(S\Omega_{n+1})$.

Fix $n\in\N_0$ and for any $\alpha\in [0,2\pi)$, define a~section of $S\Omega_n$ by
\begin{equation}\label{section_deff}
\mathcal{S}_\alpha=\big\{[re^{\ii\alpha},t]\in S\Omega_n\colon r>0,\, 0<t<1\big\}.
\end{equation}
Next, we define two sections of $S\Omega_{n+1}$ corresponding to the two (antipodal) square roots of $e^{\ii\alpha}$, that is,
$$
\mathcal{R}_\pm=\big\{[\pm re^{\ii\alpha/2},t]\in S\Omega_{n+1}\colon r>0,\, 0<t<1\big\}.
$$
Let also
$$
\mathcal{R}^\prime=\big\{[x,t]\in\mathcal{R}_+\cup\mathcal{R}_{-}\colon [-x,t]\in S\Omega_{n+1}\big\}.
$$
Now, we consider a subalgebra of $C(S\Omega_{n+1})$ defined by
$$
\Aa_0=\big\{f\in C(S\Omega_{n+1})\colon f([x,t])=f([-x,t])\mbox{ for every }[x,t]\in\mathcal{R}^\prime\big\}.
$$
Clearly, we have $\Aa_0\cong C(S\Omega_{n+1}^{\usim})$, where $S\Omega_{n+1}^{\usim}=S\Omega_{n+1}/{{}_{\!\sim}}$ is the quotient space defined by the relation $[re^{\ii\alpha/2},t]\sim [-re^{\ii\alpha/2},t]$.

Of course, equation \eqref{surj_tau_g} cannot serve as a~definition of $\tau$ because the functions $g\circ Sp_n$ do not exhaust the whole of $C(S\Omega_{n+1})$. As we will see below, we can get rid of the restriction of preserving antipodal points for all pairs of such points except these which correspond to the direction $\alpha/2$, i.e. the pairs $(re^{\ii\alpha/2},-re^{\ii\alpha/2})$. To this end, for any $r>0$, consider the~`circle section' of $\Omega_{n+1}$ defined as $\Omega_{n+1}\cap\mathbb{T}_r$, where $\mathbb{T}_r=\{z\in\C\colon\abs{z}=r\}$, and proceed with the following `twisting maneuver'. Namely, we cut $\mathbb{T}_r$ at the two antipodal points $\pm re^{\ii\alpha/2}$, twist both parts to circles of radii $r^2$ by identifying the cutting points, and glue them together at the one point corresponding to $\pm re^{\ii\alpha/2}$. In other words, we replace $\mathbb{T}_r$ by the wedge sum $\mathbb{T}_{r^2}\vee\mathbb{T}_{r^2}$. Any function on $\mathbb{T}_r$ which preserved just one pair of antipodal points can be now identified with a~function on $\mathbb{T}_{r^2}\vee\mathbb{T}_{r^2}$ which preserves all pairs of antipodal points. We can extend this procedure naturally to the suspension $S\Omega_{n+1}$. This maneuver, on one circle section of $\Omega_{n+1}$ is illustrated on figure~\ref{twisting_figure} below.

%%%%%%%%%%%%%%%%%%%%%%%%%%%%%%%%%%%%%%%%%%%%%%%%%%%%%%%%%%%%
\begin{figure}[ht]
\centering

\begin{tikzpicture}[scale=.7]

\coordinate (A) at (55:2.5);
\coordinate (B) at (344:2.5);
\coordinate (A1) at ($(0,0)-(A)$);
\coordinate (B1) at ($(0,0)-(B)$);

\coordinate (E1) at (2,2.8);
\coordinate (E2) at (2.6,3);
\coordinate (E3) at (0,3.5);
\coordinate (E4) at (-2.7,2.7);
\coordinate (E5) at (-3.6,1);
\coordinate (E6) at (-2.8,1.2);

\coordinate (E7) at (0,1.8);
\coordinate (E8) at (0.5,2);
\coordinate (E9) at (1.2,1.5);

\coordinate (E11) at ($(0,0)-(E1)$);
\coordinate (E21) at ($(0,0)-(E2)$);
\coordinate (E31) at ($(0,0)-(E3)$);
\coordinate (E41) at ($(0,0)-(E4)$);
\coordinate (E51) at ($(0,0)-(E5)$);
\coordinate (E61) at ($(0,0)-(E6)$);
\coordinate (E71) at ($(0,0)-(E7)$);
\coordinate (E81) at ($(0,0)-(E8)$);
\coordinate (E91) at ($(0,0)-(E9)$);

\draw[darkdarkblue, thick, fill=darkblue] (A) to [curve through = {(E1) .. (E2) .. (E3) .. (E4) .. (E5) .. (E6) .. (B1) .. (E7) .. (E8) .. (E9)}] (A);

\draw[darkdarkblue, thick, fill=darkblue] (A1) to [curve through = {(E11) .. (E21) .. (E31) .. (E41) .. (E51) .. (E61) .. (B) .. (E71) .. (E81) .. (E91)}] (A1);

%two arrows for A(Omega_{n+1})
\draw[{Stealth[scale=1.4]}-{}] (-3.3,1.8) .. controls (-4,1.6) and (-4.4,-0.3) .. (-4.7,-0.7);
\draw[{Stealth[scale=1.4]}-{}] (-1.9,-3.5) .. controls (-3.4,-2.9) and (-3.35,-2) .. (-4.6,-1.4);
\node[black, scale=.9, xshift=-7pt, yshift=1pt] at (-4.1,-1.1) {$\mathsf{A}(\Omega_{n+1})$};

\node[black, scale=1.1, xshift=0pt, yshift=0pt] at (3.2,2.4) {$\Omega_{n+1}$};

\draw[{Stealth[scale=1.4]}-{}] (315:2.5) .. controls (2.3,-2.4) and (2.5,-2.5) .. (2.8,-3.3);
%\node[fill=red, shape=circle, inner sep=.6mm] at (2.3,-2.4) {};
%\node[fill=red, shape=circle, inner sep=.6mm] at (2.5,-2.5) {};
\node[black, scale=1, xshift=0pt, yshift=2pt] at (2.9,-3.9) {$\Omega_{n+1}\cap\mathbb{T}_r$};
%we need an arrow

\coordinate (C) at (190:2.5);
\coordinate (D) at (208:2.5);

\draw[darkblue, thick, fill=bluee] (C) to [curve through = {(-3,-0.7) .. (-2.7,-1.4) .. (D) .. (-2.1,-0.5)}] (C);

\draw[blue, ultra thin, dashed] (10:2.5) to [curve through = {(3,0.7) .. (2.7,1.4) .. (28:2.5) .. (2.1,0.5)}] (10:2.5);

\draw[darkblue, thick, fill=bluee] (36:2.5) to [curve through = {(40:2.7) .. (46:2.5) .. (40:2.2)}] (36:2.5);

\draw[blue, ultra thin, dashed] (216:2.5) to [curve through = {(220:2.7) .. (226:2.5) .. (220:2.2)}] (216:2.5);

\draw (-5,0) -- (5,0) (0,-4.4) -- (0,4.4);
\draw[black, thin, dashed] (0,0) circle [radius=2.5];
\draw[black, ultra thick] (A) arc (55:164:2.5);
\draw[black, ultra thick] (235:2.5) arc (235:344:2.5);
\draw[black, ultra thick] (C) arc (190:208:2.5);
\draw[black, ultra thick] (36:2.5) arc (36:46:2.5);

%red section
\draw[red, thick, dashed] (70:2.5) -- (250:2.5);
\node[draw=red, fill=red, shape=circle, inner sep=.7mm] at (70:2.5) {};
\node[black, scale=.9, xshift=6pt, yshift=12pt] at (70:2.5) {$re^{\ii\alpha/2}$};
\node[draw=red, fill=red, shape=circle, inner sep=.7mm] at (250:2.5) {};
\node[black, scale=.9, xshift=-2pt, yshift=-11pt] at (250:2.5) {$-re^{\ii\alpha/2}$};
%%%%%%%%%%%%%%%%%%%%%%%%%%%%%%%%%%%%%%%%%%%%%%%%%%%%

%snake arrow
\path[draw=black, {}-{Stealth}, very thick, decoration = {snake, pre length=3pt, post length=7pt}, decorate] (5.6,0) -- (7.4,0);

%two glued circles
\draw[black, thin, dashed] (10,2) circle [radius=2];
\draw[black, thin, dashed] (10,-2) circle [radius=2];

%black parts on the upper circle
\draw[black, ultra thick] (10,0) arc (-90:98:2);
\path (10,2) ++(150:2) coordinate (Q);
\draw[black, ultra thick] (Q) arc (150:186:2);
\path (10,2) ++(240:2) coordinate (Q1);
\draw[black, ultra thick] (Q1) arc (240:270:2);

%black parts on the lower circle
\draw[black, ultra thick] (10,0) arc (90:278:2);
\path (10,-2) ++(22:2) coordinate (R);
\draw[black, ultra thick] (R) arc (22:42:2);
\path (10,-2) ++(60:2) coordinate (R1);
\draw[black, ultra thick] (R1) arc (60:90:2);

\node[draw=red, fill=red, shape=circle, inner sep=.7mm] at (10,0) {};

%coloured rectangles
\node[rectangle, draw=black, fill=green, minimum size=1mm] at (120:2.5) {};
\path (10,2) ++(10:2) coordinate (Z1);
\node[rectangle, draw=black, fill=green, minimum size=1mm] at (Z1) {};

\node[rectangle, draw=black, fill=my_green1, minimum size=1mm] (N1) at (300:2.5) {};
\path (10,-2) ++(190:2) coordinate (Z2);
\node[rectangle, draw=black, fill=my_green1, minimum size=1mm] at (Z2) {};

\draw[black, fill=my_orange] (201:2.7) rectangle (197:2.3);
\path (10,2) ++(174:2.15) coordinate (Z31);
\path (10,2) ++(162:1.85) coordinate (Z32);
\draw[black, fill=my_orange] (Z31) rectangle (Z32);

\draw[black, fill=my_yellow] (40.7:2.38) rectangle (41.3:2.62);
\path (10,-2) ++(31:1.85) coordinate (Z41);
\path (10,-2) ++(33:2.15) coordinate (Z42);
\draw[black, fill=my_yellow] (Z41) rectangle (Z42);

\node[black, scale=1, xshift=0pt, yshift=-21pt] at (10,-3.8) {$(\Omega_n\cap\mathbb{T}_{r^2})\vee(\Omega_n\cap\mathbb{T}_{r^2})$};
\end{tikzpicture}
\caption{Twisting maneuver on one circle section of $\Omega_{n+1}$}
\label{twisting_figure}
\end{figure}
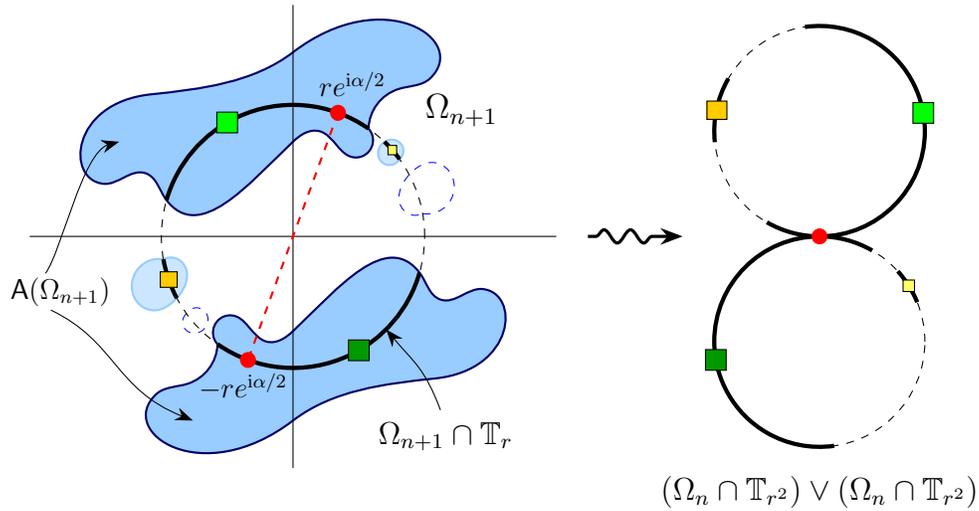
%%%%%%%%%%%%%%%%%%%%%%%%%%%%%%%%%%%%%%%%%%%%%%%%%%%
%%%%%%%%%%%%%%%%%%%%%%%%%%%%%%%%%%%%%%%%%%%%%%%%%%%

Formally, our operation is described as follows. Let $T_\pm$ be the `upper'/`lower' semicircles of the unit circle $\mathbb{T}$ which are determined by the antipodal points $\pm e^{\ii\alpha/2}$, that is, $T_+=\{e^{\ii (\alpha/2+t\pi)}\colon 0\leq t<1\}$ and $T_-=\mathbb{T}\setminus T_+$. For $x\in\Omega_{n}$, let $\sqrt{x}$ be the set of square roots of $x$, and let $s_0(x)\in\sqrt{x}$ be determined by the condition
$$
s_0(x)\in\left\{\begin{array}{cl}
 \sqrt{\abs{x}}T_+ & \mbox{if } \sqrt{\abs{x}}T_+\cap\Omega_{n+1}\not=\varnothing\\
\sqrt{\abs{x}}T_- & \mbox{otherwise}.
\end{array}\right.
$$
We define $s_1(x)$ similarly by swapping $T_+$ and $T_-$. Then, for $j=0,1$ and $f\in\Aa_0$, we set
$$
\Delta_jf([x,t])=f([s_j(x),t])\quad ([x,t]\in S\Omega_n).
$$

In the next lemma, we introduce the first geometric condition which requires that the subset of $\Omega_{n+1}$ consisting of points whose antipodes are not in $\Omega_{n+1}$ is separated from its complement. For any set $E\subseteq\C$, we denote by $\mathsf{A}(E)$ the set of those $z\in E$ for which $-z\in E$. Of course, $\mathsf{A}(E)$ is closed provided $E$ is closed.

\begin{lemma}\label{delta_cont_L}
Suppose that for the fixed $n\in\N_0$, we have
\begin{equation}\label{sep_assumption}
\oo{\Omega_{n+1}
\setminus\mathsf{A}(\Omega_{n+1})}\cap \mathsf{A}(\Omega_{n+1})=\varnothing.
\end{equation}
Then for every $f\in\mathcal{A}_0$, $\Delta_jf$ are continuous on $S\Omega_n$ {\rm (}$j=0,1${\rm )}.
\end{lemma}
\begin{proof}
Consider the case $j=0$ and notice that $S\mathsf{A}(\Omega_{n+1})$ naturally imbeds into $S\Omega_{n+1}$. We duplicate the `upper half' of $f$ by defining $f^\prime\in C(S\mathsf{A}(\Omega_{n+1}))$ by
$$
f^\prime([x,t])=\left\{\begin{array}{ll}
f([x,t]) & \mbox{if }x\in \abs{x}T_+\\
f([-x,t]) & \mbox{if }x\in\abs{x}T_-.
\end{array}\right.
$$
Of course, continuity of $f^\prime$ follows from the fact that $f\in\Aa_0$. Observe that for every $[x,t]\in Sp_n(\mathsf{A}(\Omega_{n+1}))$, we have $\Delta_0f([x,t])=f^\prime ([\sqrt{x},t])$, where we can choose any element of $\sqrt{x}$ as $f^\prime$ coincides on antipodal points. This shows that $\Delta_0f$ is continuous on $Sp_n(\mathsf{A}(\Omega_{n+1}))$ (recall that $p_n(z)=z^2$). By a~similar argument, duplicating the `lower half' of $f$, we show that $\Delta_0f$ is continuous on $Sp_n(\Omega_{n+1}\setminus\mathsf{A}(\Omega_{n+1}))$. According to our assumption \eqref{sep_assumption}, these two sets are separated except possibly the two vertices of $S\Omega_n$, but there $\Delta_0f$ is plainly continuous because so is $f$.
\end{proof}

Concluding, under assumption \eqref{sep_assumption}, we have two maps $\Delta_0,\Delta_1\colon\mathcal{A}_0\to C(S\Omega_n)$ and therefore we can consider $^\ast$-homomorphisms
\begin{equation}\label{resulting_hom}
C(S\Omega_{n+1})\supset\Aa_0\ni f\xmapsto[\phantom{xxx}]{}\lambda(\Delta_j f)\quad (j=0,1).
\end{equation}
Notice that for every $f\in C(S\Omega_{n+1})$ which preserves antipodal points, i.e. $f([x,t])=f(-[x,t])$ for every $[x,t]\in S\Omega_{n+1}$, we have $\Delta_0f=\Delta_1f$, so if $f=g\circ Sp_n$ for some $g\in C(S\Omega_n)$, then $\lambda(\Delta_j f)=\lambda(g)$ for $j=0,1$. Consequently, our task is to extend the $^\ast$-homomorphisms \eqref{resulting_hom} (modulo unitary equivalence) to the whole of $C(S\Omega_{n+1})$, which will be the topic of the two subsequent subsections. The first approach involves Calkin's representation $\gamma$ and splitting the composition $\gamma\circ \lambda$ into cyclic representations.

%%%%%%%%%%%%%%%%%%%%%%%%%%%%%%%%%%%%%%%%%%%%%%%%%%%%%%%%%%%%%%%%%%%%%%%%%%%%%%%%%%%%%%%%%%%%%%%%%
%%%%%%%%%%%%%%%%%%%%%%%%%%%%%%%%%%%%%%%%%%%%%%%%%%%%%%%%%%%%%%%%%%%%%%%%%%%%%%%%%%%%%%%%%%%%%%%%%
\subsection{An empty direction and the CRISP property}
In order to show that the connecting maps in the inverse system $\{\Ext(S\Omega_n),(Sp_n)_\ast\}_{n\geq 0}$ are surjective, we use a~standard technique of extending representations (see, e.g., \cite[Prop.~II.6.4.11]{blackadar}). However, the main difficulty is to ensure that the range algebra do not get larger in the process of extending. This will require some additional geometric conditions on the sets $\Omega_n$, as well as some special properties of the Calkin algebra.

In this subsection, we will use the fact that $\QQ(\Hh)$ has the {\it countable Riesz separation property} (CRISP) which means that for arbitrary sequences $(x_n)_{n=1}^\infty$ and $(y_n)_{n=1}^\infty$ of self-adjoint elements of $\QQ(\Hh)$ satisfying
$$
\ldots\leq x_n\leq x_{n+1}\leq\ldots\leq y_{n+1}\leq y_n\leq\ldots
$$
there exists a~self-adjoint $z\in\QQ(\Hh)$ such that $x_n\leq z\leq y_n$ for each $n\in\N$. (Olsen and Pedersen proved that every corona algebra has the CRISP property; see \cite[Thm.~3.1]{OP}.) 

Fix $n\in\N_0$. Below, we will use the notation introduced in the previous subsection. We also fix Calkin's $^\ast$-representation $\gamma\colon\QQ(\Hh)\to\BB(\mathbb{H})$.

For the fixed unital $^\ast$-monomorphism $\lambda\colon C(S\Omega_n)\to\QQ(\Hh)$, consider $\varrho\coloneqq \gamma\circ\lambda$ which is a~$^\ast$-representation of $C(S\Omega_n)$ on $\mathbb{H}$. By passing to the essential subspace of $\varrho$, if necessary, we may assume that $\varrho$ is nondegenerate and write $\varrho=\bigoplus_{i\in I}\varrho_i$, where each $\varrho_i\colon C(S\Omega_n)\to\BB(\mathbb{H}_i)$ is a~cyclic representation on some $\mathbb{H}_i\subseteq\mathbb{H}$.

Fix an arbitrary $i\in I$ and pick a~unit cyclic vector $\xi_i\in\mathbb{H}_i$ for $\varrho_i$. Let $\p_i$ be the corresponding vector state, i.e. $\p_i(g)=\ip{\varrho_i(g)\xi_i}{\xi_i}$ ($g\in C(S\Omega_n)$), and $\mu_i$ be the probabilistic Borel measure on $S\Omega_n$ that represents $\p_i$. It is well-known that $\varrho_i$, being a~cyclic representation, is unitarily equivalent to the representation given by multiplication operators on the space $L^2(S\Omega_n,\mu_i)$. More precisely, the operator $U_i\colon C(S\Omega_n)\to\mathbb{H}_i$ defined by $U_i(g)=\varrho_i(g)\xi_i$ extends uniquely to a~unitary operator from $L^2(S\Omega_n,\mu_i)$ onto $\mathbb{H}_i$, which we still denote by $U_i$. We then have
$$
\varrho_i(g)=U_i M_g^{\mu_i} U_i^\ast\qquad (g\in C(S\Omega_n)),
$$
where $M_g^{\mu_i}(f)=fg$ (see the proof of \cite[Thm.~II.1.1]{davidson}). 

Recall that for any $\alpha\in [0,2\pi)$, we defined
$$
\mathcal{S}_\alpha=\big\{[re^{\ii\alpha},t]\in S\Omega_n\colon r>0,\, 0<t<1\big\}.
$$
For a fixed $i\in I$, one can obviously find $\alpha$ such that $\mu_i(\mathcal{S}_\alpha)=0$, as $\{\mathcal{S}_\beta\colon 0\leq \beta<2\pi\}$ is an~uncountable collection of pairwise disjoint Borel subsets of $S\Omega_n$. In the next result, however, we require that one can pick $\alpha$ common for all $i\in I$. This is a~rather technical assumption, but observe that it is trivially satisfied if the set $\Omega_n$ simply omits the direction determined by $\alpha$, that is, if $\mathcal{S}_\alpha\cap S\Omega_n=\varnothing$. Another situation when this condition is satisfied is when $\rho$ decomposes as the direct sum of countably many cyclic representations.

\begin{theorem}\label{empty_dir_T}
Assume that condition \eqref{sep_assumption} is satisfied, and there exists $\alpha\in [0,2\pi)$ such that $\mu_i(\mathcal{S}_\alpha)=0$ for all $i\in I$. Then, the homomorphism $(S p_n)_\ast\colon\Ext(S\Omega_{n+1})\to\Ext(S\Omega_n)$ is surjective.
\end{theorem}
\begin{proof}
Recall that, by Lemma~\ref{delta_cont_L}, we have defined maps $\Delta_1,\Delta_2\colon \Aa_0\to C(S\Omega_n)$. Consider two representations of $\Aa_0$ given by
$$
C(S\Omega_{n+1})\supset\Aa_0\ni f\xmapsto[\phantom{xxx}]{}\tau_{0,i}^{(j)}(f)\coloneqq\varrho_i(\Delta_j f)\quad (j=0,1).
$$
Obviously, $\xi_i$ is again a cyclic vector, as we have $[\tau_{0,i}^{(j)}(\Aa_0)\xi_i]=[\varrho_i(C(S\Omega_n))\xi_i]=\mathbb{H}_i$. Hence, the corresponding vector state is given by
\begin{equation}\label{new_v_state}
    %\begin{split}
        \p^{(j)}(f)=\ip{\tau_{0,i}^{(j)}(f)\xi_i}{\xi_i}=\ip{\varrho_i(\Delta_jf)\xi_i}{\xi_i}=\int_{S\Omega_n}\!\!\Delta_jf\,\dd\mu_i.
    %\end{split}
\end{equation}
For $j=0,1$ define $\Phi_j\colon S\Omega_n\to S\Omega_{n+1}^{\usim}$ by $\Phi_j([x,t])=[s_j(x),t]_\sim$. Let $\w\nu_i^{\,(j)}$ be the pull-back Borel measure on $S\Omega_{n+1}^{\usim}$ defined by 
$$
\w\nu_i^{\,(j)}=\mu_i\circ\Phi_j^{-1}.
$$
Then, applying a~simple change of variables in \eqref{new_v_state}, we obtain
\begin{equation*}
    \p^{(j)}(f)=\int_{S\Omega_{n+1}^{\usim}} \! f\,\dd\w\nu_i^{\,(j)}\qquad (f\in\Aa_0,\, j=0,1).
\end{equation*}
The measure $\w\nu_i^{\,(j)}$ naturally induces a~Borel probabilistic measure $\nu_i^{\,(j)}$ on $S\Omega_{n+1}$ obtained by `unsticking' the points identified by the relation $\sim$. Namely, for any Borel set $E\subseteq S\Omega_{n+1}$, we define
$$
\nu_i^{\,(j)}(E)=\w\nu_i^{\,(j)}(E\setminus (\mathcal{R}_+\cup\mathcal{R}_-)),
$$

\vspace*{1mm}\noindent
where the latter set is understood formally as an~equivalence class. Notice that in fact we have $\nu_i^{\,(j)}(E)=\w\nu_i^{\,(j)}([E]_\sim)$, since $\w\nu_i^{\,(j)}([\mathcal{R_\pm}]_\sim)=0$ which follows from $\mu_i(\mathcal{S}_\alpha)=0$. 

Similarly as for the representation $\varrho_i$, we note that the operator 
$$
U_i^{(j)}\colon C(S\Omega_{n+1}^{\usim})\xrightarrow[\phantom{xxx}]{}\mathbb{H}_i,\quad U_i^{(j)}(f)=\tau_{0,i}^{(j)}(f)\xi_i
$$
extends uniquely to a~unitary operator from $L^2(S\Omega_{n+1}^{\usim},\w\nu_i^{\,(j)})$ onto $\mathbb{H}_i$, still denoted by $U_i^{(j)}$, and such that
\begin{equation}\label{tau_0_rep}
\tau_{0,i}^{(j)}(f)=U_i^{(j)}M_f^{\w\nu_i^{(j)}} U_i^{(j)\ast}\quad (f\in\Aa_0\cong C(S\Omega_{n+1}^{\usim})).
\end{equation}
Consider an~operator 
$$
\iota\colon L^2(S\Omega_{n+1},\nu_i^{\,(j)})\xrightarrow[\phantom{xxx}]{}L^2(S\Omega_{n+1}^{\usim},\w\nu_i^{\,(j)}),\quad \iota(f)([x,t]_\sim)=\left\{\begin{array}{ll}
f([x,t]) & \mbox{if }[x,t]\not\in\mathcal{R}_\pm\\
0 & \mbox{if }[x,t]\in\mathcal{R}_\pm.
\end{array}\right.
$$
Now, we can extend the representation $\tau_{0,i}^{(j)}$ to the whole of $C(S\Omega_{n+1})$ with the aid of the following diagram
$$
\xymatrix{
L^2(S\Omega_{n+1},\nu_i^{\,(j)}) \ar[r]^\iota & L^2(S\Omega_{n+1}^{\usim},\w\nu_i^{\,(j)}) \ar[r]^{\qquad\quad U_i^{(j)}} & \mathbb{H}_i\\
L^2(S\Omega_{n+1},\nu_i^{\,(j)}) \ar[u]_{M_f} & L^2(S\Omega_{n+1}^{\usim},\w\nu_i^{\,(j)}) \ar[l]_{\,\,\,\,\iota^\ast} & \mathbb{H}_i \ar[l]_{\qquad\quad\,\, U_i^{(j)\ast}}
}
$$
that is, for any $f\in C(S\Omega_{n+1})$, we set
\begin{equation}\label{tau_extended}
\tau_i^{(j)}(f)=U_i^{(j)}\iota M_f (U_i^{(j)}\iota)^\ast,
\end{equation}
where $M_f$ is the multiplication operator on $L^2(S\Omega_{n+1},\nu_i^{\,(j)})$ given by $f$. This formula gives rise to representations $\tau_i^{(j)}\colon C(S\Omega_{n+1})\to \BB(\mathbb{H}_i)$, for $j=0,1$. Finally, we define
$$
\tau_i\colon C(S\Omega_{n+1})\xrightarrow[\phantom{xx}]{}\MM_2(\BB(\mathbb{H}_i)),\quad \tau_i=\begin{pmatrix}
\tau_i^{(0)} & 0\\
0 & \tau_i^{(1)}\end{pmatrix}
%\Bigg[\!\begin{array}{ll}
%\tau_i^{(1)} & 0\\
%0 & \tau_i^{(2)}
%\end{array}\!\!\!\Bigg].
$$

Given any $j\in\{0,1\}$ and $f\in C(S\Omega_{n+1})$, consider the operator $T_f\coloneqq \iota M_f\iota^\ast$. For any $g\in L^2(S\Omega_{n+1}^{\usim})$ and $h\in L_2(S\Omega_{n+1})$ we have
$$
\ip{h}{\iota^\ast(g)}=\ip{\iota(h)}{g}=\int_{S\Omega_{n+1}^{\usim}}\!\!\!\iota(h)\oo{g}\,\dd\w\nu_i^{(j)}=\int_{S\Omega_{n+1}}\!\!\! h\oo{g}\,\dd\nu_i^{(j)}=\ip{h}{g},
$$
where in the last two expressions we regard $g$ naturally as a~function on $S\Omega_{n+1}$ by putting the value zero on $\mathcal{R}_\pm$. Therefore, what the operator $T_f$ does is first `unsticking' the points identified by $\sim$, multiplying by $f$, and finally gluing the points of $\mathcal{R}_\pm$ back together. But this is nothing else than multiplication by $f$ on $L^2(S\Omega_{n+1}^{\usim})$, with $f$ regarded as a~(not necessarily continuous) function on $S\Omega_{n+1}^{\usim}$ (the values on $\mathcal{R}_\pm$ being ignored).

\vspace*{2mm}\noindent
{\it Claim 1. }For any $j\in\{0,1\}$, and $f\in C(S\Omega_{n+1})$, there is a~sequence $(f_k)_{k=1}^\infty\subset \Aa_0$ such that
$$
f_k([x,t])\xrightarrow[\phantom{xx}]{}f([x,t])\quad  \nu_i^{(j)}\mbox{-a.e. on }S\Omega_{n+1}\quad (i\in I).
$$

\noindent
Indeed, as we have already observed,  $\mathcal{R}_+\cup\mathcal{R}_-$ is of measure zero, so it is enough to guarantee pointwise convergence on $S\Omega_{n+1}\setminus(\mathcal{R}_+\cup\mathcal{R}_-)$. To this end, take any continuous map $\beta\colon [0,1]\to [0,\frac{\pi}{2})$ with $\beta(0)=\beta(1)=0$ and positive elsewhere. Define
$$
V(\beta)=\big\{[x,t]\in S\Omega_{n+1}\colon \vert \mathrm{arg}(xe^{-\ii\alpha/2})\vert<\beta(t)\,\mbox{ or }\,\vert\mathrm{arg}(-xe^{-\ii\alpha/2})\vert<\beta(t)\big\},
$$
where we adopt any convention for $\mathrm{arg}(0)$ so that $[0,t]\not\in V(\beta)$ for any $[0,t]\in S\Omega_{n+1}$. Note that the vertices $\ttb_0$ and $\ttb_1$ of $S\Omega_{n+1}$ do not belong to $V(\beta)$ either, and $V(\beta)$ is an~open subset of $S\Omega_{n+1}$. Let $F(\beta)=(S\Omega_{n+1}\setminus V(\beta))\cup\mathcal{R}_+\cup\mathcal{R}_-$, which is a~closed subset of $S\Omega_{n+1}$, and define a~map $f_\beta\colon F(\beta)\to\C$ by 
$$
f_\beta([x,t])=\left\{\begin{array}{ll}
f([x,t]) & \mbox{if }[x,t]\in S\Omega_{n+1}\setminus V(\beta)\\
f([0,t]) & \mbox{if }[x,t]\in \mathcal{R}_+\cup\mathcal{R}_-\mbox{ and }0\in\Omega_{n+1}\\
(1-t)f(\ttb_0)+tf(\ttb_1) & \mbox{if }[x,t]\in \mathcal{R}_+\cup\mathcal{R}_-\mbox{ and }0\not\in\Omega_{n+1}.
\end{array}\right.
$$ 
It is easy to see that $f_\beta$ is continuous. We take any continuous extension to $S\Omega_{n+1}$ and denote it by the same symbol. Then, by the very definition, $f_\beta\in\Aa_0$. Take any sequence of functions $(\beta_k)_{k=1}^\infty$ as above, pointwise convergent to zero on $[0,1]$. Then, $(f_{\beta_k})_{k=1}^\infty\subset\Aa_0$ converges pointwise to $f$ outside $\mathcal{R}_+\cup\mathcal{R}_-$, as desired.

\medskip
By manipulating with maxima and minima, we may also guarantee that the sequence $(f_k)_{k=1}^\infty$ is pointwise increasing and $f_k\leq f$ for each $k\in\N$, and conversely, that it is pointwise decreasing and $f_k\geq f$ for each $k\in\N$. Hence, for a~fixed $f\in C(S\Omega_{n+1})$ and $j\in\{0,1\}$, let us choose $(f_k)_{k=1}^\infty, (g_k)_{k=1}^\infty\subset\mathcal{A}_0$ so that
$$
f_k([x,t]),\, g_k([x,t])\xrightarrow[\phantom{xx}]{}f([x,t])\quad  \nu_i^{(j)}\mbox{-a.e. on }S\Omega_{n+1}\quad (i\in I),
$$
and $\ldots\leq f_k\leq f_{k+1}\leq\ldots\leq g_{k+1}\leq g_k\leq\ldots $

For any $i\in I$, by Lebesgue's theorem, we have $f_k, g_k\xrightarrow[]{w\ast}f$ in $L^\infty(S\Omega_{n+1}^{\usim})$ and hence, 
$$
\wot_{k\to\infty}\iota M_{f_k}\iota^\ast=\wot_{k\to\infty}\iota M_{g_k}\iota^\ast=\iota M_f\iota^\ast
$$
(see \cite[Lemma~II.2.3]{davidson}). Combining \eqref{tau_0_rep} and \eqref{tau_extended}, and using the fact that multiplication is separately weakly continuous, we infer that
\begin{equation}\label{wot_tau}
\wot\limits_{k\to\infty}\tau_i^{(j)}(f_k)=\wot\limits_{k\to\infty}\tau_i^{(j)}(g_k)=\tau_i^{(j)}(f).
\end{equation}
Also, note that $\iota M_{f_k}\iota^\ast$ and $\iota M_{f_k}\iota^\ast$ are multiplication operators given by continuous functions on $S\Omega_{n+1}^{\usim}$. Therefore, under natural identification, we have $\tau_i^{(j)}(f_k)=\tau_{0,i}^{(j)}(f_k)=\varrho_i(\Delta_j f_k)$ for $k\in\N$, whence
$$
\bigoplus_{i\in I}\tau_i^{(j)}(f_k)=\bigoplus_{i\in I}\varrho_i(\Delta_j f_k)=\varrho(\Delta_jf_k)\in\gamma\circ\lambda(C(S\Omega_n))\quad (k\in\N),
$$
and similarly for the functions $g_k$. Hence, we can define $x_k, y_k\in\QQ(\Hh)$ ($k\in\N$) by 
$$
x_k=\gamma^{-1}\Bigg\{\bigoplus_{i\in I} U_i^{(j)}M_{f_k}{U_i^{(j)}}^\ast\Bigg\},\quad y_k=\gamma^{-1}\Bigg\{\bigoplus_{i\in I} U_i^{(j)}M_{g_k}{U_i^{(j)}}^\ast\Bigg\}. 
$$
Of course, $x_k\leq x_{k+1}\leq y_{k+1}\leq y_k$ for every $k\in\N$. By the CRISP property, we may find $z\in\QQ(\Hh)$ such that $x_k\leq z\leq y_k$ for $k\in\N$. Notice that $\gamma(z)$ is an~upper bound for the operators $\bigoplus_{i\in I} U_i^{(j)}M_{f_k}{U_i^{(j)}}^\ast\in\BB(\mathbb{H})$ and a~lower bound for $\bigoplus_{i\in I} U_i^{(j)}M_{g_k}{U_i^{(j)}}^\ast\in\BB(\mathbb{H})$ ($k\in\N$). These sequence converge in the strong operator topology to their \vspace*{-1pt}supremum and infimum, respectively (see, e.g., \cite[Lemma~I.6.4]{davidson}). Hence, by \eqref{wot_tau}, we obtain $\bigoplus_{i\in I}\tau_i^{(j)}(f)=z$. Therefore, in view of \eqref{tau_extended}, we conclude that the formula 
$$
\oo{\tau}^{(j)}(f)=\gamma^{-1}\Bigg\{\bigoplus_{i\in I}\big(U_i^{(j)}\iota M_f\iota^\ast {U_i^{(j)}}^\ast\big)(f)\Bigg\}\quad\,\,\, (f\in C(S\Omega_{n+1}))
$$
yields an extension of $\gamma^{-1}\{\bigoplus_{i\in I}\tau_{0,i}^{(j)}\}=\lambda\circ\Delta_j$ to a~$^\ast$-homomorphism into $\QQ(\Hh)$. The~whole procedure was done for any fixed $j\in\{0,1\}$. Take, for example, $j=0$ and let $[\sigma]$ be the trivial extension of $S\Omega_{n+1}$. Then $[\oo{\tau}^{(0)}\oplus\sigma]\in\Ext(S\Omega_{n+1})$ is an~extension satisfying $(Sp_n)_\ast([\oo{\tau}^{(0)}\oplus\sigma])=[\lambda]$, as desired.
\end{proof}

%%%%%%%%%%%%%%%%%%%%%%%%%%%%%%%%%%%%%%%%%%%%%%%%%%%%%%%%%%%%%%%%%%%%%%%%%%%%%%%%%%%%%%%%%%%%%%%%%
%%%%%%%%%%%%%%%%%%%%%%%%%%%%%%%%%%%%%%%%%%%%%%%%%%%%%%%%%%%%%%%%%%%%%%%%%%%%%%%%%%%%%%%%%%%%%%%%%
%%%%%%%%%%%%%%%%%%%%%%%%%%%%%%%%%%%%%%%%%%%%%%%%%%%%%%%%%%%%%%%%%%%%%%%%%%%%%%%%%%%%%%%%%%%%%%%%%
%%%%%%%%%%%%%%%%%%%%%%%%%%%%%%%%%%%%%%%%%%%%%%%%%%%%%%%%%%%%%%%%%%%%%%%%%%%%%%%%%%%%%%%%%%%%%%%%%
\subsection{A cross retract and Kasparov's technical theorem}
The approach presented in this subsection is based on a~`cutting lemma' which, roughly speaking, allows us to extend $\ast$-homomorphisms from $C(X)$ into corona algebras by cutting the spectrum $X$ along its retract. The main tool here is Kasparov's theorem (see, e.g., \cite[Lemma~11.3.5]{loring}); we need a~special case of that result which we recall below.

\begin{theorem}[{\bf Kasparov's Technical Theorem}]\label{KTT_T}
Let $E$ be a $\sigma$-unital \cs-algebra and $\mathscr{C}(E)=\mathscr{M}(E)/E$ be the corona algebra. Suppose that $D$ is a~separable subset of $\mathscr{C}(E)$. If $x,y\in \mathscr{C}(E)\cap D^\prime$ satisfy $x,y\geq 0$ and $xy=0$, then there exists $0\leq z\leq 1$, $z\in\mathscr{C}(E)\cap D^\prime$ such that $zx=0$ and $zy=y$.
\end{theorem}

Our idea is based on the cutting lemma quoted below (see \cite[Lemma~18.1.2]{loring}). We use the same notation as in \cite[Ch.~18]{loring}: if $X$ is a~compact meric space and $X_0,X_1\subseteq X$ are closed subspaces, we write $X=X_0\cup_Z X_1$ for the `join over $Z$' which means that $X_0\cap X_1=Z$ and there exist retractions $r_j\colon X_j\to Z$ ($j=0,1$). If this is the case, we also have retractions $\eta\colon X\to Z$ and $\eta_j\colon X\to X_j$, for $j=0,1$, defined by
$$
\eta(x)=\left\{\begin{array}{cl} r_0(x), & x\in X_0\\ r_1(x), & x\in X_1,\end{array}\right. \quad
\eta_0(x)=\left\{\begin{array}{cl} x, & x\in X_0\\ r_1(x), & x\in X_1,\end{array}\right. 
\quad\eta_1(x)=\left\{\begin{array}{cl} r_0(x), & x\in X_0\\ x, & x\in X_1.\end{array}\right. 
$$
Let also $\widehat{\eta},\wh{\eta}_0,\wh{\eta}_1\colon C(X)\to C(X)$ be the corresponding composition maps, i.e. $\wh{\eta}(f)=f\circ\eta$ and $\wh{\eta}_j(f)=f\circ\eta_j$ for $j=0,1,$.

Define open disjoint sets $V_0,V_1\subset X$ by $V_j=X\setminus X_j$; define also:
\begin{itemize}[leftmargin=24pt]
\setlength{\itemsep}{2pt}
    \item $\widehat{X}=X_0\sqcup X_1$ (the disjoint union)
    
    \item $\widetilde{X}=\{(x,t)\in X\times [0,1]\colon x\in V_i\,\Longrightarrow\, t=i\}$;
\end{itemize}
the latter set can also be regarded as a~subset of $X\times [0,1]$, specifically, $\widehat{X}=X_0\times\{0\}\cup X_1\times\{1\}$.

\begin{lemma}[{\bf `Cutting lemma'}]\label{cutting_L}
Let $X$ be a compact metric space which is a~join over a~closed zet $Z\subseteq X$, i.e. $X=X_0\cup_Z X_1$, and let $E$ be a~$\sigma$-unital \cs-algebra. For every $\ast$-homomorphism $\p\colon C(X)\to\mathscr{C}(E)$, there is a~unitary equivalence between
$$
\p\oplus (\p\circ\widehat{\eta}\,)\colon C(X)\xrightarrow[]{\phantom{xx}}\mathbb{M}_2(\mathscr{C}(E))
$$
and
$$
(\p\circ\wh{\eta}_0)\oplus (\p\circ\wh{\eta}_1)\colon C(X)\xrightarrow[]{\phantom{xx}}\mathbb{M}_2(\mathscr{C}(E))
$$
In particular, there is an extension of $\p\oplus (\p\circ\wh{\eta})$ to $C(\wh{X})$.
\end{lemma}

Fix $n\in\N_0$. Using the same notation as in Subsection~5.1, we have the $^\ast$-monomorphism
$$
C(S\Omega_{n+1}^{\usim})\cong\Aa_0\ni f\xmapsto[\phantom{xxx}]{}\p(f)\coloneqq
\lambda(\Delta_0f)\oplus\lambda(\Delta_1f)\in\mathbb{M}_2(\QQ(\Hh)),
$$
as well as the modified version defined by $\widetilde{\p}(f)=
\lambda(\Delta_0f)\oplus\sigma(f)$, where $[\sigma]$ is the zero element of $\Ext(S\Omega_{n+1}^{\usim})$.

For any $0\leq\theta<2\pi$, the line $\R e^{\ii\theta}$ splits the plane into two closed half-planes $H_0$ and $H_1$ such that $H_0\cap H_1=\R e^{\ii\theta}$. For any set $W\subset\C$, we call $W\cap H_0$ and $W\cap H_1$ the `left' part and the `right' part of $W$, respectively (the order of $H_0$ and $H_1$ is irrelevant). In the theorem below we assume that $\Omega_{n+1}$ has two different line sections, both of which are retracts of their both left and right parts of $\Omega_{n+1}$, as shown on figure \ref{kasparov_figure}.

\tdplotsetmaincoords{80}{110}

\begin{figure}[ht]
\centering

\begin{tikzpicture}[tdplot_main_coords, scale=1.5]
%\draw[-latex] (-4,0,0) -- (5,0,0) node[below left] {$x$};
%\draw[-latex] (0,-4.5,0) -- (0,5,0) node[below left] {$y$};
%\draw[-latex] (0,0,0) -- (0,0,3.5) node[below left] {$z$};

\begin{scope}[canvas is xy plane at z=0]
    \path[postaction={decorate, decoration={
            markings, mark=at position 0.1 with
             {\coordinate (A);}}}] 
        (0,2) ellipse (1.4 and 2); 
    %\node[draw=red, fill=red, shape=circle, inner sep=.7mm] at (A) {};
    \gettikzxy{(A)}{\ax}{\ay}
    
    %elliptical arc
    \coordinate (B) at (-\ax,\ay);
    %\node[draw=red, fill=red, shape=circle, inner sep=.7mm] at (B) {};
    \draw[color=white, pattern=north east lines, pattern color=bluee2] (0,2) ellipse (1.4 and 2);
    
    \draw[thick, black, dashed] (B) arc (156:270:1.4 and 2);
    
     \coordinate (A1) at (\ax,-\ay);
     \coordinate (B1) at (-\ax,-\ay);
     
     \draw[color=white, pattern=north east lines, pattern color=bluee2] (0,-2) ellipse (1.4 and 2);
     \draw[color=white, pattern=vertical lines, pattern color=bluee2] (0,-2) ellipse (1.4 and 2);
     \draw[color=white, pattern=vertical lines, pattern color=bluee2] (0,2) ellipse (1.4 and 2);
     \draw[thick, black, dashed, postaction={decorate, decoration={
            markings, mark=at position 1 with
             {\coordinate (Q);}}}] (0,0) arc (90:181:1.4 and 2);
     \draw[thick, black, dashed, postaction={decorate, decoration={
            markings, mark=at position 1 with
             {\coordinate (B1);}}}] (Q) arc (181:220:1.4 and 2);         
     \draw[ultra thick, black] (B1) arc (-140:90:1.4 and 2);

    \path[name path=elli, pattern=vertical lines, pattern color=darkdarkblue] (0,0) arc (-90:36:1.4 and 2);
    \path[name path=ell2, pattern=vertical lines, pattern color=darkdarkblue] (A1) arc (-28:90:1.4 and 2);
    
    %\draw[thick, dashed, darkred] (0,-4) arc (-90:90:2.8 and 4);
    \draw[thin, shadow] (0,0) ellipse (2.8 and 4);
    \draw[thick, dashed, darkred] (-2.8,0) -- (2.8,0);
    \draw[thick, darkred] (-4.4,0) -- (-2.8,0) (5.1,0) -- (2.8,0);
    \node[darkred, scale=1, xshift=-6pt, yshift=-8pt] at (5.1,0) {$\R e^{\ii\alpha/2}$};

    \draw[shadow, thick] (0,-4.6) -- (0,5.5);
    \node[shadow, scale=1, xshift=3pt, yshift=10pt] at (0,5.5) {$\R e^{\ii\theta}$};

\end{scope}

\begin{scope}[canvas is plane = {O(0,0,0)x(1,0,0)y(0,0,1)}]
    \path[postaction={decorate, decoration={
            markings, mark=at position 0.1 with
             {\coordinate (C1);}, mark=at position 0.4 with
             {\coordinate (C2);}, mark=at position 0.6 with
             {\coordinate (C3);}, mark=at position 0.9 with
             {\coordinate (C4);}}}] (1.6,0) arc (0:180:1.6 and 1.4);
    \draw[darkred,thick,{}-{Stealth}] (C1) arc (18:72:1.6 and 1.4);
    \draw[darkred,thick,{}-{Stealth}] (C4) arc (162:108:1.6 and 1.4);
            
\end{scope}

\path[pattern=vertical lines, pattern color=darkdarkblue] (0,0,0) -- (0,0,2.2) to [bend right=20] (A);

\path[pattern=vertical lines, pattern color=darkdarkblue] (0,0,0) -- (0,0,2.2) to [bend left=36] (A1); 

\path[pattern=vertical lines, pattern color=bluee2] (A) to [bend left=20] (0,0,2.2) .. controls (-0.5,0.2,1) and (-1.5,1,0.2) .. (B);

\path[pattern=vertical lines, pattern color=bluee2] (A1) to [bend right=36] (0,0,2.2) .. controls (-0.5,-0.7,1) and (-1.5,-1.5,0.2) .. (B1);

 %\node[draw=red, fill=red, shape=circle, inner sep=.7mm] at (0,0,2.2) {};

 \draw[ultra thick, black] (0,0) arc (-90:156:1.4 and 2);
 \draw[ultra thick, black] (0,0,2.2) to [bend right=20] (A);
 \draw[ultra thick, black] (0,0,2.2) to [bend left=36] (A1);
 \draw[ultra thick, black] (0,0,2.2) .. controls (-0.5,0.2,1) and (-1.5,1,0.2) .. (B);
 \draw[ultra thick, black] (0,0,2.2) .. controls (-0.5,-0.7,1) and (-1.5,-1.5,0.2) .. (B1);

\draw[darkred, ultra thick] (0,0,0) -- (0,0,2.2);

%labels and arrows
\node[black, scale=1, xshift=-6pt, yshift=-6pt] at (0,1.3,2.3) {$\pi_0(S\Omega_{n+1}^{\usim})$};

\draw[{Stealth[scale=1.4]}-{}] (-0.5,1.6,0.2) to [bend left=26] (-1,2.1,1.2);
\node[black, scale=1, xshift=9pt, yshift=6pt] at (-1,2.1,1.2) {$H_0$};

\draw[{Stealth[scale=1.4]}-{}] (0.5,1,0.3) to [bend right=24] (1.4,0.5,-0.7);
\node[black, scale=1, xshift=1pt, yshift=-8pt] at (1.4,0.5,-0.7) {$H_1$};
\end{tikzpicture}
\caption{The section $t=0$ of $S\Omega_{n+1}^{\usim}$ as in Theorem \ref{main_K_T}}
\label{kasparov_figure}
\end{figure}
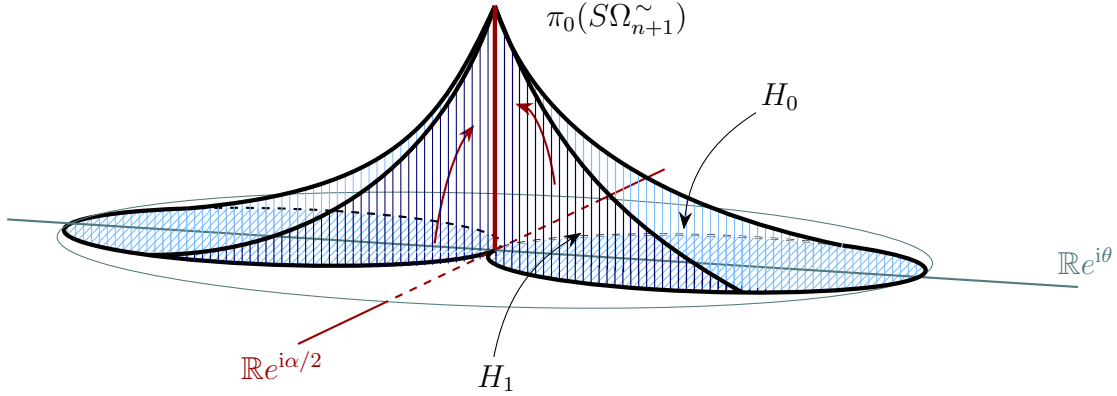

\begin{theorem}\label{main_K_T}
Assume that condition \eqref{sep_assumption} is satisfied, and there exist $\alpha,\theta\in [0,2\pi)$, $\tfrac{\alpha}{2}\not\in \{\theta,\theta-\pi\}$ such that each of the sections
$$
\mathsf{S}_{\alpha/2}=\R e^{\ii\alpha/2}\cap\Omega_{n+1},\,\,\,\,\, \mathsf{S}_\theta=\R e^{\ii\theta}\cap\Omega_{n+1}
$$

\vspace*{1mm}\noindent
is a~retract of both the corresponding left and the right part of $\Omega_{n+1}$. Then, both $\p$ and $\w\p$ can be extended to $^\ast$-monomorphisms $C(S\Omega_{n+1})\to\QQ(\Hh)$ and the homomorphism $(Sp_n)_\ast\colon\Ext(S\Omega_{n+1})\to\Ext(S\Omega_n)$ is surjective.
%If $\varnothing\neq\mathsf{A}(\Omega_{n+1})\subsetneq\Omega_{n+1}$, then $\mathrm{ker}\,(Sp_n)_\ast\neq 0$.
\end{theorem}
\begin{proof}
We work with the compact metric space $X=S\Omega_{n+1}^{\usim}$ which is split into its left and right part. To be more specific, define
$$
X_0=\big\{[re^{\ii\beta},\,t]_\sim\in S\Omega_{n+1}^{\usim}\colon r\geq 0,\, \tfrac{\alpha}{2}-\pi\leq\beta\leq\tfrac{\alpha}{2},\, t\in I\big\}
$$
and the other side by
$$
\,X_1=\big\{[re^{\ii\beta},\,t]_\sim\in S\Omega_{n+1}^{\usim}\colon r\leq 0,\, \tfrac{\alpha}{2}-\pi\leq\beta\leq\tfrac{\alpha}{2},\, t\in I\big\}.
$$

\vspace*{1mm}
Let $H_0$, $H_1$ be the left and the right part of $\Omega_{n+1}$ corresponding to the section $\mathsf{S}_\theta$. It follows from our assumption that the `cross' $\mathsf{S}_{\alpha/2}\cup\mathsf{S}_\theta$ is a~retract of $H_0$ and $H_1$. Indeed, for $z$ belonging to the `quadrant' given by the left parts with respect to $\alpha/2$ and $\theta$, we can define
$$
r(z)=\frac{\dist(z,\mathsf{S}_{\alpha/2})r_{\theta}(z)+\dist(z,\mathsf{S}_{\theta})r_{\alpha/2}(z)}{\dist(z,\mathsf{S}_{\alpha/2})+\dist(z,\mathsf{S}_{\theta})}
$$

\vspace*{1mm}\noindent
(and $r(0)=0$), where $r_{\alpha/2}$ and $r_\theta$ are the retractions corresponding to the both sections. A~similar formula works for other quadrants. Passing to the suspension and taking the quotient space, we see that the set
\begin{equation*}
\begin{split}
Z\coloneqq\big\{[x,t]_\sim &\in S\Omega_{n+1}^{\usim}\colon x\in\mathsf{S}_\theta,\, t\in I\big\}\\
&\cup\big\{[\pm re^{\ii\alpha/2},t]_\sim\in S\Omega_{n+1}^{\usim}\colon t\in I\big\}
\end{split}
\end{equation*}
is a~retract of $X_0$ and $X_1$, so that we have $X=X_0\cup_Z X_1$. In what follows, we use the notation introduced before Lemma~\ref{cutting_L}, and we work with the $^\ast$-monomorphism $\p$; similar calculations go through for $\widetilde{\p}$.

\vspace*{2pt}
Let $\mu\colon C(X)\to C(\w{X})$ be the embedding given by $\mu(f)(\mathbf{x},v)=f(\mathbf{x})$ for $\mathbf{x}=[x,t]_\sim\in X$ and $v\in I$. We treat $\mu\otimes\id$ as a~homomorphism $\mathbb{M}_2(C(X))\to\mathbb{M}_2(C(\w{X}))$. Define also
\begin{equation*}
\begin{split}
\Phi &=\id\,\oplus\,\wh\eta\colon\, C(X)\xrightarrow[\phantom{xx}]{}\mathbb{M}_2(C(X)),\\
\Psi &=\wh\eta_0\oplus\wh\eta_1\colon C(X)\xrightarrow[\phantom{xx}]{}\mathbb{M}_2(C(X)).
\end{split}
\end{equation*}

\vspace*{1mm}\noindent
{\it Claim 1. }There exists $\mathbf{u}\in\mathbb{U}_2(C(\w{X}))$, that is, a~unitary $2\times 2$ matrix with entries from $C(\w{X})$, such that 
\begin{equation}\label{ad_u}
\mathrm{ad}_{\mathbf{u}}\big[(\mu\otimes\id)\circ \Phi\big]= (\mu\otimes\id)\circ\Psi.
\end{equation}

To prove this claim, consider the matrices $\mathbf{U}, \mathbf{R}_v\in \mathbb{U}_2$, for $0\leq v\leq 1$, given by

\vspace*{1mm}
$$
\mathbf{U}=\frac{1}{\sqrt{2}}\begin{pmatrix}1 & 1\\
-1 & 1
\end{pmatrix},\quad\,\,\, \mathbf{R}_v=\mathbf{U}^\ast\begin{pmatrix}
1 & 0\\
0 & e^{\ii\pi v}
\end{pmatrix}\mathbf{U}.
$$

\vspace*{2mm}\noindent
In other words, $\mathbf{R}_v=\mathbf{U}^\ast(\mathbf{I}_1\oplus e^{\ii\pi v}\mathbf{I}_1)\mathbf{U}$ and $[0,1]\ni v\mapsto \mathbf{R}_v$ is a~continuous path in $\mathbb{U}_2$ connecting the identity matrix $\mathbf{I}_2$ with 
$$
\mathbf{R}_1=\begin{pmatrix}
0 & 1\\
1 & 0
\end{pmatrix}.
$$
For every $(x,v)\in\w{X}$, we then have
\begin{equation}\label{R_v_equality}
    \mathbf{R}_v\big[f(x)\mathbf{I}_1\oplus f(\eta(x))\mathbf{I}_1\big]\mathbf{R}_v^\ast=f(\eta_0(x))\mathbf{I}_1\oplus f(\eta_1(x))\mathbf{I}_1.
\end{equation}
This can be verified directly in the same manner as in \cite[Lemma~18.1.2]{loring}. 

Now, define $\mathbf{u}(\mathbf{x},v)=\mathbf{R}_v$ for any $(\mathbf{x},v)\in \w{X}$. Then, for every $f\in C(X)$ and $(\mathbf{x},v)\in\w{X}$, formula \eqref{R_v_equality} implies that
\begin{equation*}
    \begin{split}
        \big\{\mathbf{u} (\mu\otimes\id)\circ\Phi(f)\mathbf{u}^\ast\big\}(\mathbf{x},v) &=\mathbf{R}_v (\mu\otimes\id)\big[f(x)\mathbf{I}_2\oplus f(\eta(x))\mathbf{I}_2\big]\mathbf{R}_v^\ast\\
        &=(\mu\otimes\id)\big(f(\eta_0(x))\mathbf{I}_1\oplus f(\eta_1(x))\mathbf{I}_1\big)\\
        &=\big\{(\mu\otimes\id)\circ\Psi(f)\big\}(\mathbf{x},v),
    \end{split}
\end{equation*}
which shows that formula \eqref{ad_u} holds true, and finishes the proof of our~claim.

\vspace*{2mm}
Next, as in \cite[Lemma~18.1.1]{loring}, we extend both \vspace*{-1pt}summands $f\mapsto \lambda(\Delta_jf)$ ($j=0,1$) of $\p$ to $^\ast$-homomorphisms defined on $C(\w{X})$ and taking values in $\QQ(\Hh)$. More precisely, pick any functions $f_0,f_1\in C(X)$ such that $f_i(x)>0$ if and only if $x\in V_i=X\setminus X_i$. As for each $j=0,1$, $\lambda(\Delta_jf_0)$ and $\lambda(\Delta_jf_1)$ are positive and orthogonal elements of $\QQ(\Hh)$, we can apply Kasparov's Theorem~\ref{KTT_T} to the separable set $D=\mathrm{rg}\,(\lambda\circ\Delta_0)\cup\mathrm{rg}\,(\lambda\circ\Delta_1)$. Thus, we obtain $\mathfrak{z}_0,\mathfrak{z}_1\in\QQ(\Hh)\cap D^\prime$ with $0\leq\mathfrak{z}_0,\mathfrak{z}_1\leq 1$ such that
\begin{equation}\label{K_system}
\left\{\begin{array}{lcl}
    \mathfrak{z}_j\lambda(\Delta_j f_0) & = & 0\\
    \mathfrak{z}_j\lambda(\Delta_j f_1) & = & \lambda(\Delta_j f_1)
\end{array}\right.\quad\,\,\,\, (j=0,1).
\end{equation}
Then, for $j=0,1$, we consider the $^\ast$-homomorphism
$$
\Lambda_j\colon C(X\times I)\xrightarrow[\phantom{xx}]{}\QQ(\Hh),\quad \Lambda_j(f\otimes g)=\lambda(\Delta_jf)g(\mathfrak{z}_j),
$$
with the latter factor given by functional calculus. Equations \eqref{K_system} imply that $\Lambda_j$ vanishes on the ideal $C_0((V_0\times (0,1])\cup (V_1\times [0,1)))$ (see the proof of \cite[Lemma~18.1.1]{loring}), from which it follows that $\Lambda_j$ defines an~extension of $\lambda\circ\Delta_j$ to $C(\w{X})$. For simplicity, \vspace*{-1pt}we denote this extension by $\Lambda_j$ keeping in mind that $\Lambda_j(f-h)=0$ whenever $f$ and $g$ agree on $\w{X}$. Therefore, 
$$
\oo{\p}\colon C(\w{X})\xrightarrow[\phantom{xx}]{}\mathbb{M}_2(\QQ(\Hh)),\quad \oo{\p}=\Lambda_0\oplus\Lambda_1
$$
is an extension of $\p$.

Now, define $\mathbf{w}\in\mathbb{M}_4(\QQ(\Hh))$ by $\mathbf{w}=(\oo{\p}\otimes\id)(\mathbf{u})$, where $\mathbf{u}\in\mathbb{U}_2(C(\w{X}))$ is produced by Claim~1. Here, we have $\id\colon\mathbb{M}_2(\C)\to\mathbb{M}_2(\C)$. Then, using \eqref{ad_u}, we obtain
\begin{equation*}
    \begin{split}
        \mathrm{ad}_{\mathbf{w}}\big[\p\oplus (\p\circ \wh\eta)\big] &=\mathrm{ad}_{\mathbf{w}}\big[(\p\otimes \id)\circ\Phi\big]\\
        &=\mathrm{ad}_{\mathbf{w}}\big[(\oo{\phi}\otimes\id)\circ (\mu\otimes\id)\circ\Phi\big]\\
        &=(\oo{\phi}\otimes\id)\circ\mathrm{ad}_{\mathbf{u}}\big[(\mu\otimes\id)\circ\Phi\big]\\
        &=(\oo{\phi}\otimes\id)\circ (\mu\otimes\id)\circ\Psi\\
        &=(\p\otimes\id)\circ\Psi\\
        &=(\p\circ\wh\eta_0)\oplus (\p\circ\wh\eta_1).
    \end{split}
\end{equation*}
Hence, 
\begin{equation}\label{adw_star}
    \p\oplus(\p\circ\wh\eta)=\mathrm{ad}_{\mathbf{w}^\ast}\big[(\p\circ\wh\eta_0)\oplus (\p\circ\wh\eta_1)\big].
\end{equation}
On the other hand, recall that
\begin{equation*}
%\begin{split}
\mathbf{u}(\mathbf{x},v)=\mathbf{U}^\ast \begin{pmatrix}
1 & 0\\
0 & e^{\ii\pi v}
\end{pmatrix}\mathbf{U}=\frac{1}{2}\begin{pmatrix}
1 & 1\\
1 & 1
\end{pmatrix}+\frac{1}{2}e^{\ii\pi v}\begin{pmatrix}
1 & -1\\
-1 & 1
\end{pmatrix},
%\end{split}
\end{equation*}
which means that
$$
\mathbf{u}=\frac{1}{2}\,\mathbf{1}\otimes \begin{pmatrix}
1 & 1\\
1 & 1
\end{pmatrix}+\frac{1}{2}\,e^{\ii\pi v}\otimes\begin{pmatrix}
1 & -1\\
-1 & 1
\end{pmatrix},
$$
where $\mathbf{1}\in C(\w{X})$ is the constant $1$ function. Notice that $\oo{\p}(\mathbf{1})=1_{\QQ(\Hh)}\oplus 1_{\QQ(\Hh)}\in\mathbb{M}_2(\QQ(\Hh))$, whereas
$$
\oo{\p}(e^{\ii\pi v})=(\Lambda_0\oplus\Lambda_1) (\mathbf{1}\otimes e^{\ii\pi v})=\begin{pmatrix}
\exp(\ii\pi \mathfrak{z}_0) & 0\\
0 & \exp(\ii\pi\mathfrak{z}_1)
\end{pmatrix}\in\mathbb{M}_2(\QQ(\Hh)).
$$
Therefore,
\begin{equation*}
     \begin{split}
        \mathbf{w} &=\frac{1}{2}(\oo{\p}\otimes\id)\mathbf{1}\otimes\begin{pmatrix} 1 & 1\\ 1 & 1\end{pmatrix}+\frac{1}{2}(\oo{\p}\otimes\id)e^{\ii\pi v}\otimes \begin{pmatrix} 1 & -1\\ -1 & 1\end{pmatrix}\\[5pt]
        &=\frac{1}{2}\begin{pmatrix}
        1+\exp(\ii\pi\mathfrak{z}_0) & 0 & 1-\exp(\ii\pi\mathfrak{z}_0) & 0\\
        0 & 1+\exp(\ii\pi\mathfrak{z}_1) & 0 & 1-\exp(\ii\pi\mathfrak{z}_1)\\
        1-\exp(\ii\pi\mathfrak{z}_0) & 0 & 1+\exp(\ii\pi\mathfrak{z}_0) & 0\\
        0 & 1-\exp(\ii\pi\mathfrak{z}_1) & 0 & 1+\exp(\ii\pi\mathfrak{z}_1)
        \end{pmatrix}.
     \end{split}
\end{equation*}
Calculating the first $2\times 2$ direct summand of the right-hand side of \eqref{adw_star}, we obtain
\begin{equation*}
\begin{split}
\p=\frac{1}{2}\big[(1 &+\cos(\pi\mathfrak{z}_0))\,\lambda\circ\Delta_0\circ \wh\eta_0+(1-\cos(\pi\mathfrak{z}_0))\,\lambda\circ\Delta_0\circ\wh\eta_1\big]\\
&\oplus \frac{1}{2}\big[(1+\cos(\pi\mathfrak{z}_1))\,\lambda\circ\Delta_1\circ \wh\eta_0+(1-\cos(\pi\mathfrak{z}_1))\,\lambda\circ\Delta_1\circ\wh\eta_1\big].
\end{split}
\end{equation*}
By virtue of this formula, we may extend $\p$ to a~$^\ast$-monomorphism $\wh\p\colon C(X_0\sqcup X_1)\to\QQ(\Hh)$. Indeed, for any $\wh{f}=\wh{f}_0\sqcup\wh{f}_1\in C(X_0\sqcup X_1)$, we define
\begin{equation*}
\begin{split}
\wh\p(\wh{f})=\frac{1}{2}\big[(1 &+\cos(\pi\mathfrak{z}_0))\,\lambda\circ\Delta_0({\wh{f}_0}^{\,\prime}\circ\eta_0)+(1-\cos(\pi\mathfrak{z}_0))\,\lambda\circ\Delta_0(\wh{f}_1^{\,\prime}\circ\eta_1)\big]\\
&\oplus \frac{1}{2}\big[(1+\cos(\pi\mathfrak{z}_1))\,\lambda\circ\Delta_1(\wh{f}_0^{\,\prime}\circ\eta_0)+(1-\cos(\pi\mathfrak{z}_1))\,\lambda\circ\Delta_1(\wh{f}_1^{\,\prime}\circ\eta_1)\big],
\end{split}
\end{equation*}
where $\wh{f}_j^{\,\prime}$ is any continuous extension of $\wh{f}_j$ to $X$. The map $\wh{f}$, being defined for every continuous function on the disjoint union of $X_0$ and $X_1$, is in particular defined for every $f\in C(S\Omega_{n+1})$. This follows from the observation that tearing $S\Omega_{n+1}^{\usim}$ along the `cross' $Z$ we get rid of any pairs of antipodal points.

Denote by $\tau_0, \tau_1\colon C(S\Omega_{n+1})\to\QQ(\Hh)$ the above obtained extensions of $\lambda\circ\Delta_0$ and $\lambda\circ\Delta_1$, respectively. As we have observed earlier, $\tau_j(g\circ Sp_n)=\lambda(g)$ for $j=0,1$ and for every $g\in C(S\Omega_n)$. Hence, taking the zero extension $[\sigma]$ of $S\Omega_{n+1}$ we obtain $[\tau_0\oplus\sigma]\in\Ext(S\Omega_{n+1})$ satisfying $(Sp_n)_\ast([\tau_0\oplus\sigma])=[\lambda]$, which completes the proof.
\end{proof}

%%%%%%%%%%%%%%%%%%%%%%%%%%%%%%%%%%%%%%%%%%%%%%%%%%%%%%%%%%%%%%%%%%%%%%%%%%%%%%%%%%
%%%%%%%%%%%%%%%%%%%%%%%%%%%%%%%%%%%%%%%%%%%%%%%%%%%%%%%%%%%%%%%%%%%%%%%%%%%%%%%%%%
\section{Applications}
Given a $C_0$-semigroup $(q(t))_{t\geq 0}\subset\QQ(\Hh)$, Proposition~\ref{P_spectrum} produces an~extension $\Gamma\in\Ext(\Delta)$, where $\Delta=\varprojlim\Omega_n$ is described exclusively in terms of the spectrum $\sigma(A)$ of the generator of $(q(t))_{t\geq 0}$. Moreover, under natural `BDF-type' conditions, Proposition~\ref{P_kernel} says that $\Gamma\in\mathrm{ker}\,P$, where $P$ is the induced map onto the inverse limit of the groups $\Ext(\Omega_n)$ ($n=0,1,2,\ldots$). Hence, in order to guarantee that $\Gamma=\Theta$, it is enough to show that the first derived functor in Milnor's exact sequence \eqref{milnor_for_P} vanishes. If this is indeed the case, then by Lemma~\ref{lifting_L}, there exists a~lift of $(q(t))_{t\in\mathbb{D}}$ to a~dyadic semigroup in $\BB(\Hh)$. By virtue of Theorem~\ref{no_isolated_T}, the obtained lifting is \SOT-continuous, provided that the spectrum $\Delta$ has no isolated points.

Recall that, by definition, the functor $\varprojlim{}^{(1)}$ applied to an~inverse system of groups $\{G_n,p_n\}_{n=0}^\infty$ returns the cokernel of the map
$$
\prod_{n=0}^\infty G_n\ni (a_0,a_1,\ldots)\xmapsto[]{\,\,\,\,\mathsf{S}\,\,\,\,}(a_0-p_0(a_1),a_1-p_1(a_2),\ldots)
$$
defined on the full direct product of all the $G_n$. That is, 
$$
\varprojlim{}^{(1)}G_n=\prod_{n=0}^\infty G_n/\mathsf{S}\Big(\prod_{n=0}^\infty G_n\Big)
$$
which measures how big is the loss of information in passing from $\{G_n\}$ to the inverse limit $\varprojlim G_n$. Clearly, the first derived functor is trivial whenever all the connecting maps are surjective. There is also a~more subtle condition guaranteeing that $\varprojlim{}^{(1)}G_n=0$. \vspace*{-1pt}Namely, for $0\leq n<k$ define $p_n^k\colon G_k\to G_n$ as $p_n^k=p_n\circ p_{n+1}\circ\ldots\circ p_{k-1}$, so that for all $n<k<\ell$, we have $p_n^\ell=p_n^k\circ p_k^\ell$. We say that the \vspace*{1pt}inverse system $\{G_n,p_n^k\}$ is {\it semistable} or satisfies the {\it Mittag-Leffler condition}, provided that \vspace*{-1pt}for each $n\in\N_0$ there is $\nu(n)\geq n$ such that for each $k\geq\nu(n)$, the range of $p_n^k$ is the same as the range of $p_n^{\nu(n)}$. By virtue of \cite[Thm.~11.3.2]{geoghegan}, if $\{G_n,p_n^k\}$ is semistable, then $\varprojlim{}^{(1)}G_n=0$.
 
\vspace*{1mm}
The results of Section~5 possibly provide some tools to verify that the~connecting maps corresponding to $\varprojlim{}^{(1)}\Ext(S\Omega_n)$ satisfy the Mittag-Leffler conditions, even though they are not all surjective. At this point, however, let us note that under the conditions of Theorems~\ref{empty_dir_T} and \ref{main_K_T}, assumed for all $n=0,1,2,\ldots$, every connecting map $\Ext_2(\Omega_{n+1})\to\Ext_2(\Omega_n)$ is surjective. Hence, assuming either the conjunction of \eqref{sep_assumption} and the `empty direction' condition, or the conjunction of \eqref{sep_assumption} and the `cross retract' condition, we obtain that the extension $\Gamma\in\Ext(\Delta)$ is the trivial one.

\begin{corollary}\label{empty_cross_C}
Let $(Q(t))_{t\geq 0}$ be a~collection of normal operators in $\BB(\Hh)$ satisfying \eqref{mod_K}. Assume that $(q(t))_{t\geq 0}\subset\QQ(\Hh)$, defined by $q(t)=\pi Q(t)$ for $t\geq 0$, is a~$C_0$-semigroup with respect to some faithful $^\ast$-representation $\gamma\colon \QQ(\Hh)\to\BB(\mathbb{H})$. Suppose that the spectrum of the generator $A$ satisfies one of the~following conditions:

\vspace*{1mm}
\begin{enumerate}[label={\rm (\roman*)}, leftmargin=26pt]
\setlength{\itemsep}{3pt}
\item $\sup\{\abs{\mathrm{Im}\,z}\colon z\in\sigma(A)\}<+\infty$;
\item the set $\mathrm{Re}\,\sigma(A)$ is finite and for any large enough $n\in\N$, every circle section $\Omega_n\cap\mathbb{T}_r$ is a~symmetric set whose each connected component has length smaller than $\pi r$.
\end{enumerate}
Then, there exists a~semigroup $(T(t))_{t\in\mathbb{D}}\subset\BB(\Hh)$ such that $\pi T(t)=q(t)$ for every $t\in\mathbb{D}$.
\end{corollary}

\begin{proof}
First, notice that if the homomorphisms $(Sp_n)_\ast\colon \Ext(S\Omega_{n+1})\to\Ext(S\Omega_n)$ are surjective for all sufficiently large $n\in\N$, then the inverse system $\{\Ext(S\Omega_n),(Sp_n)_\ast\}_{n=0}^\infty$ satisfies the Mittag-Leffler condition, hence $\varprojlim{}^{(1)}\Ext(S\Omega_n)=0$. 

Assuming (i) we see that for large enough $n\in\N$, we have $\abs{\mathrm{Arg}(z)}<\tfrac{\pi}{2}$ for every $z\in\Omega_{n}$ which implies condition \eqref{sep_assumption} as for any such $n$ we simply have $\mathsf{A}(\Omega_{n})=\varnothing$. Also, the section $\mathcal{S}_\alpha$ defined by \eqref{section_deff} is empty e.g. for $\alpha=\pi$. Therefore, our result follows by combining Theorem~\ref{empty_dir_T} with Corollary~\ref{split_C}.

Under assumption (ii) we have $\mathsf{A}(\Omega_n)=\Omega_n$ for sufficiently large $n\in\N$, hence condition \eqref{sep_assumption} is again satisfied. Moreover, for any such $n$ an arbitrary section $\mathsf{S}_\beta=\R e^{\ii\beta}\cap\Omega_{n+1}$ is a~retract of both the corresponding left and the~right part of $\Omega_{n+1}$. Indeed, we have only finitely many nonempty circle sections $\Omega_{n+1}\cap\mathbb{T}_r$, none of which can contain a~whole arc joining $\pm re^{\ii\beta}$. Hence, on both these two arcs we can find open arcs, say $U$ and $V$ lying in the left and the right part of $\Omega_{n+1}$, respectively, such that $(U\cup V)\cap \Omega_{n+1}=\varnothing$. Then, a~retraction $\rho$ from the left part onto $\mathsf{S}_\beta$ can be defined on the section $\Omega_{n+1}\cap\mathbb{T}_r$ by $\rho(w)=re^{\ii\beta}$ for $w\in\Omega_{n+1}$ lying on an~arc starting at $re^{\ii\beta}$ and disjoint from $U$, and $\rho(z)=-re^{\ii\beta}$ otherwise. Similarly we define a~retraction from the right part of $\Omega_{n+1}$. Therefore, by Theorem~\ref{main_K_T}, for every sufficiently large $n\in\N$, the homomorphism $(Sp_n)_\ast\colon\Ext(S\Omega_{n+1})\to\Ext(S\Omega_n)$ is surjective. As we have already explained, this fact in combination with Corollary~\ref{split_C} yields the result.
\end{proof}

It is not difficult to produce many examples of spectrum for which the~assumption (ii) above is satisfied. For instance, for any finite sets $\{s_1,\ldots,s_m\}\subset\R$ and $\{\alpha_1,\ldots,\alpha_m\}\subset [0,\tfrac{\pi}{2}]$ one can take
$$
\sigma(A)=\bigcup_{j=1}^m\Big\{s_j+\ii t\colon t\in\bigcup_{k\in\Z}\big([\alpha_j,\alpha_j+\tfrac{\pi}{2}]\cup[\alpha_j+\pi,\alpha_j+\tfrac{3\pi}{2}]\big)+2k\pi\Big\}.
$$

Arguing similarly as in Examples~\ref{ex_2k} and \ref{ex_imaginary}, we obtain another class of examples of $\sigma(A)$ for which there exists a~semigroup lift, even though neither the `empty direction' nor the `cross retract' condition is satisfied.

\begin{corollary}\label{milnorr_C}
Under the above assumptions on $(Q(t))_{t\geq 0}$ and $(q(t))_{t\geq 0}$, if
\vspace*{0mm}
\begin{enumerate}[label={\rm (\roman*)}, leftmargin=26pt]
\setlength{\itemsep}{3pt}
\item[{\rm (iii)}] the set $\mathrm{Re}\,\sigma(A)$ is either finite or an~interval, and for every $s\in\mathrm{Re}\,\sigma(A)$, the section $\{t\in\R\colon s+\ii t\in\sigma(A)\}$ is a~half-line,
\end{enumerate}
then there exists a~semigroup $(T(t))_{t\in\mathbb{D}}\subset\BB(\Hh)$ such that $\pi T(t)=q(t)$ for every $t\in\mathbb{D}$.
\end{corollary}
\begin{proof}
First, assume that $\mathrm{Re}\,\sigma(A)=\{s_1,\ldots,s_m\}$. Then each $\Omega_n$ is a~disjoint union of $m$ circles and according to \cite[Lemma~2.12]{BDF} we have $\Ext(\Omega_n)=\Z\oplus\ldots\oplus\Z$ ($m$-fold product). In this case, Milnor's exact sequence \eqref{milnor_for_P} has the form 
$$
\begin{tikzcd}
0 \arrow[r] & \varprojlim{}^{(1)}\Ext(S^2\sqcup\ldots\sqcup S^2) \arrow[r] & \Ext(X) \arrow[r] & \varprojlim \Z\oplus\ldots\oplus\Z \arrow[r] & 0
\end{tikzcd}
$$
and since $\Ext(S^2)=0$ and $\varprojlim \Z\oplus\ldots\oplus\Z=0$, we obtain $\Ext(X)=0$.

Assuming $\mathrm{Re}\,\sigma(A)$ is an~interval, we see that each $\Omega_n$ is an~annulus, thus homotopy equivalent to the circle $S^1$. Similarly, $S\Omega_n$ is homotopy equivalent to $S^2$. Hence, in view of \cite[Thm.~2.14]{BDF}, the corresponding extension groups are the same as the ones for $S^1$ and $S^2$ and the resulting Milnor's exact sequence is the same as the one in Example~\ref{ex_imaginary}. Therefore, $\Ext(X)=\Ext(\Sigma_2)=0$ and appealing to Corollary~\ref{split_C} concludes the proof.
\end{proof}

\vspace*{2mm}\noindent
{\bf Acknowledgement. }I acknowledge with gratitude the support from the National Science Centre, grant OPUS 19, project no.~2020/37/B/ST1/01052.

\bibliographystyle{amsplain}

\begin{thebibliography}{10}
\bibitem{anderson} J. Anderson, \emph{Extreme points in sets of positive linear maps on $\mathscr{B}(\mathscr{H})$}, J.~Funct. Anal.~{\bf 31} (1979), 195--217.

\bibitem{AB} J. Anderson, J. Bunce, \emph{A type II$_\infty$ factor representation of the Calkin algebra}, Amer. J.~Math.~{\bf 99} (1977), 515--521.

%\bibitem{atiyah} M.F. Atiyah, \emph{$K$-theory}, Benjamin, New York 1967.

\bibitem{berg} I.D. Berg, \emph{An extension of the Weyl-von Neumann Theorem to normal operators}, Trans. Amer. Math. Soc.~{\bf 160} (1971), 365--371.

\bibitem{blackadar-K} B. Blackadar, \emph{$K$-theory for operator algebras} (second edition), Mathematical Sciences Research Institute Publications 5, Cambridge University Press, Cambridge 1998.

\bibitem{blackadar} B. Blackadar, \emph{Operator algebras. Theory of \cs-algebras and von~Neumann algebras}, Encyclopaedia of Mathematical Sciences, vol.~122, Springer-Verlag, Berlin Heidelberg 2006.

%\bibitem{brown} L.G. Brown, \emph{The topology of the group $\mathrm{Ext}(X)$}, unpublished.

\bibitem{BDFu} L.G. Brown, R.G. Douglas, P.A. Fillmore, \emph{Unitary equivalence modulo the compact operators and extensions of \cs-algebras}, in: Proc. Conf. on Operator Theory, pp. 58--128, Lecture Notes in Mathematics 345, Springer-Verlag, Heidelberg 1973.

\bibitem{BDF} L.G. Brown, R.G. Douglas, P.A. Fillmore, \emph{Extensions of \cs-algebras and $K$-homology}, Ann. Math.~{\bf 105} (1977), 265--324.

\bibitem{calkin} J.W. Calkin, \emph{Two-sided ideals and congruences in the ring of bounded operators in Hilbert space}, Ann. Math.~{\bf 42} (1941), 839--873. 

\bibitem{davidson} K.R. Davidson, \emph{$\mathrm{C}^\ast$-algebras by example}, Fields Institute Monographs, Amer. Math. Soc., Providence, RI 1996.

%\bibitem{davies} E.B. Davies, \emph{On the Borel structure of \cs-algebras}, Commun. Math. Phys.~{\bf 8} (1968), 147--163.

\bibitem{Dou} R.G. Douglas, \emph{\cs-algebra extensions and $K$-homology}, Annals of Mathematical Studies 95, Princeton University Press, Princeton, N.J. 1980.

%\bibitem{douglas} R.G. Douglas, \emph{Banach algebra techniques in operator theory}, 2nd edition, Graduate Texts in Mathematics~179, Springer-Verlag, New York 1998.

\bibitem{DS} N. Dunford, J.T. Schwartz, \emph{Linear operators}, Part~ I: \emph{General theory}, A~Wiley-Interscience Publication, John Wiley \& Sons, Inc., New York 1988.

\bibitem{ES} S. Eilenberg, N. Steenrod, \emph{Foundations of algebraic topology}, Princeton University Press, Princeton 1952.

\bibitem{eisner} T. Eisner, \emph{Embedding operators into strongly continuous semigroups}, Arch. Math. (Basel) {\bf 92} (2009), 451--460.

\bibitem{EN} K.-J. Engel, R. Nagel, \emph{A~short course on operator semigroups}, Springer, New York 2006.

\bibitem{FHHMZ}    M. Fabian, P. Habala, P. H\'{a}jek, V.~Montesinos, V.~Zizler, \emph{Banach Space Theory. The Basis for Linear and Nonlinear Analysis}, Springer, New York 2011. 

\bibitem{farah} I. Farah, \emph{Combinatorial set theory of \cs-algebras}, Springer Monographs in Mathematics, Springer 2019.

\bibitem{geoghegan} R. Geoghegan, \emph{Topological methods in group theory}, Graduate Texts in Mathematics~243, Springer, New York 2008.

\bibitem{JP} B.E. Johnson, S.K. Parrott, \emph{Operators commuting with a~von Neumann algebra modulo the set of compact operators}, J. Funct. Anal.~{\bf 11} (1972), 39--61.

\bibitem{KS} J. Kaminker, C.~Schochet, \emph{$K$-theory and Steenrod homology: applications to the Brown-Douglas-Fillmore theory of operator algebras}, Trans. Amer. Math. Soc.~{\bf 227} (1977), 63--107.

\bibitem{loring} T.A. Loring, \emph{Lifting solutions to perturbing problems in \cs-algebras}, Fields Institute Monographs, vol.~8, Providence, R.I., Amer. Math. Soc.~1997.

\bibitem{milnor} J. Milnor, \emph{On the Steenrod homology theory} (first distributed 1961), in: S. Ferry, A.~Ranicki, J.~Rosenberg (Eds.), \emph{Novikov Conjectures, Index Theorems, and Rigidity}: Oberwolfach 1993, London Mathematical Society Lecture Note Series, pp. 79--96, Cambridge University Press, Cambridge 1995. 

\bibitem{OP} C.L. Olsen, G.K. Pedersen, \emph{Corona \cs-algebras and their applications to lifting problems}, Math. Scand.~{\bf 64} (1989), 63--86.

\bibitem{reid} G.A. Reid, \emph{On the Calkin representations}, Proc. London Math. Soc.~{\bf 23} (1971), 547--564.

\bibitem{rickart} C.E. Rickart, \emph{General theory of Banach algebras}, van Nostrand, Princeton 1960.

\bibitem{rudin} W. Rudin, \emph{Real and complex analysis}, 3rd edition, McGraw-Hill Book Co., New York 1987.

\bibitem{SS} S. Shelah, J. Stepr\={a}ns, \emph{Masas in the Calkin algebra without the continuum hypothesis}, J.~Appl. Anal.~{\bf 17} (2011), 69--89.

\bibitem{sikonia} W. Sikonia, \emph{The von Neumann converse of Weyl's theorem}, Indiana Univ. Math. J.~{\bf 21} (1971/72), 121--124.
\end{thebibliography}

\end{document}